\documentclass[11pt]{amsproc}
\pdfoutput=1
\usepackage{amsmath,amssymb,amsthm,eucal, eufrak,amscd,tikz,tikz-cd,mathtools}
\usepackage{xcolor}
\usepackage{enumitem}
\usepackage[shortalphabetic]{amsrefs}
\usepackage{xypic}
\usepackage{graphicx}
\definecolor{allrefcolors}{rgb}{0,0.2,0.5}
\usepackage[linktocpage=true,colorlinks=true,allcolors=allrefcolors,bookmarksopen,bookmarksdepth=3]{hyperref}
\usepackage[margin=1.25in]{geometry}

\def\Z{{\mathbb Z}}
\def\R{{\mathbb R}}
\def\C{{\mathbb C}}

\def\K{{\mathbf{k}}}
\def\w{\mathcal{W}(E)}

\def\B{\mathcal{B}}
\def\m{\mathcal{M}}
\def\n{\mathcal{N}}

\def\T{\mathcal{T}}
\def\e{\epsilon}
\def\cc{\mathcal{C}}
\def\dd{\mathcal{D}}
\def\mod{\mathrm{\!-\! mod}}
\def\rmod{\mathrm{mod\!-\!}}
\def\bimod{\mathrm{\!-\! mod\!-\!}}

\def\e{\epsilon}

\def\b{\beta}
\def\r#1{\mathrm{#1}}

\def\mc#1{\mathcal{#1}}

\def\w{\mathcal{W}}
\def\rw{\mathcal{RW}}

\def\d{\delta}
\def\D{\Delta}

\def\ainf{A_\infty}

\def\z2{\Z / 2\Z}
\def\y{\mc{Y}}

\def\id{\mathrm{id}}

\def\p{\partial}

\def\ob{\mathrm{ob\ }}
\def\ch{\mathrm{Ch}}
\def\perf{\mathrm{Perf}}
\def\calk{\mathrm{Calk}}
\def\fun{\mathrm{Fun}}
\def\cone{\mathrm{cone}}
\def\binf{\widehat{\B}_{\infty}}

\def\n{\mathcal{N}}

\def\fk{\mathfrak{f}}
\def\gk{\mathfrak{g}}
\def\ck{\mathfrak{c}}
\def\wcm{\widehat{\w}_{\mathfrak{f}}}
\def\wf{\widehat{\w}_{F}}
\def\ct{\widehat{\cc}_{\mathcal{T}}}
\def\hocolim{\mathrm{hocolim}}
\def\holim{\mathrm{holim}}
\def\re{\mathrm{Re}}
\def\im{\mathrm{Im}}
\def\nbd{\mathcal{N}\mathrm{bd}}

\newtheorem{lem}{Lemma}[section]
\newtheorem{prop}[lem]{Proposition}
\newtheorem{thm}[lem]{Theorem}
\newtheorem{cor}[lem]{Corollary}
\newtheorem{conj}[lem]{Conjecture}
\newtheorem{defn}[lem]{Definition}

\newtheorem{rem}[lem]{Remark}
\def\e{\epsilon}

\theoremstyle{remark}

\newtheorem{example}{Example}[section]
\numberwithin{equation}{section}
\begin{document}
\begin{abstract}
For a stopped Liouville manifold arising from a Liouville sector, 
we construct a symplectic analogue of the formal neighborhood of the stop on the level of Fukaya categories.
This geometric construction is performed via Floer-theoretic methods by allowing wrappings in the negative direction.
On the other hand, inspired by homological mirror symmetry for pairs, 
where the mirror is the formal neighborhood of a divisor in an ambient projective variety,
there is a different approach by taking a `categorical formal completion' introduced by Efimov.
Our main results establishes equivalence of these two approaches,
confirms computability of this new type of Floer theory by categorical and algebraic means,
and indicates contributions from and to computations in homological mirror symmetry.
\end{abstract}

\title{Fukaya $\ainf$-structure near infinity and the categorical formal completion}
\author{Yuan Gao}

\maketitle

\section{Introduction}

The goal of this paper is to construct a Fukaya $\ainf$-categorical structure `at infinity' for a {\it stopped Liouville manifold} $(\bar{X}, F_{0})$ \cite{GPS1}, \cite{GPS2}, \cite{sylvan1}
which arises from a corresponding {\it Liouville sector} $X$.
This structure should be interpreted as a type of Fukaya categorical structure on the {\it `formal neighborhood'} of the Liouville hypersurface $F_{0} \subset \p_{\infty} \bar{X}$ at infinity (where $F$ is the Liouville completion of the domain $F_{0}$), 
thus providing a symplectic analogue of the notion of the formal neighborhood of a divisor in an ambient variety in algebraic geometry.
The starting point of this study is wrapped Floer theory for Liouville sectors,
which is introduced and extensively studied in \cite{GPS1},
and earlier in \cite{sylvan1} for a certain class of stopped Liouville manifolds, or equivalently Liouville pairs.
In \cite{GPS2}, the theory is further extended to all {\it stopped Liouville sectors},
defining the partially wrapped Fukaya category $\w(X, \gk)$ for an arbitrary stopped Liouville sector $(X, \gk)$, where $\gk \subset \p_{\infty} X$ is a closed subset,
with objects being exact cylindrical Lagrangian submanifolds of $\p X$ disjoint from $\gk$ at infinity,
and morphisms being Floer cochains after isotoping Lagrangians such that their boundaries at infinity are positively isotoped in $\p_{\infty}X \setminus \gk$
(such isotopies are called {\it positive wrappings}).

For our study, we shall be mainly interested in the case of a stopped Liouville manifold.
The goal of this paper is to construct a new Fukaya-type $\ainf$-category using a variant of Floer theory,
which takes into account both {\it positive wrappings} not passing the stop,
as well as arbitrary {\it negative wrappings} which necessarily pass through the stop.
This type of Floer theory is to be defined using global Floer data on the target manifold, 
but will capture local information around the stop.
A rough conceptual picture of this construction can be described as follows:
wrapping passing through the stop defines a self-functor of the partially wrapped Fukaya category, 
and we can consider the nilpotents under this functor,
which form an inverse system of categories by the orders of nilpotency,
where the connecting functors are given by applying `wrapping past the stop'.
Eventually, we would hope to obtain the desired category by taking the limit of this inverse system.
However, this naive description is lack of sufficient ingredients for being a rigorous construction, 
because of the following two reasons: first, the limit might not exist or not have nice properties, even if it exists;
second, for an arbitrary stop, it is not always possible to characterize the process of `wrapping past the stop' in a precise quantitative way making the self-functor well-defined.

However, when the geometry of $(\bar{X}, \fk)$ is reasonably nice, 
there is a nice and clean way of implementing the above idea to develop such a Floer theory, 
to be spelt out in \S\ref{section: negatively wrapped Fukaya category}.
This type of Floer theory is well-behaved when the stop $\fk$ admits a {\it ribbon};
that is to say, $\fk$ is itself or the core of a Liouville hypersurface (with boundary) $F_{0} \subset \p_{\infty} \bar{X}$.
Stopped Liouville manifolds of this kind naturally come from  {\it sutured Liouville manifolds} $(\bar{X}, F_{0})$ or {\it Liouville pairs} $(\bar{X}, F)$ where $F$ is the Liouville completion of $F_{0}$,
and each of them corresponds to a unique (up to deformation) Liouville sector of the form $\bar{X} \setminus \nbd F_{0}$ and vice versa, as shown by \cite{GPS1}.
Their partially wrapped Fukaya categories are all quasi-equivalent by \cite{GPS2},
\[
\w(\bar{X} \setminus \nbd F_{0}) \cong \w(\bar{X}, F_{0})  \cong \w(\bar{X}, \fk).
\]
From now on, we will mostly take the stop to be $F_{0}$, 
but sometimes interchangeably use the core $\fk$ whenever the relevant results are easier to state.

To see a picture of what this new type of Floer theory does on objects, 
consider a pair of exact cylindrical Lagrangians $K, L$ whose boundaries at infinity $\p_{\infty}K, \p_{\infty}L$ are Legendrian submanifolds disjoint from $F_{0}$.
The morphism space for this Floer theory is a chain complex $NC^{*}_{F}(K, L)$ which can be thought of as generated by all the characteristics on the contact manifold $\p_{\infty}\bar{X}$ from Legendrians $\p_{\infty}K$ to $\p_{\infty}L$ not passing the `forbidden' $F_{0}$ at infinity,
as well as all the characteristics from $\p_{\infty}L$ to $\p_{\infty}K$ that are allowed to pass through $F$.
In \S\ref{section: negatively wrapped Fukaya category}, 
we equip these chain complexes with $\ainf$-structures, 
which are suitably packaged into an $\ainf$-category which we shall name as the {\it negatively wrapped Fukaya category}
\begin{equation}
\n(\bar{X}, F_{0}).
\end{equation}
Although the definition of this category relies on the existence of the ribbon $F$ not just a stop $\fk$, 
we also reserve the notation $\n(\bar{X}, \fk)$ for the same category, 
or simply $\n(X)$ where $X$ is the Liouville sector corresponding to the sutured Liouville manifold $(\bar{X}, F_{0})$, 
so that
\begin{equation}
\n(\bar{X}, F_{0}) = \n(\bar{X}, \fk) = \n(X).
\end{equation}

In the case of a Lefschetz fibration $\pi: \bar{X} \to \C$, this construction is essentially equivalent to one sketched in \cite{seidel6}.
A schematic geometric picture to interpret this point is as follows.
Without loss of generality, we may assume that the critical values are contained in the unit disk.
This category $\n(\bar{X}, F_{0})$ is constructed by perturbing Lagrangians by a Hamiltonian vector field in a way such that their projections to $\C$ are perturbed slightly in the counterclockwise direction in a large disk containing the unit disk,
where the projections are already linear in a collar neighborhood of the boundary of the disk,
and perturbed in arbitrary amounts in the clockwise direction outside that disk.

By construction, the category $\n(\bar{X}, F_{0})$ comes with a natural functor from the partially wrapped Fukaya category $\w(\bar{X}, F_{0})$,
but this functor does not shed much light on the structures of $\n(\bar{X}, F_{0})$.
Generally, some information is lost when passing from $\w(\bar{X}, F_{0})$ to $\n(\bar{X}, F_{0})$ via this functor.
However, one of the main goals of the paper is to show that $\n(\bar{X}, F_{0})$ is recoverable from $\w(\bar{X}, F_{0})$ in some nice situations.
A first necessary condition for the situation to be considered as `nice' is that there should be a good sense of objects in $\bar{X}$ that are {\it `supported on $F$'}.
A sufficient condition is that the core $\fk = \ck_{F}$ is {\it mostly Legendrian} in the sense of \cite{GPS2},
in which case the objects in $\bar{X}$ `supported on $F$' can be taken to be {\it Lagrangian linking disks}, 
bounded by small Legendrian linking spheres at smooth Legendrian points of $\fk$.
In general when $\fk$ is not necessarily mostly Legendrian, 
the subcategory of linking disks should be replaced by the full subcategory $\tilde{\dd}(F)$ of $\w(\bar{X}, F_{0})$ whose objects are those in the image of the cup functor
\begin{equation}\label{cup functor}
i_{F}: \w(F) \to \w(F \times T^{*}[0, 1]) \to \w(\bar{X}, F_{0}),
\end{equation}
which is defined to be the composition of the K\"{u}nneth stabilization functor and the pushforward functor induced by the canonical inclusion $T^{*}[0, 1] \xhookrightarrow{} (\bar{X}, F_{0})$.
In this case, we also need a {stop removal condition}.
We say the Liouville manifold $F$ {\it satisfies stop removal for the Liouville pair $(\bar{X}, F)$} (with the corresponding sutured Liouville manifold $(\bar{X}, F_{0})$), 
if the functor
\begin{equation}\label{stop removal functor}
\iota_{*}: \w(\bar{X}, F_{0})/\tilde{\dd}(F) \to \w(\bar{X})
\end{equation}
induced by pushforward $\w(\bar{X}, F_{0}) \to \w(\bar{X})$ is a quasi-equivalence.
We say the Liouville manifold $F$ {\it satisfies stop removal},
if it satisfies stop removal for $(\bar{X}, F)$ for any $\bar{X}$.
Changing from the assumption that $\fk$ is mostly Legendrian to the one that $F$ satisfies stop removal is a very mild generalization. 
In addition, a properness condition seems natural as well:
we say that the stop $F_{0}$ (or $\fk$) is a {\it full stop} if $\w(\bar{X}, F_{0}) = \w(\bar{X}, \fk)$ is proper.

The conditions stated above ensure good control over the behavior of Floer theory, 
allowing us to compute the new category $\n(\bar{X}, F_{0})$ from the partially wrapped Fukaya category $\w(\bar{X}, F_{0})$ using a certain purely algebraic machinery.
The relevant algebraic structure is that of the {\it categorical formal completion of a category along a subcategory}, which was introduced in \cite{efimov1} in the case of dg categories. 
The definition receives a straightforward generalization,
giving rise to the notion of the {\it formal completion} $\ct$ of an $\ainf$-category $\cc$ along a subcategory $\T$ of $\cc$ (or of $\cc \mod$, the category of left modules over $\cc$, or its derived category $\tilde{D}(\cc)$);
for the precise definition, see Definition \ref{defn: formal completion} in \S\ref{section: formal completion}.
Notably, the categorical formal completion is a categorical analogue of the more familiar algebro-geometric concept of the formal neighborhood $\hat{Y}_{D}$ of a closed subscheme $D$ in a given ambient (Noetherian, seperated) scheme $Y$.
To present a shortcut definition here, consider the case where $\T$ is a full subcategory of $\cc$.
Let 
\begin{equation}
Z_{\T} = im(y^{l} \circ i: \T^{op} \to \cc \mod)
\end{equation}
be the image of $\T^{op}$ under the Yoneda embedding $\T^{op} \xhookrightarrow{} \cc^{op} \to \cc \mod$.
Tautologically, this defines an $\ainf$-bimodule over $(Z_{\T}, \cc^{op})$, by
\begin{equation}
(\hat{i}, y^{l})^{*}(\cc \mod)_{\D},
\end{equation}
where $\hat{i}: Z_{\T} \to \cc \mod$ is the inclusion of the subcategory and $y^{l}: \cc^{op} \to \cc \mod$ is the Yoneda embedding.
Such a bimodule gives rise to a module-valued functor 
\begin{equation}
F_{\T}: \cc \to Z_{\T} \mod = \rmod Z_{\T}^{op}
\end{equation}
and we define $\ct$ to be the essential image of $F_{\T}$.
For $\T$ a subcategory of $\cc \mod$, define $Z_{\T} = \T$ and perform a similar construction.
The formal completion is independent of the subcategory up to split-generation (when embedded in $\cc \mod$), 
so is always quasi-equivalent to $\widehat{\cc}_{\perf \T}$ (Corollary \ref{cor: completion along perfect complexes}).
By construction, the formal completion comes with a natural functor $\cc \to \ct$,
such that if $\T$ is a full subcategory of $\cc$, 
the restriction of the functor to $\T$ gives a fully faithful embedding $\T \xhookrightarrow{} \ct$ (Lemma \ref{property of k functor}).
In this way, $\ct$ can be interpreted as a formal neighborhood of $\T$ in $\cc$.
See  \S \ref{section: algebra} for more details about the definitions and properties of $\ct$.

Having introduced the necessary ingredients introduced above, we can now state the first main result about the category $\n(\bar{X}, F_{0})$ as follows.

\begin{thm}\label{thm:main}
For any Liouville pair $(\bar{X}, F)$,
let $\tilde{\mathcal{D}}(F)$ be the subcategory of $\w(\bar{X}, F_{0})$ whose objects are images under the cup functor \eqref{cup functor},
and let 
\begin{equation}\label{formal completion of partially wrapped}
\wf := \widehat{\w(\bar{X}, F_{0})}_{\tilde{\dd}(F)}
\end{equation}
denote the categorical formal completion of $\w(\bar{X}, F_{0})$ along the subcategory $\tilde{\dd}(F)$.

\begin{enumerate}[label=(\roman*)]

\item There is a canonical $\ainf$-functor
\begin{equation}\label{functor psi}
\Psi: \n(\bar{X}, F_{0}) \to \wf.
\end{equation}

\item Suppose $\bar{X}$ and $F$ are both non-degenerate,
and $F$ satisfies stop removal and is a full stop.
Then the functor $\Psi$ \eqref{functor psi} is a quasi-equivalence.
\end{enumerate}

\end{thm}

It is straightforward to check the conditions are satisfied for a symplectic Landau-Ginzburg model, 
yielding the following consequence.

\begin{cor}
Suppose $\pi: \bar{X} \to \C$ is a symplectic Landau-Ginzburg model on a non-degenerate Liouville manifold $\bar{X}$ with compact critical locus and Weinstein fiber $F$.
Then there is a canonical quasi-equivalence
\[
\Psi: \n(\bar{X}, F_{0}) \stackrel{\sim}\to \wf.
\]
\end{cor}

In fact, the quasi-equivalence between $\n(\bar{X}, F_{0})$ and $\wf$ does not rely on the seemingly artificial non-degeneracy condition.
This is odd yet interesting because it suggests that the relation between these categories has its own intrinsic meaning independent of generation of wrapped Fukaya categories.
However, if we want to drop the non-degeneracy condition, we shall need a stronger version of properness assumption than $F$ being a full stop.

\begin{defn}\label{def: bounded Reeb dynamics}
We say the contact manifold $\p_{\infty} \bar{X}$ has bounded Reeb dynamics relative to a Liouville hypersurface $F$, if for any two Legendrians $\Lambda_{0}, \Lambda_{1}$, and any $m \in \Z_{\ge 0}$, 
the set of Reeb chords from $\Lambda_{0}$ to $\Lambda_{1}$ which intersect $F_{0}$ transversely at most $m$ times is finite.
\end{defn}

This condition can be thought of as an open-string version of a common condition on Reeb orbits.
A standard consequence of the above bounded Reeb dynamics assumption is that the wrapped Floer cohomology between any pair of Lagrangians is degree-wise finite dimensional.
For the category $\n(\bar{X}, F_{0})$, it implies the following:

\begin{thm}\label{thm:main2}
Suppose $F$ satisfies stop removal for the pair $(\bar{X}, F)$,
and that the contact boundary-at-infinity $\p_{\infty} \bar{X}$ has bounded Reeb dynamics relative to $F$, and that $\w(\bar{X}, F_{0})$ is (homologically) smooth.
Then there is a quasi-equivalence
\begin{equation}
\n(\bar{X}, F_{0}) \stackrel{\sim}\to \wf.
\end{equation}
\end{thm}

Some comments about Theorem \ref{thm:main2} are in order.
Note that Theorem \ref{thm:main2} does not directly establish \eqref{functor psi} to be a quasi-equivalence as in the statement of Theorem \ref{thm:main}.
However, there is indeed a canonical functor defined by other means, which under some further technical argument can be compared to the functor \eqref{functor psi}.
Regarding the hypotheses of these two theorems, the implication of Theorem \ref{thm:main2} is slightly stronger than Theorem \ref{thm:main} as we do not assume $\bar{X}$ is non-degenerate. 
Of course, when $F$ is non-degenerate, and if we further assume that $\w(\bar{X}, F_{0})$ admits a finite collection of split-generators,
then the assumption that $F$ satisfies stop removal will imply $\bar{X}$ is non-degenerate.
In \cite{sylvan1}, it is shown that under some assumptions, including Abouzaid's generation criterion \cite{abouzaid_gc}, such a non-degenerate $F$ satisfies stop removal. 
It is expected that any non-degenerate $F$ satisfies stop removal, but the answer to the converse remains unclear.
As for Theorem \ref{thm:main2}, the assumption that $F$ satisfies stop removal is sufficient for our conclusion to be drawn.
In fact, Theorem \ref{thm:main2} only requires $F$ to satisfy stop removal for the particular pair $(\bar{X}, F)$.
Lastly, smoothness of $\w(\bar{X}, F_{0})$ is also an essential condition. 
There are geometric conditions leading to smoothness of $\w(\bar{X}, F_{0})$: for example by \cite{GPS2} Corollary 1.19 if $(\bar{X}, F)$ is {\it weakly Weinstein} then $\w(\bar{X}, F_{0})$ is smooth (since we already assume that $\fk$ has a ribbon $F$).
However, that assumption is unnecessarily strong for Theorem \ref{thm:main2} to hold.

Theorem \ref{thm:main} is closely related to the main result of \cite{GGV},
which concerns about the {\it Rabinowitz Fukaya category} $\rw(\bar{X})$ of a Liouville manifold.
When the stop $F$ is empty, the category $\n(\bar{X}, \varnothing)$ recovers the Rabinowitz Fukaya category
\[
\n(\bar{X}, \varnothing) = \rw(\bar{X}).
\]
The corresponding closed-string theory is more familiarly known as the {\it Rabinowitz Floer homology} of (the boundary of the domain completing to $\bar{X}$ in) $\bar{X}$, $RFH^{*}(\bar{X})$, introduced in \cite{CF}, 
whose algebra structures have received recent attention \cites{COcone, CHOduality}.
For a general stop $\fk$ that admits a ribbon $F$, 
while it is natural to expect a closed-string theory for the category $\n(\bar{X}, F_{0})$ (which nonetheless does not carry a ring structure) should exist, it is not the purpose of this paper to describe that construction.
However in the open-string case, $\n(\bar{X}, F_{0})$ is a version of Rabinowitz theory for Liouville sectors,
which is much less functorial compared to the partially wrapped Fukaya category.
In particular, there is no pushforward functor associated to an arbitrary inclusion of stopped Liouville {\it sectors}.
However, it is still possible to find a certain restricted type of inclusions of stopped Liouville {\it manifolds} of the form $(\bar{X}, F_{0}) \xhookrightarrow{} (\bar{X}, F'_{0})$ for an inclusion of a Liouville subdomain $F'_{0} \subset F_{0}$.

\begin{thm}\label{thm:pushforward is localization}
Suppose $F'_{0} \subset F_{0}$ is a Liouville subdomain. 
Then there is a pushforward functor
\begin{equation}
\iota_{\sharp}: \n(\bar{X}, F_{0}) \to \n(\bar{X}, F'_{0})
\end{equation}
which is identity on objects, and on morphism spaces is given by the inclusion of subcomplexes.
If $\iota: (\bar{X}, F_{0}) \to (\bar{X}, F'_{0})$ is a stop removal inclusion,
the pushforward functor is a localization.
\end{thm}

On the other hand, Theorem \ref{thm:main} provides an effective tool for computing the category $\n(\bar{X}, F_{0})$ directly algebraically from $\w(\bar{X}, F_{0})$ and its subcategory $\tilde{\dd}(F)$,
where in many instances the latter categories can be explicitly described using geometry, as addressed in \cite{GPS2}.

\subsection{Functoriality and variants}

We have seen that the category $\n(\bar{X}, F_{0})$ is related to the partially wrapped Fukaya category $\w(\bar{X}, F_{0})$ in two ways: 
first, there is a natural $\ainf$-functor 
\begin{equation}
j: \w(\bar{X}, F_{0}) \to \n(\bar{X}, F_{0})
\end{equation}
which is realized by chain-level inclusions (not necessarily injective on cohomology though);
second, the formal completion of $\w(\bar{X}, F_{0})$ along a subcategory is quasi-equivalent to $\n(\bar{X}, F_{0})$.
There is yet one more interesting relation, which can be established by introducing certain `intermediate' variants of the category $\n(\bar{X}, F_{0})$.
The first variant among these is an $\ainf$-category 
\begin{equation}
\n_{>-2}(\bar{X}, F_{0})
\end{equation}
defined in a way similar to $\n(\bar{X}, F_{0})$,
but by only taking negative wrappings of Lagrangians passing through the stop at most once.
This category plays a special role as it is expected to capture information on the stop.
In more detailed terms, the morphism spaces in $\n_{>-2}(\bar{X}, F_{0})$ are quotient spaces of the morphism spaces in $\n(\bar{X}, F_{0})$ by certain subspaces
(see \S\ref{sec: variants} for more details).
In \cite{seidel6}, Seidel gave a different yet equivalent construction of $\n_{>-2}(\bar{X}, F_{0})$ in the case of a Lefschetz fibration $\pi: \bar{X} \to \C$, in which case the morphism spaces in that category are cones of continuation maps, 
and proved that it is the fiber at $\infty$ of a noncommutative pencil, 
and is in turn quasi-equivalent to the Fukaya category of the fiber $F$.
In particular, there is a natural $\ainf$-functor
\begin{equation}
\pi: \n(\bar{X}, F_{0}) \to \n_{>-2}(\bar{X}, F_{0})
\end{equation}
which can be interpreted as the `restriction-to-the-fiber' functor.

More generally, in \S\ref{sec: variants}  we construct for each negative integer $m \le -1$ an $\ainf$-category 
\begin{equation}
\n_{>m}(\bar{X}, F_{0})
\end{equation}
 whose objects the same as those in $\w(\bar{X}, F_{0})$, 
and whose morphisms are Floer cochains for Lagrangian isotopies in the negative direction passing through the stop $\fk$ by at most $-m-1$ times.
These can be thought of as `infinitesimal thickenings' of the category $\n_{>-2}(\bar{X}, F_{0})$.
These fit into a diagram of $\ainf$-categories (see \S\ref{sec: diagram of Fukaya categories}) related by $\ainf$-functors
\begin{equation}
\pi_{m, m'}: \n_{>m}(\bar{X}, F_{0}) \to \n_{>m'}(\bar{X}, F_{0})
\end{equation}
for any $m, m' \in \Z_{\le -1}$ with $m \le m'$,
which are roughly speaking `restrictions' from negative wrappings which pass through the stop at most $-m-1$ times to those which pass through the stop at most $-m'-1$ times.
This diagram of $\ainf$-categories will be a key ingredient in the proof of Theorem \ref{thm:main2}.

These categories $\n_{>m}(\bar{X}, F_{0})$ come in company with natural $\ainf$-functors
\begin{equation}
j_{m}: \w(\bar{X}, F_{0}) \to \n_{>m}(\bar{X}, F_{0})
\end{equation}
and
\begin{equation}
\pi_{m}: \n(\bar{X}, F_{0}) \to \n_{>m}(\bar{X}, F_{0}),
\end{equation}
which are analogues of $j$ and $\pi$.
Intuitively, $j_{m}$ is an `inclusion', by allowing Lagrangians to be negatively wrapped past the stop at most $-m-1$ times,
and $\pi_{m}$ is a `restriction', by restricting to negative wrappings which pass through the stop at most $-m-1$ times.
By consistently choosing Floer data defining the various Fukaya categories, 
we may construct the functors so that 
\[
\{j_{m}: \w(\bar{X}, F_{0}) \to \n_{>m}(\bar{X}, F_{0})\}_{m \in \Z_{\le -1}}
\]
 and 
 \[
 \{\pi_{m}: \n(\bar{X}, F_{0}) \to \n_{>m}(\bar{X}, F_{0})\}_{m \in \Z_{\le -1}}
 \]
  form strictly compatible diagrams of functors.

\subsection{Relation with homological mirror symmetry}\label{intro: hms}

The idea of the construction of the category $\n(X, F_{0})$ is inspired by homological mirror symmetry for pairs,
as pointed out in \cite{seidel6}.
Because of the connection with the results of \cite{GGV}, 
it is also a warning to the reader that some expository parts in this subsection share a substantial amount of overlaps with those,
but are explored with a different focus.

Suppose the Liouville pair $(\bar{X}, F)$ is mirror to a pair $(Y, D)$,
where $Y$ is a smooth projective variety and $D$ is an anti-canonical ample divisor.
Homological mirror symmetry predicts a Morita equivalence between $\w(\bar{X}, F_{0})$ and $Coh(Y)$,
and one between $\w(\bar{X})$ and $Coh(Y \setminus D)$,
such that the two equivalences are compatible with the natural functors. 
That is,
\begin{equation}\label{hms for pairs}
\begin{tikzcd}
\perf \w(\bar{X}, F_{0}) \arrow[d, "\sim"] \arrow[r] & \perf \w(\bar{X}) \arrow[d, "\sim"] \\
Coh(Y) \arrow[r] & Coh(Y \setminus D).
\end{tikzcd}
\end{equation}
The full mirror symmetry picture should also describe $D$ as the mirror of the fiber $F$,
together with a Morita equivalence between $D^{b}Coh(D)$ and the Fukaya category of $F$,
compatible with other two equivalences in the diagram \eqref{hms for pairs}.
From this point of view, $F$ can be interpreted as a {\it divisor} of a `compactification' of $X$ when the relevant Fukaya categories are concerned.

In algebraic geometry, one can take the formal neighborhood of $D$ in $Y$, $\hat{Y}_{D}$.
whose category of (algebraizable) perfect complexes  $\perf_{alg}(\hat{Y}_{D})$ is defined by the obvious homotopy limit.
In \cite{efimov1}, Efimov proved that $\perf_{alg}(\hat{Y}_{D})$ can be constructed from $\perf(Y)$ and its subcategory $\perf_{D}(Y)$ of complexes supported on $D$,
via the categorical formal completion.
Going to the symplectic side, one can perform a similar construction to the partially wrapped Fukaya category, which yields \eqref{formal completion of partially wrapped}.

By removing the divisor $D$, one can also take the formal punctured neighborhood of $D$ in $Y$, $\hat{Y}_{D} \setminus D$,
which is not well-defined classically but has a well-defined category of perfect complexes due to \cite{efimov2}.
It turns out that the formal punctured neighborhood only depends on the divisor complement $Y \setminus D$,
and its category of (algebraizable) perfect complexes can be constructed purely algebraically from $Coh(Y \setminus D)$ (assuming $Y \setminus D$ is smooth).
Moreover, the construction of the formal punctured neighborhood of infinity can be generalized in a purely categorical way (also due to \cite{efimov2}), 
allowing us to consider on the symplectic side the formal punctured neighborhood of infinity of the wrapped Fukaya category,
denoted by $\widehat{\w(\bar{X})}_{\infty}$, 
which is shown in \cite{GGV} to be equivalent to the Rabinowitz Fukaya category $\rw(\bar{X})$ whenever $\bar{X}$ is non-degenerate in the sense of \cite{ganatra}.

The above symplecto-geometric interpretation for the formal punctured neighborhood of infinity also suggests a natural candidate for the symplectic mirror of the formal neighborhood of $D$ in $Y$.
By Theorem \ref{thm:main} or Theorem \ref{thm:main2}, it turns out that our category $\n(\bar{X}, F_{0})$ precisely matches with the categorical formal completion up to quasi-equivalence.
Under HMS, this provides a purely algebraic formula for the category $\n(\bar{X}, F_{0})$.

\begin{cor}\label{cor: coh of completion}
Let $\pi: \bar{X} \to \C$ be a symplectic Landau-Ginzburg model with compact critical locus and Weinstein fibers $F$.
Suppose $(\bar{X}, \pi)$ is mirror to a pair $(Y, D)$ where $Y$ is smooth projective and $D$ is an anti-canonical ample divisor,
such that HMS gives Morita equivalences $\mathcal{FS}(\bar{X}, \pi) \cong Coh(Y)$ and $\w(\bar{X}) \cong Coh(Y \setminus D)$.
Then the category $\n(\bar{X}, F_{0})$ is Morita equivalent to $\perf_{alg}(\hat{Y}_{D})$, 
the category of algebraizable perfect complexes on the formal scheme $\hat{Y}_{D}$.
\end{cor}

In nice situations, $\perf_{alg}(\hat{Y}_{D})$ can be computed directly explicitly as a homotopy limit.
Thus Corollary \ref{cor: coh of completion} provides an effective computation for the category $\n(\bar{X}, F_{0})$.

\subsection*{Overview of paper}
In \S\ref{section: algebra}, we recall the definition of the categorical formal completion due to Efimov,
and study its relation to the categorical formal punctured neighborhood.
In \S\ref{section: Fukaya categories}, we study various versions of Fukaya categories of a stopped Liouville manifold $(\bar{X}, F_{0})$,
define the new category $\n(\bar{X}, F_{0})$,
and establish relations between these Fukaya categories.
In \S\ref{section: equivalence}, we prove Theorem \ref{thm:main} (a consequence of Theorem \ref{thm:equivalence}) and the more general version Theorem \ref{thm:main2}.
In Appendix \ref{app: homotopy limits}, we provide a minimal package for a model for the homotopy limit needed for the relevant discussions in the paper.

\subsection*{Acknowledgements}
The author thanks Sheel Ganatra for inspiring this question and helpful conversations on details in partially wrapped Floer theory.
He also thanks Mohammed Abouzaid for pointing out additional potential applications of this theory.

\section{The categorical formal completion}\label{section: algebra}

\subsection{Algebraic preliminaries}

Fix a ground commutative ring $\K$.
Let $\cc$ be a (graded) $\ainf$-category over $\K$.
For objects $X_{0}, \ldots, X_{k} \in \ob \cc$, we write 
\[
\cc(X_{0}, X_{1}, \ldots, X_{k}) = \cc(X_{k-1}, X_{k}) \otimes_{\K} \cdots \otimes_{\K} \cc(X_{0}, X_{1})
\]
for the $k$-th fold tensor product of composable morphism spaces.
The dg category of left (resp. right) modules over $\cc$ is defined to be the dg category of $\ainf$-functors $\cc \mod = \fun(\cc, \ch_{\K})$ (resp. $\rmod \cc = \fun(\cc^{op}, \ch_{\K})$).
The Yoneda functors
\[
y^{l}: \cc \to \cc \mod, y^{r}: \cc \to \rmod \cc
\]
are fully faithful, where by a fully faithful $\ainf$-functor, we always mean it is cohomologically fully faithful.
Given two $\ainf$-categories $\cc$ and $\dd$, the dg category of $\ainf$-bimodules over $(\cc, \dd)$ (or $\cc-\dd$ bimodule) is the dg category of $\ainf$-bilinear functors $\cc \bimod \dd = \fun^{bi}(\cc^{op} \times \dd, \ch_{\K})$.

Given a left (resp. right) $\cc$-module $M$, its {\it linear dual} $M^{\vee} = \hom_{\K}(M, \K)$ is a right (resp. left) $\cc$-module, with grading specified by
\begin{equation}\label{grading on linear dual}
(M^{\vee})^{*}(X) = \hom_{\K}(M^{-*}(X), \K)
\end{equation}
for every $X \in \ob \cc$.

The convenience of working with $\ainf$-modules and bimodules is based on the following two standard facts (good references include \cite{seidel_book} \cite{seidel_ainf} \cite{ganatra}).

\begin{lem}\label{composition with quasi-iso}
Composition with any quasi-isomorphism of $\cc$-modules (or $\cc-\dd$ bimodules) induces quasi-isomorphism of morphism complexes in the category $\rmod \cc$ (or $\cc \bimod \dd$).
\end{lem}

\begin{lem}\label{quasi-iso invertible}
Any quasi-isomprhism of $\cc$-modules (or $\cc-\dd$ bimodules) has a homotopy inverse.
\end{lem}

For an $\ainf$-category $\cc$, its {\it derived category} of right $\ainf$-modules $D(\cc)$ is simply the homotopy category of right modules:
\[
D(\cc) = H^{0}(\rmod \cc),
\]
as any quasi-isomorphisms of $\ainf$-modules automatically has a homotopy inverse.
Similarly, the derived category of left $\ainf$-modules is denoted by
\[
\tilde{D}(\cc) = H^{0}(\cc \mod).
\]
In practice, one could be more interested in a smaller derived category, for example the bounded derived category, which is defined to be the homotopy category of twisted complexes
\begin{equation}\label{dbcat}
D^{b}(\cc) = H^{0}(\mathrm{Tw}\cc),
\end{equation}
as well as its Karoubi completion (or split closure)
\begin{equation}\label{dpicat}
D^{\pi}(\cc) = H^{0}(\perf \cc).
\end{equation}
All these derived categories are triangulated, induced from the natural pre-triangulated structures on $\rmod \cc, \mathrm{Tw}\cc$ and $\perf \cc$.
We refer the reader to \cite{seidel_book} for their constructions and basic properties.

Given a $(\cc, \dd)$-bimodule $P$ and a $(\dd, \mathcal{E})$-bimodule $Q$, their (convolution) {\it tensor product} $P \otimes_{\dd} Q$ is defined to be the $(\cc, \mathcal{E})$-bimodule with underlying cochain spaces
\[
P \otimes_{\dd} Q(X, Y) = \bigoplus_{Z_{0}, \ldots, Z_{s}} P(X, Z_{0}) \otimes \dd(Z_{0}, \ldots, Z_{s}) \otimes Q(Z_{s}, Y).
\]
For a right $\cc$-module $M$, and a $(\cc, \dd)$-bimodule $P$, their (convolution) tensor product $M \otimes_{\cc} P$ is a $\dd$-module with
\[
M \otimes_{\cc} P(X, Y) = \bigoplus_{Z_{0}, \ldots, Z_{s}} M(X, Z_{0}) \otimes \dd(Z_{0}, \ldots, Z_{s}) \otimes P(Z_{s}, Y).
\]

For a $\cc-\cc$ bimodule $P$, the {\it Hochschild cohomology cochain complex} of $\cc$ with coefficients in $P$ is defined to be
\begin{equation}
\r{CC}^*(\cc, P) = \prod_{X_{0}, \ldots, X_{d} \in \cc} \hom_{\K}(\cc(X_{0}, \ldots, X_{d}), P(X_{0}, X_{d})),
\end{equation}
with the differential 
\begin{equation}
\d(F)^{k}(x_{k}, \ldots, x_{1}) = \sum \mu^{r, s}_{P}(x_{k}, \ldots, F^{k-r-s}(\cdots), x_{s}, \ldots, x_{1}) - \sum F^{k-j+1}(x_{k}, \ldots, \mu^{j}_{\cc}(\cdots), x_{i}, \ldots, x_{1}).
\end{equation}
The morphism space in the dg category of $\ainf$-functors $\fun(\cc, \dd)$ can be interpreted as a Hochschild cohomology cochain complex of $\cc$ with coefficients in the pullback diagonal bimodule $(F, G)^{*} \D_{\dd}$
\begin{equation}\label{fun as cc}
        \fun(\cc, \dd)(F,G) = \r{CC}^*(\cc, (F, G)^{*} \D_{\dd}).
\end{equation}

Given a left $\cc$-module $M$ and a right $\dd$-module $N$, their {\it linear tensor product}
\[
M \otimes_{\K} N
\]
defines a $(\cc, \dd)$-bimodule, 
with underlying chain complex
\begin{equation}
(M \otimes_{\K} N)(X, Y) = M(X) \otimes_{\K} N(Y),
\end{equation}
the differential given by the tensor product differential,
and bimodule structure maps induced from the the module structure maps $\mu^{k}_{M}$ and $\mu^{k}_{N}$.
For a right $\cc$-module $M$ and a right $\dd$-module $N$, the {\it space of linear homomorphisms} from $M$ to $N$
\[
\hom_{\K}(M, N)
\]
defines a $(\cc, \dd)$-bimodule,
with underlying chain complex
\begin{equation}
\hom_{\K}(M, N)(X, Y) = \hom_{\K}(M(X), N(Y)), X \in \ob \cc, Y \in \ob \dd,
\end{equation}
the differential given by the $\hom_{\K}$-differential, 
and bimodule structure maps induced from the module structure maps for $M$ and $N$.

\subsection{The categorical formal completion}\label{section: formal completion}

The categorical formal completion of a compactly generated (by a set of objects) enhanced triangulated category
along a full thick essentially small triangulated subcategory is introduced in \cite{efimov1} via the notion of the `derived double centralizer',
which in the case where the relevant categories come from algebraic geometry recovers the ordinary formal completion of a scheme along a closed subscheme in a purely categorical way without (but related to) limits.
The definitions carry over to the $\ainf$-setting in an almost straightforward way, and in fact can be greatly simplified thanks to Lemmas \ref{composition with quasi-iso} and \ref{quasi-iso invertible}.

Let $\cc$ be a small $\ainf$-category. 
Consider a small full subcategory $\T$ of the derived category of right $\cc$-modules $D(\cc)$, not necessarily triangulated.
We wish to consider the notion of the {\it categorical formal completion} of $\cc$ along $\T$.
Most of the definitions and basic properties in this subsection are essentially due to \cite{efimov1},
but we also include some new discussions in the favorable case where $\T$ is an actual subcategory of $\cc$,
including Lemma \ref{lem: quasi-iso of pt with qt}, Corollary \ref{pullback of P is right yoneda}, Proposition \ref{prop: morphism in t = in zt} which will be useful for computing the morphism spaces in the categorical formal completion.
Also, we should point out that we are performing the categorical completions along {\it left} modules while \cite{efimov1} performs the construction along {\it right} modules.
This is a technical point, but will be of importance to our construction in Fukaya categories.

First, following \cite{efimov1} (noting the difference in left and right modules), we make the following definition.

\begin{defn}
Let $\cc$ be a small $\ainf$-category over $\K$,
and $\T$ a full small subcategory of $\tilde{D}(\cc)$.
Define $Z_{\T}$ to be the $\ainf$-category with $\ob Z_{\T} = \ob \T$,
and with morphisms
\begin{equation}
Z_{\T}(M, N) = \hom_{\cc \mod}(M, N).
\end{equation}
That is, $Z_{\T}$ is the full subcategory of $\cc \mod$ with objects in $\T$.

If $\T$ itself is a subcategory of $\cc \mod$, we define $Z_{\T} = \T$.
\end{defn}

The subcategory $\T$ naturally gives rise to a $(Z_{\T}, \cc^{op})$-bimodule $P_{\T}$ as follows.
For any $(M, X) \in \ob Z_{\T} \times \ob \cc^{op}$, put
\begin{equation}\label{bimodule p}
P_{\T}(M, X) = M(X).
\end{equation}
The $\ainf$-bimodule structure comes from the dg-structure on $Z_{\T}$ (as a full subcategory of $\cc \mod$) and the $\ainf$-structure on $\cc$, compatible with Yoneda embedding.
More concretely, this bimodule structure is related to the diagonal bimodule of the category $\cc \mod$ in the following way.
Let 
\[
(\id_{\cc \mod}, y^{r})^{*} (\cc \mod)_{\D} = (\cc \mod)_{\D}(-, y^{r}(-)) = \hom_{\cc \mod}(y^{r}(-), -)
\]
 be the graph $(\cc \mod, \cc^{op})$ bimodule of the right Yoneda embedding $y^{l}: \cc \to \cc \mod$.
If $\hat{i}: Z_{\T} \xhookrightarrow{} \cc \mod$ denotes the inclusion of the subcategory $Z_{\T}$ of $\cc \mod$,
then there is a canonical map of $(Z_{\T}, \cc^{op})$-bimodules 
\begin{equation}\label{equivalence of P to pullback diagonal}
P_{\T} \to (\hat{i}, \id_{\cc^{op}})^{*} (id_{\cc \mod}, y^{l})^{*} (\cc \mod)_{\D} = (\hat{i}, y^{l})^{*} (\cc \mod)_{\D},
\end{equation}
which on chain complexes consists of the canonical maps
\begin{equation}\label{generalized Yoneda map}
M(X) \to \hom_{\cc \mod}(Y^{l}_{X}, M),
\end{equation}
for $M \in \ob Z_{\T} \subset \ob \cc \mod, X \in \cc$.
Fixing $X$, $P_{\T}(-, X)$ defines a left $Z_{\T}$-module, or equivalently a right $Z_{\T}^{op}$-module.
Phrased in a more functorial way, we may say that the $(Z_{\T}, \cc^{op})$-bimodule $P_{\T}$ defines a module-valued functor
\begin{equation}\label{module-valued functor}
F_{\T} = F(P_{\T}): (\cc^{op})^{op} = \cc \to \rmod Z_{\T}^{op}.
\end{equation}

\begin{defn}\label{defn: formal completion}
The formal completion $\ct$ of $\cc$ along $\T$ is defined to be the $\ainf$-category with objects $\ob \ct = \ob \cc$,
and morphisms
\begin{equation}\label{morphism in completion}
\ct(X, Y) = \hom_{\rmod Z_{\T}^{op}}(P_{\T}(-, X), P_{\T}(-, Y)).
\end{equation}
That is, $\ct$ is the essential image of $\cc$ in $\rmod Z_{\T}^{op}$ under the module-valued functor $F_{\T}$ \eqref{module-valued functor}.
\end{defn}

Immediately by definition, there is a natural pushforward functor
\begin{equation}\label{pushforward to formal completion}
\iota_{\T}: \cc \to \ct
\end{equation}
which is identity on objects, and is essentially the functor $F_{\T}$.
On the level of morphism spaces, it may be written as
\[
\cc(X, Y) \stackrel{F_{\T}}\to \hom_{\rmod Z_{\T}^{op}}(P_{\T}(-, X), P_{\T}(-, Y)) = \ct(X, Y),
\]
which sends a morphism $c \in \cc(X, Y) = \cc^{op}(Y, X)$ to the canonical pre-module homomorphism $P_{\T}(-,X) \to P_{\T}(-,Y)$ (as right modules over $\cc^{op}$) by post-composing with $c$ using the right module action by $\cc^{op}$ via the bimodule maps $\mu^{k, 1}_{P_{\T}}$.
Higher order maps are defined in a similar way, giving rise to pre-module homomorphisms by post-composing with composable chains of morphisms in $\cc$ using the higher order bimodule structure maps.

\begin{lem}\label{invariance of formal completion}
If $\T_{1}, \T_{2}$ are two subcategories of $\tilde{D}(\cc)$ which split-generate each other, 
then the two formal completions $\widehat{\cc}_{\T_{1}}, \widehat{\cc}_{\T_{2}}$ are naturally quasi-equivalent.
\end{lem}
\begin{proof}
Define a $(Z_{\T_{2}}, Z_{\T_{1}})$-bimodule $S$, or equivalently a bimodule over $(Z_{\T_{1}}^{op}, Z_{\T_{2}}^{op})$ as follows.
Let $P_{\T_{1}}, P_{\T_{2}}$ be the bimodules defined as \eqref{bimodule p} for the two subcategories $\T_{1}$ and $\T_{2}$ respectively.
On the level of underlying cochain spaces, the bimodule $S$ takes values
\begin{equation}\label{bimodule S}
S(M_{2}, M_{1}) = \hom_{\cc \mod}(P_{\T_{1}}(M_{1}, -), P_{\T_{2}}(M_{2}, -)), M_{1} \in \ob Z_{\T_{1}}, M_{2} \in \ob Z_{\T_{2}},
\end{equation}
which carries the natural bimodule structure via the inclusions of full subcategories $Z_{\T_{i}} \xhookrightarrow{} \cc \mod$.
The convolution tensor product with $S$ induces an $\ainf$-functor 
\[
Z_{\T_{1}} \mod = \rmod Z_{\T_{1}}^{op} \to Z_{\T_{2}} \mod = \rmod Z_{\T_{2}}^{op}.
\]
Since $\T_{1}, \T_{2}$ split-generate each other, we have that $\perf Z_{\T_{1}}^{op} = \perf Z_{\T_{2}}^{op}$, 
which implies that $\rmod Z_{\T_{1}}^{op} \cong \rmod \perf Z_{\T_{1}}^{op}$ is quasi-equivalent to $\rmod Z_{\T_{2}}^{op} \cong \rmod \perf Z_{\T_{2}}^{op}$.
In fact, such an equivalence is induced by the convolution tensor product with $S$.

Now let us consider the convolution tensor product 
$S \otimes_{Z_{\T_{1}}} P_{\T_{1}} $
of the $(Z_{\T_{1}}, \cc)$-bimodule $P_{\T_{1}}$ with $S$.
By the definition of $S$ in \eqref{bimodule S}, there is a natural morphism of $(Z_{\T_{2}}, \cc)$-bimodules
\begin{equation}\label{ev bimodule}
ev: S \otimes_{Z_{\T_{1}}} P_{\T_{1}}  \to P_{\T_{2}},
\end{equation}
which for each $X \in \ob \cc$ specializes to a morphism of $Z_{\T_{2}}^{op}$-modules,
\[
(S \otimes_{Z_{\T_{1}}} P_{\T_{1}})(-, X) \to P_{\T_{2}}(-, X).
\]
By definition of convolution tensor product, there is a natural inclusion map
\[
S \otimes_{Z_{\T_{1}}} P_{\T_{1}}(-,X) \to (S \otimes_{Z_{\T_{1}}} P_{\T_{1}})(-, X),
\]
and therefore the composition of the above two morphisms
\begin{equation}\label{ev module}
\tilde{ev}: S \otimes_{Z_{\T_{1}}} P_{\T_{1}}(-,X) \to P_{\T_{2}}(-, X).
\end{equation}
It is not hard to see that \eqref{ev module} is a quasi-isomorphism,
because for every $M_{2} \in \ob Z_{\T_{2}}$, the $\cc$-module $P_{\T_{2}}(M_{2}, -) = M_{2}(-)$ is in the split closure of $\T_{1}$ by the assumption that $\T_{1}$ split-generates $\T_{2}$.
\end{proof}

The categorical formal completion can also be defined when $\T$ is a subcategory of $\cc$, or a subcategory of $\mathrm{Tw}\cc$, or $\perf \cc$.
We will be mostly interested in such cases in applications.

\begin{defn}
If $\T$ is a subcategory of $\cc$ (or $\r{Tw}\cc$ or $\perf \cc$), 
we define $Z_{\T}$ to be the image of $\T^{op}$ in $\cc \mod$ under the Yoneda embedding $\T^{op} \xhookrightarrow{} \cc^{op} \to \cc \mod$ (naturally extended over $\mathrm{Tw}\cc^{op}$ or $\perf \cc^{op}$).
That is, the objects of $Z_{\T}$ are left Yoneda modules $Y^{l}_{K}$ for every $K \in \ob \T^{op}$, 
and morphisms are
\[
Z_{\T}(Y^{l}_{K}, Y^{l}_{L}) = \hom_{\cc \mod}(Y^{l}_{K}, Y^{l}_{L})
\]
The bimodule $P_{\T}$ has the following underlying cochain space
\begin{equation}\label{bimodule p for t in c}
P_{\T}(Y^{l}_{K}, X) = Y^{l}_{K}(X) = \hom_{\cc}(K, X) = \hom_{\cc^{op}}(X, K),
\end{equation}
but the bimodule structure maps are different from those for the diagonal bimodule $\cc^{op}_{\D}$ of $\cc^{op}$,
and are induced by $\ainf$-structure maps $\mu^{k}_{\cc \mod}$ and $\mu^{k}_{\cc^{op}}$ of the categories $\cc \mod$ and $\cc^{op}$.

Given such $Z_{\T}$ and bimodule $P_{\T}$, define the formal completion $\ct$ in the same way as Definition \ref{defn: formal completion}.
\end{defn}

With this definition, Lemma \ref{invariance of formal completion} immediately implies:

\begin{cor}\label{cor: completion along perfect complexes}
Suppose $\T$ is a full subcategory of $\cc$ (or $\r{Tw}\cc$ or $\perf \cc$). 
Then there is a quasi-equivalence
\[
\ct \cong \widehat{\cc}_{\perf \T}.
\]
\end{cor}
\begin{proof}
Since $\T$ is a full subcategory of $\cc$, the split-closure of $\T^{op}$ in $\cc \mod$ under the Yoneda embedding $\T^{op} \xhookrightarrow{} \cc^{op} \to \cc \mod$ is $\perf \T^{op}$.
But for perfect complexes we have $(\perf \T)^{op} \cong \perf \T^{op}$.
\end{proof}

Even if $\T$ is a subcategory of $\cc$ (or $\r{Tw}\cc$ or $\perf \cc$), the bimodule $P_{\T}$ is different from but indeed related to the diagonal bimodule of $\cc^{op}$.
If $\T$ is a subcategory of $\cc$ (or $\r{Tw}\cc$ or $\perf \cc$) with inclusion $i: \T \xhookrightarrow{} \cc$ (or $\r{Tw}\cc$ or $\perf \cc$), 
then by definition the composition of the natural functor $j: \T^{op} \to Z_{\T}$ with the inclusion $\hat{i}: Z_{\T} \xhookrightarrow{} \rmod \cc$ is just the Yoneda embedding restricted to $\T^{op}$, 
\[
\hat{i} \circ j = y^{l}|_{\T^{op}}: \T^{op} \xhookrightarrow{} \cc^{op} (\text{or } \r{Tw}\cc^{op} \text{ or } \perf \cc^{op}) \to \cc \mod,
\]
which factors through $i^{op}$.
There is a graph $(\T^{op}, \cc^{op})$-bimodule for $i^{op}: \T^{op} \to \cc^{op}$ defined by
\begin{equation}\label{bimodule qt}
Q_{\T} = (i^{op}, \id_{\cc^{op}})^{*} \cc^{op}_{\D},
\end{equation}

The relation between the bimodule $P_{\T}$ and the diagonal bimodule $(\cc \mod)_{\D}$ is much simpler when $\T$ is a full subcategory of $\cc$.

\begin{lem}
Suppose $\T$ is a full subcategory of $\cc$.
Then the bimodule map $P_{\T} \to (\hat{i}, y^{l})^{*} (\cc \mod)_{\D})$ \eqref{equivalence of P to pullback diagonal} is a quasi-isomorphism.
\end{lem}
\begin{proof}
The canonical map $M(X) \to \hom_{\cc \mod}(Y^{l}_{X}, M)$ \eqref{generalized Yoneda map} is a precisely the chain map for the Yoneda functor if $M$ is a left Yoneda module.
\end{proof}

\begin{lem}\label{lem: quasi-iso of pt with qt}
Suppose $\T$ is a full subcategory of $\cc$.
There is a quasi-isomorphism of $(\T^{op}, \cc^{op})$-bimodules
\begin{equation}
(j, \id_{\cc})^{*} P_{\T} \stackrel{\sim}\to Q_{\T}.
\end{equation}
\end{lem}
\begin{proof}
By the quasi-isomorphism \eqref{equivalence of P to pullback diagonal}, there is a quasi-isomorphism of $(\T, \cc)$-bimodule
\[
(j, \id_{\cc})^{*} P_{\T} \stackrel{\sim}\to (y^{l}|_{\T^{op}}, y^{l})^{*}(\cc \mod)_{\D}.
\]
Yoneda lemma asserts that there is a quasi-isomorphism of $(\cc^{op}, \cc^{op})$-bimodules from $\cc^{op}_{\D}$ to the pullback
\[
(y^{l}, y^{l})^{*}(\cc \mod)_{\D}.
\]
By the definition \eqref{bimodule qt}, the lemma follows by choosing a homotopy inverse of $\cc^{op}_{\D} \stackrel{\sim}\to (y^{l}, y^{l})^{*}(\cc \mod)_{\D}$.
\end{proof}

\begin{cor}\label{pullback of P is right yoneda}
Suppose $\T$ is a full subcategory of $\cc$.
For each $X \in \ob \cc$, there is a quasi-isomorphism of left $\T^{op}$-modules, or equivalently right $\T$-modules
\begin{equation}\label{quasi-iso of t modules}
 j^{*} P_{\T}(-, X) \stackrel{\sim}\to i^{*}Y^{r}_{X},
\end{equation}
where $i: \T \to \cc$ is the inclusion, $j: \T^{op} \to Z_{\T}$ is the image of the left Yoneda functor $\T^{op} \to \cc^{op} \to \cc \mod$.
\end{cor}
\begin{proof}
A direct computation shows 
\[
Q_{\T}(-, X) = (i^{op}, \id_{\cc^{op}})^{*}\cc^{op}_{\D}(-, X) = \cc^{op}_{\D}(i^{op}(-), X) = \cc^{op}(X, i^{op}(-)) = \cc(i(-), X) = i^{*}Y^{r}_{X}.
\]
\end{proof}

This has the following consequence on computing morphism spaces in $\ct$:

\begin{prop}\label{prop: morphism in t = in zt}
Suppose $\T$ is a full subcategory of $\cc$.
Then for every $X, Y \in \ob \cc = \ob \ct$, there is a quasi-isomorphism of chain complexes
\begin{equation}
\hom_{\rmod \T}(i^{*}Y^{r}_{X}, i^{*}Y^{r}_{Y}) \to \hom_{\rmod Z_{\T}^{op}}(P_{\T}(-, X), P_{\T}(-, Y)).
\end{equation}
Moreover, this quasi-isomorphism is compatible with the product structure on the dg categories $\rmod \T$ and $\rmod Z_{\T}^{op}$.
By the interpretation of $\hom$ in the functor category as Hochschild cohomology cochain complexes given in \eqref{fun as cc}, this is a quasi-isomorphism
\begin{equation}\label{pullback of cc}
\r{CC}^{*}(\T^{op}, \hom_{\K}(i^{*}Y^{r}_{X}, i^{*}Y^{r}_{Y})) \stackrel{\sim}\to \r{CC}^{*}(Z_{\T}, \hom_{\K}(P_{\T}(-, X), P_{\T}(-, Y))),
\end{equation}
which is compatible with the Yoneda products on Hochschild cohomology cochain complexes.
In particular, this provides a (partially-defined) quasi-equivalence from a subcategory of $\rmod \T$ having objects being pullbacks of Yoneda modules associated to objects of $\cc$ to $\rmod Z_{\T}^{op}$ with vanishing higher order terms.
\end{prop}
\begin{proof}
Since $\T$ is a full subcategory of $\cc$, the functor $j: \T^{op} \to Z_{\T}$ is a quasi-equivalence.
By Corollary \ref{pullback of P is right yoneda}, there is a quasi-isomorphism of $(\T, \T)$-bimodules
\begin{equation}
\hom_{\K}(i^{*}Y^{r}_{X}, i^{*}Y^{r}_{Y}) \stackrel{\sim}\to \hom_{\K}(j^{*} P_{\T}(-, X), j^{*} P_{\T}(-, Y)) = (j, j)^{*} \hom_{\K}(P_{\T}(-, X), P_{\T}(-, Y)).
\end{equation}
Thus, the quasi-isomorphism \eqref{pullback of cc} follows from the general property of Hochschild cohomology,
and it respects the Yoneda products.
\end{proof}

If $\T$ is a subcategory of $\cc$, there is also a functor
\begin{equation}\label{k functor}
\kappa: \T \to \ct
\end{equation}
which is the restriction of $\iota_{\T}$ \eqref{pushforward to formal completion} to $\T$.
In this case, the following lemma is a direct consequence of a computation of morphism spaces, with the help of Yoneda lemma.

\begin{lem}\label{property of k functor}
Suppose $\T$ is a full subcategory of $\cc$.
Then the functor $\kappa: \T \to \ct$ is fully faithful.
\end{lem}
\begin{proof}
When $\T$ is a subcategory of $\cc$, the category $Z_{\T}$ is the image of $\T^{op} \xhookrightarrow{} \cc^{op} \to \cc \mod$,
and the natural functor $j: \T^{op} \to Z_{\T}$ is a quasi-equivalence.
By Corollary \ref{pullback of P is right yoneda}, the pullback $j^{*}P_{\T}(-,K)$ is quasi-isomorphic to $i^{*}Y^{r}_{K} = Y^{r}_{K}$ for $K \in \ob \T \subset \ob \cc$.
Thus, there are canonical quasi-isomorphisms
\begin{equation}
\cc(K, L) = \T(K, L) \stackrel{\sim}\to \hom_{\rmod \T}(Y^{r}_{K}, Y^{r}_{L}) \stackrel{\sim}\to \hom_{\rmod Z_{\T}^{op}}(P_{\T}(-, K), P_{\T}(-, L))
\end{equation}
for every $K, L \in \ob \T$,
where the first quasi-isomorphism is by Yoneda lemma for $\T$, and the second quasi-isomorphism is given Proposition \ref{prop: morphism in t = in zt}.
This implies that $\kappa$ is fully faithful.
\end{proof}

Let us give two concrete examples to illustrate the above definitions, as a conclusion to this subsection.

\begin{example}
Consider an $\ainf$-category with a single object $X$ whose endomorphism space is an $\ainf$-algebra $\mathcal{A}$. 
Let $\mathcal{M}$ be a left $\ainf$-module over $\mathcal{A}$, which gives rise to a subcategory $\T$ of the dg category of right modules over the given $\ainf$-category, or equivalently the dg category of left modules over $\mathcal{A}$, 
which has a single object $\mathcal{M}$.
In this case,
\begin{equation}
Z_{\T} = Z_{\mathcal{M}} := \hom_{\mathcal{A} \mod}(\mathcal{M}, \mathcal{M}) = \prod_{k \ge 0} \hom_{\K}(\mathcal{A}^{\otimes k} \otimes  \mathcal{M},  \mathcal{M})[-k].
\end{equation}
This is a dg algebra, and there is an obvious projection morphism $Z_{\mathcal{M}} \to \hom_{\K}(\mathcal{M}, \mathcal{M})$ of dg algebras,
making $\mathcal{M}$ naturally a right module over $Z_{\mathcal{M}}^{op}$.
The formal completion is given by the following formula
\begin{equation}
\widehat{\mathcal{A}}_{\T} = \widehat{\mathcal{A}}_{\mathcal{M}} := \hom_{\rmod Z_{\mathcal{M}}^{op}}(\mathcal{M}, \mathcal{M}).
\end{equation}
\end{example}

\begin{example}[A corollary of Theorem 1.1 of \cite{efimov1}]
Let $\cc = \perf(X)$ be the dg category of perfect complexes over a smooth and proper scheme $X$ over a field $\K$,
and $\T = \perf_{Y}(X)$ be the subcategory of perfect complexes with cohomology supported on a closed subscheme $Y \subset X$.
Define the formal neighborhood $\widehat{X}_{Y}$ of $Y$ in $X$, as well as the dg category of algebraizable perfect complexes $\perf_{alg}(\widehat{X}_{Y})$ by taking the homotopy limit over the $n$-th infinitesimal neighborhood $Y_{n}$ of $Y$ with respect to the obvious inclusions $i_{n, n+1}: Y_{n} \to Y_{n+1}$ and $i_{n}: Y_{n} \to X$.
In this case, there is a natural quasi-equivalence
\begin{equation}
 \perf_{alg}(\widehat{X}_{Y}) \stackrel{\sim}\to \ct = \widehat{\perf(X)}_{\perf_{Y}(X)}.
\end{equation}
\end{example}

\subsection{Relation to the categorical formal punctured neighborhood}\label{section: relation to formal punctured neighborhood}

There is another categorical construction, called the {\it categorical formal punctured neighborhood} of an $\ainf$-category,
also due to Efimov \cite{efimov2}.
To give the definition, we first recall that the dg-category of Calkin complexes over $\K$, $\calk_{\K}$,
is defined to be the dg-quotient of all chain complexes over $\K$ by perfect chain complexes:
\begin{equation}
\calk_{\K} = \ch_{\K} / \perf \K.
\end{equation}
It comes with a natural projection $\ch_{\K} \to \calk_{\K}$,
so that for any $\ainf$-category $\B$, 
the Yoneda functor $y: \B \to \fun(\B^{op}, \ch_{\K})$ further gives rise to a functor $\bar{y}: \B \to \fun(\B^{op}, \calk_{\K})$.
More generally, we can define
\begin{equation}
\calk_{\B} = \rmod \B / y(\B)
\end{equation}
the quotient of $\rmod \B$ by the subcategory of representable modules, i.e. the image of the Yoneda functor $y: \B \to \rmod \B$.
Its idempotent split-closure is denoted by $\overline{\calk}_{\B}$.

\begin{defn}
The categorical formal punctured neighborhood of a small $\ainf$-category $\B$, denoted by $\binf$,
is the essential image of the functor $\bar{y}: \B \to \fun(\B^{op}, \calk_{\K})$.
\end{defn}

For a smooth $\B$, this can be refined as follows. 
We first define a functor $F_{\B}: \fun(\B^{op}, \calk_{\K}) \to \overline{\calk}_{\B}$ to be the composition
\begin{equation}\label{functor to calkin modules}
F_{\B}: \fun(\B^{op}, \calk_{\K}) \to \perf(\B \otimes \calk_{\K}) \to \overline{\calk}_{\B},
\end{equation}
where smoothness ensures that the first functor $\fun(\B^{op}, \calk_{\K}) \to \perf(\B \otimes \calk_{\K})$ is a fully faithful embedding.
By definition the induced Yoneda functor $\bar{y}: \B \to \fun(\B^{op}, \calk_{\K})$ has image in $\ker(F_{\B})$.
Then we define $\binf$ to be the essential image of $\bar{y}: \B \to \ker(F_{\B})$.

In the following situation, the notion of the formal completion along a subcategory is related to the formal punctured neighborhood of infinity.
Consider a short exact sequence of small $\ainf$-categories
\begin{equation}\label{starting ses}
\T \xhookrightarrow{} \cc \to \B.
\end{equation}
To describe the relation between $\ct$ and $\binf$, we first give a more explicit description of $\binf$ in the presence of the short exact sequence \eqref{starting ses}.
The pullback functor on the categories of right modules $\rmod \cc \to \rmod \T$ preserves proper modules, thus induces a functor
\begin{equation}\label{prop c to prop t}
r_{\T}: \r{Prop} \cc = \perf \cc \to \r{Prop} \T,
\end{equation}
where $\r{Prop} \cc = \perf \cc$ because $\cc$ is smooth and proper. 
In particular, by restricting to $\cc$ we get a functor $\cc \to \perf \cc \to \r{Prop} \T$.
Since $\T$ is proper, every perfect complex defines a proper module, i.e. $\perf \T$ embeds in $\r{Prop} \T$ under the Yoneda embedding, so we have a fully faithful embedding $\T \xhookrightarrow{} \perf \T \xhookrightarrow{} \r{Prop} \T$.
Then we take the quotient by $\T$ to get a functor
\begin{equation}
\bar{r}_{\T}: \cc / \T \cong \B \to \r{Prop} \T / y(\T)
\end{equation}
where the first quasi-equivalence comes from the short exact sequence \eqref{starting ses}.
Since $y(\T)$ split-generates $\perf \T$, the projection $\r{Prop} \T / y(\T) \to \r{Prop} \T / \perf \T$ is a quasi-equivalence,
so that the above functor can be replaced by the following one
\begin{equation}
\bar{r}_{\T}: \B \to \r{Prop} \T / \perf \T.
\end{equation}

\begin{lem}\label{lem: binf as prop t}
There is a quasi-equivalence 
\begin{equation}\label{binf as prop t}
\binf \stackrel{\sim}\to ess-im(\bar{r}_{\T}: \B \to \r{Prop} \T / \perf \T).
\end{equation}
\end{lem}
\begin{proof}
By applying $\overline{\calk}$ and $\fun(-, \calk_{\K})$ to the categories in the short exact sequence, we obtain a commutative diagram.
By tracing the commutative diagram, Theorem 3.2 of \cite{efimov2} asserts that there a quasi-equivalence 
\[
 \ker(F_{\B})  \stackrel{\sim}\to \overline{\r{Prop} \T / \perf \T},
\]
where $F_{B}$ is as in \eqref{functor to calkin modules}, and $\overline{\r{Prop} \T / \perf \T}$ is the idempotent split closure of $\r{Prop} \T / \perf \T$.
This is compatible with the natural functors $\B \to \ker(F_{\B})$ and $\B \stackrel{\bar{r}_{\T}}\to \r{Prop} \T / \perf \T$,
so we conclude that there is a quasi-equivalence
\[
 \binf \stackrel{\sim}\to ess-im(\bar{r}_{\T}: \B \to \r{Prop} \T / \perf \T).
\]
\end{proof}

\begin{thm}\label{quotient of formal completion}
Suppose we have a short exact sequence \eqref{starting ses} in which $\cc$ is both smooth and proper, 
so that $\T$ is proper and $\B$ is smooth.
Then there is a canonical functor
\begin{equation}\label{functor from completion to punctured}
\hat{\pi}: \ct \to \binf
\end{equation}
inducing a quasi-equivalence
\begin{equation}\label{equivalence from quotient of completion to punctured}
\bar{\hat{\pi}}: \ct / \kappa(\T) \to \binf,
\end{equation}
where $\kappa$ is the functor \eqref{k functor}.
\end{thm}

\begin{proof}
To construct the desired functor $\ct \to \binf$, it suffices to construct a functor $\ct \to \r{Prop} \T$,
and show that modulo the image of the fully faithful embedding $\iota_{\T}: \T \to \ct$ Lemma \ref{property of k functor},
the image of the quotient functor lands in the essential image of $\B \to \r{Prop} \T / \perf \T$.

Since $\T$ is proper, and for each $X \in \ob \cc$ the right Yoneda module $Y^{r}_{X}$ is perfect and hence proper,
it follows that the right $\T$-module $i^{*} Y^{r}_{X}$ is proper. 
This allows us to define a functor 
\begin{equation}\label{completion to prop t}
\tilde{\pi}: \ct \to \r{Prop} \T,
\end{equation}
which on the level of objects sends each $X \in \ob \ct = \ob \cc$, presented as the $Z_{\T}^{op}$-module $P_{\T}(-, X)$, 
to $i^{*}Y^{r}_{X} \in \ob \rmod \T$,
and on the level of morphism spaces is given by the following composition of maps
\begin{equation}\label{chain map of pi functor}
\begin{split}
\tilde{\pi}^{1}: \ct(X, Y) = & \hom_{\rmod Z_{\T}^{op}}(P_{\T}(-, X), P_{\T}(-, Y)) \stackrel{j^{*}}\to \hom_{\rmod \T}(j^{*} P_{\T}(-, X), j^{*}P_{\T}(-, Y))\\
& \stackrel{\sim}\to \hom_{\rmod \T}(i^{*}Y^{r}_{X}, i^{*} Y^{r}_{Y}),
\end{split}
\end{equation}
where the second quasi-isomorphism is \eqref{quasi-iso of t modules}.
By the definition of the pullback of modules, the second order map of $\tilde{\pi}$ is given by the second order map of $j^{*}$,
which in this case is zero as $j: Z_{\T} \xhookrightarrow{} \cc \mod$ is a strict inclusion of subcategories.
Since both $\ct$ and $\r{Prop} \T$ are dg categories, there are no higher order terms,
so the construction of the functor $\tilde{\pi}: \ct \to \r{Prop} \T$ is complete.
By further projecting to the quotient, it induces a functor 
\begin{equation}\label{completion to prop t mod perf t}
\bar{\tilde{\pi}}: \ct \to \r{Prop} \T / \perf \T.
\end{equation}

Next, recall that the functor $r_{\T}: \r{Prop} \cc \to \r{Prop} \T$ \eqref{prop c to prop t} is the pullback functor for right modules,
so the induced functor $\bar{r}_{\T}: \B \cong \cc / \T \to \r{Prop} \T / \perf \T$ sends each $X \in \ob \B$, regarded as an object of $\cc$ via localization, 
to the pullback of the right Yoneda module $i^{*} Y^{r}_{X}$ over $\T$.
It follows that the image of $\bar{\tilde{\pi}}: \ct \to \r{Prop} \T / \perf \T$ lands in the essential image of $\bar{r}_{\T}: \B \to \r{Prop} \T / \perf \T$.
Using the identification of $\binf$ with the essential image of $\bar{r}_{\T}: \B \to \r{Prop} \T/\perf \T$ by Lemma \ref{lem: binf as prop t}, we obtain the desired functor $\hat{\pi}: \ct \to \binf$ \eqref{functor from completion to punctured}.
Our construction immediately implies that the image of $\T$ in $\ct$ under $\kappa$ is mapped to zero in $\binf$ by $\hat{\pi}$.
Hence, the functor $\hat{\pi}$ induces $\bar{\hat{\pi}}: \ct / \T \to \binf$ \eqref{equivalence from quotient of completion to punctured}.

It remains to show that $\bar{\hat{\pi}}: \ct / \T \to \binf$ is a quasi-equivalence.
Clearly, it is essentially surjective by the construction above, based on Lemma \ref{lem: binf as prop t}.
To show fully-faithfulness, observe that the functor $\tilde{\pi}: \ct \to \r{Prop} \T$ \eqref{completion to prop t} is fully faithful,
since its linear terms \eqref{chain map of pi functor} are compositions of two quasi-isomorphisms by Proposition \ref{prop: morphism in t = in zt}.
Since $\kappa: \T \to \ct$ is fully faithful by Lemma \ref{property of k functor}, and its image under $\tilde{\pi}: \ct \to \r{Prop} \T$ is mapped to the subcategory $y(\T)$ consisting of Yoneda modules, 
the induced quotient functor \eqref{completion to prop t mod perf t} is well-defined on the quotient $\ct/\kappa(\T)$ and thus gives rise to a functor
\begin{equation}
\bar{\tilde{\pi}}: \ct/\kappa(\T) \to \r{Prop} \T / \perf \T,
\end{equation}
which is fully faithful because it factors as $\ct/\kappa(\T) \to \r{Prop} \T / y(\T) \to \r{Prop} \T / \perf \T$.
This implies that $\bar{\hat{\pi}}: \ct / \kappa(\T) \to \binf$ \eqref{equivalence from quotient of completion to punctured} is fully faithful.
\end{proof}

Stating the last conclusion in Theorem \ref{quotient of formal completion} in a slightly different way, we have:

\begin{cor}\label{unique fit into ses}
If $\cc$ is both smooth and proper, the short exact sequence \eqref{starting ses} induces a short exact sequence
\begin{equation}
\T \xhookrightarrow{} \ct \to \binf.
\end{equation}
\end{cor}

Note that, in Theorem \ref{quotient of formal completion} and Corollary \ref{unique fit into ses}, the assumption that $\cc$ is both smooth and proper is important.
The conclusions fail to hold if either assumption is dropped.

\subsection{The categorical formal completion as limits}\label{sec: formal completion as limits}

In algebraic geometry, the formal completion is defined via limits, which we mentioned in \S\ref{intro: hms}.
The goal of this subsection is to present a purely categorical analogue of that construction.
Let $\cc$ be an $\ainf$-category, and consider a diagram of $\ainf$-categories 
\[
\{\cc_{i}\}_{i \in I}
\]
indexed by some poset $I$ (or more generally any small category),
together with a strictly compatible system of functors 
\[
\{f_{i}: \cc \to \cc_{i}\}_{i \in I},
\]
where strict compatibility means that the composition of $\iota_{i}$ with any morphism $\cc_{i} \to \cc_{j}$ in the diagram corresponding to a morphism $i \to j$ in $I$ is exactly equal to (not just homotopic to) $f_{j}$.
We can define the {\it homotopy limit} of the diagram
\begin{equation}
\holim_{i \in I} \cc_{i},
\end{equation}
which comes with canonical functors
\begin{equation}
\pi_{i}: \holim_{i \in I} \cc_{i} \to \cc_{i}
\end{equation}
and
\begin{equation}
f: \cc \to \holim_{i \in I} \cc_{i},
\end{equation}
such $\pi_{i} \circ f$ agrees with $f_{i}$ up to homotopy. An explicit model of the homotopy limit is given in Appendix \ref{app: homotopy limits}.

For each object $X_{i} \in \ob \cc_{i}$, let $Y^{r}_{X_{i}}$ be the associated (right) Yoneda module.
The image of $\cc_{i}$ under the (right) Yoneda embedding $y_{\cc_{i}}: \cc_{i} \to \rmod \cc_{i}$, i.e. the subcategory whose objects are all Yoneda modules, is often referred to as the subcategory of {\it representable modules}.
The Yoneda embedding is compatible with pullback, in the sense that there is a canonical natural transformation
\begin{equation}\label{yoneda nat trans}
T_{F}: y_{\cc} \to F^{*} \circ y_{\dd} \circ F,
\end{equation}
for any functor $F: \cc \to \dd$.
Also recall that the category of {\it twisted complexes} $\mathrm{Tw} \cc_{i}$ over $\cc_{i}$ is a pre-triangulated closure of $\cc_{i}$, 
such that the Yoneda embedding extends over and maps $\mathrm{Tw} \cc_{i}$ to the smallest pre-triangulated subcategory of $\rmod \cc_{i}$ containing $y_{\cc_{i}}(\cc_{i})$.
The Yoneda lemma implies that $H^{0}(y_{\cc_{i}}(\cc_{i}))$ is a full subcategory of $D(\cc_{i})$,
whose triangulated closure is equivalent to the image of the homotopy category of twisted complexes over $\cc$
\[
H^{0}(\mathrm{Tw} \cc_{i})
\]
under the cohomology functor of the Yoneda embedding.
The pullback $f_{i}^{*} y_{\cc_{i}}(\cc_{i})$ gives a subcategory of $\rmod \cc$,
which is not in general contained in the subcategory $y_{\cc}(\cc)$ of representable $\cc$-modules,
but will be if $f_{i}$ is essentially surjective because of the natural transformation \eqref{yoneda nat trans}.
For the same reason, the categories $y_{\cc_{i}}(\cc_{i})$ do not in general form a a diagram unlike $\rmod \cc_{i}$, neither do their pullbacks $f_{i}^{*}y_{\cc_{i}}(\cc_{i})$.
However, for a single object $X \in \ob \cc$ we have a diagram of modules over $\cc$
\begin{equation}
\{f_{i}^{*}Y^{r}_{f_{i}(X)}\}_{i \in I},
\end{equation}
with morphisms of modules in the diagram given by components of the natural transformation \eqref{yoneda nat trans} associated to the functor $\cc_{i} \to \cc_{j}$ for a given morphism $i \to j$.
This is a well-defined diagram of $\cc$-modules exactly because of the compatibility of the functors $f_{i}$ with the functors $\cc_{i} \to \cc_{j}$ in the diagram.
In addition, there is a morphism $Y^{r}_{X} \to f_{i}^{*}Y^{r}_{f_{i}(X)}$ for every $i$, compatible with the diagram of $\cc$-modules above.
For each $\cc$-module $M$, we have a chain map
\begin{equation}
\begin{split}
&\lambda_{M}:  M(X) \to \hom_{\rmod \cc}(Y^{r}_{X}, M) \\
&\lambda_{M}(m)^{k}(u, c_{k-1}, \ldots, c_{1})  = \mu^{k+1}_{M}(m, u, c_{k-1}, \ldots, c_{1}),
\end{split}
\end{equation}
which is a generalization of the first order map of the Yoneda functor to an arbitrary module $M$.

Now consider the situation where all the categories $\cc_{i}$ have the same collection of objects and the functors $f_{ij}$ act as identity on objects.

\begin{defn}\label{def: subcat of constant sections}
In the situation as above, define the subcategory $\mathcal{S}(\{\cc_{i}\}_{i \in I})$ of the homotopy limit $\holim_{i \in I} \cc_{i}$ to be the subcategory consisting of objects being constant sections in the sense of Definition \ref{def: constant section}, i.e. concretely those sections with $s(i) = X$ for some object $X \in \ob \cc_{i}$ (for some and therefore any $i$) and $s_{ij} = 1_{X} \in \cc_{j}(X, X)$ (by assumption $f_{ij}X=X$).
\end{defn}

The following lemma is due to \cite{efimov1} in the case of dg algebras, 
but can be adapted to the case of $\ainf$-categories in a straightforward way,
as the proof only requires checking the statement on the level of quasi-isomorphisms of chain complexes (and will be omitted).
However, we do need a mild assumption on the diagram $\{\cc_{i}\}_{i \in I}$ following the setup of Appendix \S\ref{app: map from inverse limit}, stated as condition (i) in the following lemma.

\begin{lem}\label{formal completion as limits}
Assume $\T$ is a full subcategory of $\cc$. 
Let $\{\cc_{i}\}_{i \in I}$ be a diagram of $\ainf$-categories such that 
\begin{enumerate}[label=(\roman*)]

\item all the categories $\cc_{i}$ have the same collection of objects,
and that the functors $\cc_{i} \to \cc_{j}$ for $i \le j$ in the diagram act as identity on objects, and;

\item the pullback of the right Yoneda module $f_{i}^{*} Y^{r}_{f_{i}(X)}$ for every $X \in \ob \cc$ is representable by some object in $\perf \T$; 

\item For each $K \in \ob \T$ and each $X \in \ob \cc$, the natural map
\begin{equation}
\hocolim_{i \in I^{op}} \hom_{\rmod \cc}(f_{i}^{*} Y^{r}_{f_{i}(X)}, Y^{r}_{K}) \to \hom_{\rmod \cc}(Y^{r}_{X}, Y^{r}_{K})
\end{equation}
is a quasi-isomorphism of chain complexes.

\end{enumerate}
Then there is a natural quasi-equivalence
\begin{equation}
\hat{f}^{*}: \mathcal{S}(\{\cc_{i}\}_{i \in I}) \stackrel{\sim}\to \ct,
\end{equation}
where $\mathcal{S}(\{\cc_{i}\}_{i \in I})$ is defined in Definition \ref{def: subcat of constant sections}.
\end{lem} \qed

\section{Various versions of Fukaya categories}\label{section: Fukaya categories}

In this section, we briefly review the definition of the partially wrapped Fukaya category, 
and then introduce the main definition: the category $\n$ and discuss its properties.
Although the framework of \cite{GPS1} which uses the wrapping categories and categorical localizations to define the partially wrapped Fukaya category is nice and clean, 
for technical convenience we shall use concrete geometric perturbations by quadratic Hamiltonians as in \cite{sylvan1}.
The convenience is in two places: the definition of the category $\n$ and the construction of the functor in Theorem \ref{thm:main}.

\subsection{Liouville sectors and Liouville pairs} 

A Liouville sector of dimension $2n$ is a Liouville manifold-with-boundary $(X, \omega = d \lambda, Z)$,
where $\lambda$ is the Liouville form and $Z$ is the Liouville vector field,
on which there exists a function, called the {\it defining function} $I: \p X \to \R$ such that the Hamiltonian vector field $X_{I}$ points outward along $\p X$.
The defining function $I$ provides an identification of a {\it cylindrical neighborhood} $\nbd^{Z} \p X$ of $\p X$, i.e. a neighborhood that is invariant under the flow of $Z$), with $F \times \C_{\re \ge 0}$, 
where $F$ is a Liouville manifold, called the {\it symplectic boundary} of $X$, which is the completion of a Liouville domain $F_{0}$ sitting inside $\p_{\infty} X$ as a Liouville hypersurface with boundary.
The Liouville manifold $F$ can be obtained from $\p X$ by symplectic reduction as the quotient by the integrable characteristic foliation, which indeed carries two Liouville forms $\lambda_{\pm \infty}$ by embedding $F=I^{-1}(a)$ as $a \to \pm \infty$.
Every Liouville sector $X$ can be deformed such that $\lambda_{+\infty} = \lambda_{-\infty}$ (Proposition 2.28 of \cite{GPS1}),
in which case we say that $X$ has exact boundary.
There is an adequate supply of {\it cylindrical} almost complex structures $J$ on $X$ compatible with $\omega$ making the projection $\pi: \nbd^{Z} \p X \to \C_{\re \ge 0}$ holomorphic.

Given a Liouville sector $X$ as above, there is a Liouville manifold $\bar{X}$, called the {\it convex completion} of $X$, defined as gluing $F \times \C_{\re \le 0}$ to $X$ using the identification $\nbd^{Z} \p X \cong F \times \C_{\re \ge 0}$,
\begin{equation}
\bar{X} = X \bigcup_{F \times \C_{0 \le \re \le \e}} F \times \C_{\re \le \e}
\end{equation}
In this way we obtain a {\it sutured Liouville manifold}, meaning a pair $(\bar{X}, F_{0})$ where $\bar{X}$ is a Liouville manifold which is the completion of a Liouville domain, and $F_{0} \subset \p \bar{X}_{0}$ is a Liouville hypersurface with boundary.
Conversely, given a sutured Liouville manifold $(\bar{X}, F_{0})$, we may obtain a Liouville sector $X$ by removing a standard neighborhood of $F_{0}$, $X = \bar{X} \setminus (F_{0} \times \R_{s \ge 0} \times \R_{|t| \le \e})^{\circ}$.
In fact, the above correspondence is unique in the homotopical sense:

\begin{lem}[Lemma 2.32 of \cite{GPS1}]
Every Liouville sector $X$ comes from a unique (up to deformation) sutured Liouville manifold $(Y, F_{0})$.
In fact, $Y$ must agree with the convex completion $\bar{X}$ of $X$.
\end{lem}

Note that $F_{0} \subset \p_{\infty} X$ is a closed subset, 
so the pair $(\bar{X}, F_{0})$ is also called a {\it stopped Liouville manifold} following \cite{GPS2}.
The {\it core} $\fk = \ck_{F}$ of $F$ is the set of all points which not escape to infinity under the Liouville flow. 
It is a closed subset of $F_{0}$ and therefore of $\p_{\infty} \bar{X}$,
and therefore also a stop for $\bar{X}$.

A Lagrangian submanifold $L$ of a Liouville sector $X$ is understood to be properly embedded in the interior of $X$. 
$L$ is {\it cylindrical} if away from a compact set, $L$ is invariant under the flow of $Z$. 
In particular, any closed Lagrangian submanifold is cylindrical.
From now on we shall only consider exact cylindrical Lagrangians. 

A Hamiltonian $H: X \to \R$ is {\it linear} (resp. {\it quadratic}) at infinity, if outside of a compact set it satisfies
\[
dH(Z) = H (\text{ resp. } 2H)
\]
When constructing the partially wrapped Fukaya category using Hamiltonian perturbations, we shall instead go to the convex completion $\bar{X}$ and consider quadratic Hamiltonians on $\bar{X}$ satisfying certain conditions specified by the stop.

\begin{defn}[Definition 2.14 of \cite{sylvan1}]\label{H compatible with stop}
A Hamiltonian $H: \bar{X} \to \R$ is said to be {\it compatible with the stop} $F_{0}$ if it satisfies
\begin{enumerate}[label=(\roman*)]

\item $H$ is positive;

\item The Hamiltonian vector field $X_{H}$ is tangent to $F_{0}$;

\item $d\theta(X_{H})$ is non-negative in a cylindrical neighborhood of $F_{0} \times \R_{s>0}$ in $\bar{X}$,
where $\theta$ is the angular coordinate on $\R_{s \ge 0} \times \R_{|t| \le \e}$ as a subset of the closed right half plane $\C_{\re \ge 0}$ with $\R_{s \ge 0}$ identified with the non-negative real axis.

\end{enumerate}
\end{defn}

The last condition (iii) implies that the intersection of any Hamiltonian chord of $H$ with the ray $F_{0} \times \R_{s > 0}$ is positive.
Quadratic Hamiltonians compatible with given stops are shown to always exist, by Lemma 2.14 of \cite{sylvan1}.
In fact, the same argument, with very small changes of the numbers, can be used to derive the following fact:

\begin{lem}
Linear and quadratic Hamiltonians compatible with the stop $F_{0}$ exist.
\end{lem}

\subsection{The partially wrapped Fukaya category}\label{sec: partially wrapped}

Let $X$ be a Liouville sector, which corresponds to the sutured Liouville manifold $(\bar{X}, F_{0})$.
Fix an at most countable collection $\bar{\mathbf{L}}$ of exact cylindrical Lagrangian submanifolds of $\bar{X}$, 
equipped with {\it gradings} and {\it Spin structures},
such that for every $L_{i}, L_{j} \in \bar{\mathbf{L}}$,
all Reeb chords on $\p_{\infty} \bar{X}$ from $\p_{\infty} L_{i}$ to $\p_{\infty} L_{j}$ are non-degenerate.

Fix a positive quadratic Hamiltonian $H: \bar{X} \to \R$ that is compatible with the stop $F_{0}$.
For a pair $(L_{0}, L_{1})$ of Lagrangians in $\mathbf{L}$, let $\chi(L_{0}, L_{1}) = \chi(L_{0}, L_{1}; H)$ be the set of all time-one chords for $H$ from $L_{0}$ to $L_{1}$.
The {\it wrapped Floer cochain space} is defined to be
\begin{equation}\label{Floer cochains}
CW^{i}(L_0, L_1; H) = \bigoplus_{x \in \chi(L_{0}, L_{1}), \deg(x) = i} |o_x|_{\K}.
\end{equation}

Let $L_{0}, \ldots, L_{k} \in \mathbf{L}$. 
Let $\bar{\mathcal{R}}^{k+1}$ denote the Deligne-Mumford compactification of the moduli space of disks with $k+1$ boundary points,
and $\bar{\mathcal{S}}^{k+1} \to \bar{\mathcal{R}}^{k+1}$ the universal family.
The fiber over $[\Sigma]$ is a nodal stable disk $\Sigma$. 
For any $i=0,1,2$, let $\Omega^{i} \to \bar{\mathcal{R}}^{k+1}$ be the fibration with fiber at $[\Sigma]$ being the space of $i$-forms $\Omega^{i}(\Sigma)$.
Make universal choices of {\it Floer data}, which consist of:
\begin{enumerate}[label=(\roman*)]

\item compatible families of universal strip-like coordinates 
\begin{equation}
\epsilon^{+}_{L_{0},\ldots,L_{k}; j}: [0, +\infty) \times [0, 1] \times \bar{\mathcal{R}}^{k+1} \to \bar{\mathcal{S}}^{k+1},
\end{equation}
\begin{equation}
\epsilon^{-}_{L_{0},\ldots,L_{k};0}: (-\infty, 0] \times [0, 1] \times \bar{\mathcal{R}}^{k+1} \to \bar{\mathcal{S}}^{k+1},
\end{equation}
equipped with weights $\nu_{j}: \bar{\mathcal{R}}^{k+1} \to [1, +\infty)$ satisfying
\begin{equation}\label{stokes}
\sum_{j=1}^{k} \nu_{j} \le \nu_{0}
\end{equation}

\item compatible families of rescaling functions
\begin{equation}
\rho \in \Gamma(\bar{\mathcal{R}}^{k+1}, \Omega^{0}_{\ge 1}),
\end{equation}
which equals $\nu_{j}$ in the strip-like coordinates.
Here $\Omega^{0}_{\ge 1}$ means the space of fiberwise $\R$-valued functions with values $\ge 1$.

\item compatible families of sub-closed one-forms, 
\begin{equation}
\beta \in \Gamma(\bar{\mathcal{R}}^{k+1}, \Omega^{1}),
\end{equation}
Here sub-closedness means that $d\b_{[\Sigma]} \le 0$ on $\Sigma$ and $\b_{[\Sigma]}|_{\p \Sigma} = 0$.

\item compatible families of Hamiltonians
\begin{equation}
H: \bar{\mathcal{S}}^{k+1} \to \mathcal{H}(\bar{X})
\end{equation}
which agrees with $H$ up to a conformal rescaling factor near each of the chosen universal strip-like coordinates.

\item compatible families of cylindrical almost complex structures
\begin{equation}
J_{L_{0},\ldots,L_{k}}: \bar{\mathcal{S}}^{k+1} \to \mathcal{J}(\bar{X})
\end{equation}
which makes the projection $\pi: F \times \C_{\re \le \e} \to \C_{\re \le \e}$ holomorphic.

\end{enumerate}
In all the above choices, compatibility is defined with respect to gluing of various boundary strata of $\bar{\mathcal{R}}^{k+1}$ and $\bar{\mathcal{S}}^{k+1}$.

Let $S$ be a disk with $k+1$ distinct boundary points $z_{0}, \ldots, z_{k}$ removed.
Denote by $\p_{j} S$ the boundary component between $z_{j}$ and $z_{j+1}$, where by convention $z_{k+1}=z_{0}$.
Consider $\mathcal{R}^{k+1}(\mathbf{x})$ the moduli space of maps $u: S \to X$ satisfying
\begin{equation}\label{CR equation}
\begin{cases}
&(du - X_{H_{S}} \otimes \beta_{S})^{0, 1} = 0, \\
&u(z) \in (\psi^{\rho_{S}(z)})^{*} L_{j} \text{ for } z \in \p_{j} S, j = 0, \ldots, k, \\
&\lim\limits_{s \to + \infty} u \circ \epsilon_{j}(s, \cdot) = (\psi^{\nu_{j}})^{*} x_{j}, j = 1, \ldots, k,\\
&\lim\limits_{s \to -\infty} u \circ \e_{0}(s, \cdot) = (\psi^{\nu_{0}})^{*} x_{0}.
\end{cases}
\end{equation}

Define a map of degree $2-k$
\begin{equation}\label{mu for fully wrapped}
\begin{split}
\mu^{k}_{\w(\bar{X})}:  &CW^{*}(L_{k-1}, L_{k}; H) \otimes \cdots \otimes CW^{*}(L_{0}, L_{1}; H) \to CW^{*}(L_{0}, L_{k}; H) [2-k]\\
\mu^{k}_{\w(\bar{X})}([x_{k}] \otimes \cdots \otimes [x_{1}]) & = \sum_{\substack{x_{0}\\\deg(x_{0}) = \deg(x_{1}) + \cdots + \deg(x_{k}) + 2 - k}} 
\sum_{u \in \mathcal{R}^{k+1}(\mathbf{x})} (-1)^{*_{k}} \mathcal{R}^{k+1}_{u}([x_{k}] \otimes \cdots \otimes [x_{1}]),
\end{split}\end{equation}
where
\begin{equation}
*_{k} = \sum_{j=1}^{k} j \deg(x_{j}).
\end{equation}
It is by now a standard argument to show that $\mu^{k}$ satisfy the $\ainf$-relations.

\begin{defn}
Fix an at most countable collection $\bar{\mathbf{L}}$ of exact cylindrical Lagrangians of $\bar{X}$,
equipped with gradings and spin structures.
The (fully) wrapped Fukaya category $\w(\mathbf{L})$ is the $\ainf$-category with objects in $\mathbf{L}$,
morphism spaces being wrapped Floer cochain spaces
\[
CW^{*}(L_{0}, L_{1}; H),
\]
and $\ainf$-structure maps given by $\mu^{k}_{\w(\bar{X})}$ \eqref{mu for fully wrapped}.
When no confusion may occur, we refer to this as the wrapped Fukaya category of $\bar{X}$, 
and denote it by $\w(\bar{X})$.
\end{defn}

To define the partially wrapped Fukaya category, we shall now assume that the Hamiltonian is compatible with the stop $F_{0}$.
Take a sub-collection $\mathbf{L} \subset \bar{\mathbf{L}}$ of Lagrangians that are disjoint from $F_{0}$ at infinity.
For each $L_{0}, L_{1} \in \mathbf{L}$, observe by the condition (iii) of Definition \ref{H compatible with stop} that every possible intersection of a time-one $H$-chord $x$ from $L_{0}$ to $L_{1}$ with $F_{0} \times \R_{s>0}$ in $\bar{X}$ is positive.
Write $x \cdot (F_{0} \times \R_{s>0})$ for the intersection number.
We can consider the following subspace:

\begin{defn}\label{defn: partially wrapped CW}
The partially wrapped Floer cochain space is defined to be the subspace of $CW^{*}(L_{0}, L_{1}; H)$ generated by time-one $H$-chords which do not intersect $F_{0} \times \R_{s>0}$, 
\begin{equation}\label{partially wrapped CW}
CW^{*}_{F}(L_{0}, L_{1}; H) := \bigoplus_{\substack{x \in \chi(L_{0}, L_{1}; H)\\ x \cdot (F_{0} \times \R_{s>0})=0}} |o_{x}|_{\K}.
\end{equation}
\end{defn}

The Floer differential, and more generally all the $\ainf$-structure maps $\mu^{k}$ do not increase the number of intersection points of chords with $F_{0} \times \R_{s>0}$ (see \cite{sylvan1}).
Therefore we get well-defined induced maps on these subspaces
\begin{equation}\label{mu for partially wrapped}
\mu^{k}_{\w(X)}: CW^{*}_{F}(L_{k-1}, L_{k}; H) \otimes \cdots \otimes CW^{*}_{F}(L_{0}, L_{1}; H) \to CW^{*}_{F}(L_{0}, L_{k}; H) [2-k].
\end{equation}

\begin{defn}
Fix an at most countable collection $\mathbf{L}$ of exact cylindrical Lagrangians of $\bar{X}$ disjoint from $F_{0}$ at infinity, 
equipped with gradings and spin structures.
The partially wrapped Fukaya category $\w(\mathbf{L})$ is the $\ainf$-category with objects being Lagrangians in $\mathbf{L}$, 
morphisms being
\[
\hom_{\w(X)}(L_{0}, L_{1}) = \hom_{\w(\bar{X}, F_{0})}(L_{0}, L_{1}) = CW^{*}_{F}(L_{0}, L_{1}; H),
\]
and $\ainf$-structure maps given by $\mu^{k}_{\w(X)}$ \eqref{mu for partially wrapped}.
When no confusion may occur, we refer to this as the partially wrapped Fukaya category $\w(X, \bar{F}_{0})$ of the sutured Liouville manifold $(X, F_{0})$,
or $\w(X)$ of the Liouville sector (if the Lagrangians are indeed contained in the interior of $X$).
\end{defn}

By construction, there is a canonical {\it acceleration functor}
\begin{equation}
i_{*, F}: \w(X) = \w(\bar{X}, F_{0}) \to \w(\bar{X}),
\end{equation}
which is identity on objects,
such that the linear term $i_{*, F}^{1}$ is the inclusion $CW^{*}_{F}(L_{0}, L_{1}; H) \subset CW^{*}(L_{0}, L_{1}; H)$ on morphism spaces,
and all the higher order terms $i_{*, F}^{k} = 0$ for $k \ge 2$.
This functor is a special case of the more general pushforward functors associated to inclusions of Liouville sectors introduced in \cite{GPS1}; see \S\ref{section: pushforward} for more discussion about pushforward functors.

\begin{rem}
In \cite{GPS2}, the category $\w(\bar{X}, F_{0})$ is slightly larger than $\w(X)$, 
as it possibly contains more Lagrangians.
However, there are natural pushforward functors 
\begin{equation}
\w(X) \stackrel{\sim}\to \w(\bar{X}, \p_{\infty}(F \times \C_{\re \ge 0})) \stackrel{\sim}\to \w(\bar{X}, F_{0}) \stackrel{\sim}\to \w(\bar{X}, \fk)
\end{equation}
associated to the inclusions $X \xhookrightarrow{} (\bar{X}, F_{0}) \xhookrightarrow{} (\bar{X}, \fk)$, 
which are all quasi-equivalences by deformation invariance.
\end{rem}

Thus, we may understand any of $\w(X), \w(\bar{X}, F_{0}), \w(\bar{X}, F_{0})$ and $\w(\bar{X}, \fk)$ as the same Fukaya category,
and refer to either $F_{0}$ or $\fk$ as the stop for the Liouville manifold $\bar{X}$ associated to the Liouville sector $X$.
For a Liouville sector, the definition of the partially wrapped Fukaya category is equivalent to the one presented in this paper (a reference for this equivalence can be found e.g. in \cite{sylvan2}).

\subsection{Some structural results}

This subsection contains a brief summary of several structural results about partially wrapped Fukaya categories that are needed for relevant discussions in this paper, mostly due to \cite{GPS1}, \cite{GPS2}.
The most fundamental result underlying many various structural results of the partially wrapped Fukaya category of a stopped Liouville sector is an exact triangle describing the effect of wrapping through the stop.

\begin{thm}[Theorem 1.10 of \cite{GPS2}]\label{thm: wrapping exact triangle}
Let $(X, \fk)$ be a stopped Liouville sector, and $p \in \fk$ be a point near which $\fk$ is a Legendrian submanifold.
Suppose $L \subset X$ is a Lagriangian submanifold and $L^{w}$ is obtained from $L$ by a positive isotopy by passing $\p_{\infty} L$ through $f$ transversely at $p$ once.
Let $D_{p}$ be the small Lagrangian disk linking $\fk$ at $p$.
Then there is an exact triangle 
\begin{equation}
L^{w} \to L \to D_{p} \stackrel{[1]}\to
\end{equation}
in $\w(X, \fk)$, where $L^{w} \to L$ is the continuation map induced by the isotopy $L \rightsquigarrow L^{w}$.
\end{thm}

The existence of smooth Legendrian points on the stop $\fk$ is important in many other results based on the wrapping exact triangle Theorem \ref{thm: wrapping exact triangle}.
Because of this, it is natural to make the following definition, following \cite{GPS2}.

\begin{defn}\label{defn: mostly legendrian}
A closed subset $\fk$ of a contact manifold $Y$ of dimension $2n-1$ is called mostly Legendrian if it admits a decomposition $\fk = \fk^{subcrit} \cup \fk^{crit}$,
where $\fk^{subcrit}$ is closed and contained in the image of a manifold of dimension $< n-1$ under a smooth map,
and $\fk^{crit} \subset Y \setminus \fk^{subcrit}$ is a Legendrian submanifold.
\end{defn}

Then there is a generation result for partially wrapped Fukaya categories of stopped Weinstein manifolds, whose proof crucially relies on Theorem \ref{thm: wrapping exact triangle}.

\begin{thm}[Theorem 1.10 of \cite{GPS2}]
Suppose $(\bar{X}, \fk)$ is a stopped Weinstein manifold with the stop mostly Legendrian, $\fk = \fk^{subcrit} \cup \fk^{crit}$.
Suppose the cocores of the critical handles of $X$ are properly embedded and disjoint from $\fk$ at infinity.
Then the partially wrapped Fukaya category $\w(\bar{X}, \fk)$ is generated by these cocores and the small Lagrangian disks linking $\fk^{crit}$.
\end{thm}

A special case where the stop is mostly Legendrian is given by a Lefschetz fibration $\pi: \bar{X} \to \C$ with Weinstein fibers $F$.
In that case, the Fukaya-Seidel category $\mathcal{FS}(\bar{X}, \pi)$,
or equivalently the partially wrapped Fukaya category $\w(\bar{X}, \fk)$, 
where $\fk$ is the core of the fiber $F$ at $-\infty$,
is generated by Lefschetz thimbles. 
This follows from another generation result (Theorem 1.16) of \cite{GPS2}, 
for a stopped Liouville sector of the form $(F \times \C_{\re \ge 0}, \Lambda)$,
where $F$ is a Liouville manifold, and $\Lambda \subset \p_{\infty}(F \times \C_{\re \ge 0})^{\circ}$ is mostly Legendrian.

Another key ingredient for us is the stop removal formula.

\begin{thm}[Theorem 1.20 of \cite{GPS2}]\label{thm: stop removal}
Suppose $\mathfrak{g} \subset \fk$ are two stops for $X$ such that $\fk \setminus \mathfrak{g}$ is mostly Legendrian.
Let $\mathcal{D}(\fk \setminus \mathfrak{g})$ denote the subcategory of $\w(X, \fk)$ whose objects are small Lagrangian disks linking $(\fk \setminus \mathfrak{g})^{crit}$.
The pushforward induces a quasi-equivalence
\begin{equation}
\w(X, \fk) / \mathcal{D}(\fk \setminus \mathfrak{g}) \to \w(X, \mathfrak{g}).
\end{equation}
\end{thm}

The proof of Theorem \ref{thm: stop removal} also relies on the wrapping exact triangle (Theorem \ref{thm: wrapping exact triangle}).
If $(X, \fk)$ is the stopped Liouville manifold for a Lefschetz fibration $\pi: \bar{X} \to \C$ with Weinstein fibers $F$,
the above result applied to $\mathfrak{g} = \varnothing$ implies that there is a quasi-equivalence
\[
\w(\bar{X}, \fk) / \mathcal{D}(\fk) \stackrel{\sim}\to \w(\bar{X}).
\]
The special case for a localization relation between the Fukaya-Seidel category and the fully wrapped Fukaya category of the total space was earlier proved by Abouzaid-Seidel \cite{abouzaidseidel2} and Sylvan \cite{sylvan1},
in somewhat different forms.

Lastly, we shall need the K\"{u}nneth formula for partially wrapped Fukaya categories.

\begin{thm}[Theorem 1.5 of \cite{GPS2}]
For Liouville sectors $X$ and $Y$, there is a fully faithful $\ainf$-bilinear-functor
\begin{equation}\label{kunneth}
\w(X) \times \w(Y) \to \w(X \times Y).
\end{equation}
\end{thm}

As an immediate consequence, K\"{u}nneth formula shows that there is a stabilization operation in the partially wrapped content,
in the sense that the K\"{u}nneth bilinear functor induces a fully faithful embedding
\begin{equation}\label{kunneth stab}
\w(X) \xhookrightarrow{} \w(X \times T^{*}[0, 1]),
\end{equation}
called the K\"{u}nneth stabilization functor,
after fixing a choice of a functor $\K \to \w(T^{*}[0, 1])$
(where $\K$ stands for the $\ainf$-category with a single object whose morphism space is $\K$ with the obvious $\ainf$-structure,)
sending the object to one cotangent fiber $T^{*}_{p}$ for some fixed choice of $p \in (0, 1)$.
On objects, an exact cylindrical Lagrangian $L \subset X$ is sent to the canonical cylindrization $L \tilde{\times} T^{*}_{p}$ of the product $L \times T^{*}_{p}$ of $L$ with a cotangent fiber of $T^{*}[0, 1]$.

\subsection{The Rabinowitz Fukaya category}

The new variant of Fukaya category for a Liouville sector $X$  to be defined in the next subsection will be closely related to the {\it Rabinowitz Fukaya category} of its convex completion $\bar{X}$.
We briefly recall the relevant definitions here following \cite{GGV} in order to fix notations needed for various Floer-theoretic constructions that will appear throughout the remaining subsections \S\ref{section: negatively wrapped Fukaya category}-\S\ref{section: nonessential Lagrangians}.

Take a positive quadratic Hamiltonian $H: \bar{X} \to \R$.
Consider the negative Hamiltonian $-H$, and define the Floer cochain space for $(L_{0}, L_{1})$ with respect to $-H$ as
\begin{equation}\label{Floer chains}
CW^{i}(L_{0}, L_{1}; -H) = \prod_{\substack{x^{-} \in \chi(L_{0}, L_{1}; -H) \\ \deg(x^{-}) = i}} |o_{x^{-}}|_{\K} = \prod_{\substack{x \in \chi(L_{1}, L_{0}; H)\\ \deg(x) = n-i}} |o^{-}_{x}|_{\K}.
\end{equation}
The second equality follows from a natural bijective correspondence between time-one $(-H)$-chords $x^{-}$ from $L_{0}$ to $L_{1}$ and time-one $H$-chords $x$ from $L_{1}$ to $L_{0}$, by $x^{-}(t) = x(1-t)$,
which comes with an associated isomorphism of orientation lines $o_{x^{-}} \cong o_{x}^{-}$ as thus an isomorphism of normalized orientation spaces
\begin{align}
|o_{x^{-}}|_{\K} & \to |o_{x}^{-}|_{\K} \\
[x^{-}] & \mapsto [x]^{-},
\end{align}
where $[x]^{-}$ is the negative orientation for $x$.
The Floer differential can be defined by counting inhomogeneous pseudoholomorphic strips in a similar way to the case of defining the Floer differential on $CW^{*}(L_{0}, L_{1}; H)$,
except that we no longer get a finite sum; and that is why we need the direct product in \eqref{Floer chains}.
The Floer differential is still well-defined on the whole direct product by an action argument,
where the direct product can be interpreted as the completion with respect to a filtration on the direct sum given by putting upper bounds on the action.
Since our Lagrangians are oriented, we can make use of the Poincar\'{e} duality isomorphism
\begin{equation}\label{PD}
\bar{I}: CW^{*}(L_{0}, L_{1}; -H) \stackrel{\cong}\to \hom_{\K}(CW^{n-*}(L_{1}, L_{0}; H), \K),
\end{equation}
to give an alternative definition of the Floer complex for the negative quadratic Hamiltonian $-H$. 

There is a {\it continuation map}
\begin{equation}\label{continuation}
c: CW^{*}(L_{0}, L_{1}; -H) \to CW^{*}(L_{0}, L_{1}; H),
\end{equation}
which can be thought of as a continuation map induced by a monotone increasing homotopy from $-H$ to $H$.
The well-definedness of the continuation map \eqref{continuation} on the direct product \eqref{Floer chains} follows from an action filtration argument similar to that for the well-definedness of the differential. 
Alternatively, the map \eqref{continuation} can be defined as the dual of a {\it copairing} for the wrapped Floer cochain complexes with respect to the positive Hamiltonian $H$.
If the copairing is
\begin{equation}\label{copairing}
\hat{c}: \K \to CW^{n-*}(L_{1}, L_{0}; H) \otimes CW^{*}(L_{0}, L_{1}; H),
\end{equation}
then we set
\begin{equation}
c(\cdot) = (\bar{I}(\cdot) \otimes \id)(\hat{c}).
\end{equation}
The copairing $\hat{c}$ is defined by counting rigid inhomogeneous pseudoholomorphic disks with no input and two outputs,
one of which is a time-one $H$-chord from $L_{0}$ to $L_{1}$ and the other a time-one $H$-chord from $L_{1}$ to $L_{0}$;
for energy reasons such a count must be a finite, so a map of the form \eqref{copairing} is well-defined.
Details can be found in \cite{GGV}.

The {\it Rabinowitz Floer complex} is defined to be the mapping cone of the continuation map \eqref{continuation}, 
\begin{equation}\label{Rabinowitz complex}
\begin{split}
RC^{*}(L_{0}, L_{1}) = & \cone(c: CW^{*}(L_{0}, L_{1}; -H) \to CW^{*}(L_{0}, L_{1}; H)) \\
& = CW^{*}(L_{0}, L_{1}; -H)[1] \oplus CW^{*}(L_{0}, L_{1}; H).
\end{split}
\end{equation}
Furthermore, an $\ainf$-structure can be constructed on the Rabinowitz Floer complexes by counting certain kinds of popsicles, adapted to the quadratic Hamiltonian framework.
Below we give a brief summary of the construction of the $\ainf$-structure, and modify it to define an $\ainf$-structure on for a corresponding chain complex in the partially wrapped setting.

Let $S$ be a disk with $k+1$ boundary punctures $z_{0}, \ldots, z_{k}$. 
For $1 \le j \le k$, let $C_{j}$ be the unique hyperbolic geodesic connecting the points at infinity $z_{0}$ and $z_{j}$,
called {\it popsicle sticks}.
Recall from \cite{abouzaidseidel} that a popsicle structure on $S$ of {\it flavor} $\mathbf{p}: F \to \{1, \ldots, k\}$,
where $F$ is a finite index set,
consists of a choice of a preferred point $\sigma_{f} \in C_{\mathbf{p}(f)}$,
called a {\it sprinkle}.
Assign a non-positive integer weight $w_{i} \in \Z_{\le 0}$ to each puncture $z_{i}$ satisfying
\begin{equation}\label{weight sum condition}
w_{0} = \sum_{i=1}^{k} (w_{i} + |\mathbf{p}^{-1}(i)|),
\end{equation}
and
\begin{equation}\label{sprinkles bounded by weight}
|\mathbf{p}^{-1}(i)| \le -w_{i}, i = 0, 1, \ldots, k.
\end{equation}
We write
\begin{equation}
\mathbf{w} = (w_{0}, \ldots, w_{k})
\end{equation}
for the collection of weights.

Denote by $\mathcal{R}^{k+1, \mathbf{p}, \mathbf{w}}$ the moduli space of popsicles of flavor $\mathbf{p}$ carrying weights $\mathbf{w}$,
and $\mathcal{S}^{k+1, \mathbf{p}, \mathbf{w}} \to \mathcal{R}^{k+1, \mathbf{p}, \mathbf{w}}$ the universal family. 
Each $\mathcal{R}^{k+1, \mathbf{p}, \mathbf{w}}$ is a copy of $\mathcal{R}^{k+1, \mathbf{p}}$, i.e. independent of $\mathbf{w}$,
and is a fiber bundle over $\mathcal{R}^{k+1}$ with fiber $\R^{|F|}$.
The compactifications, by adding broken popsicles, are denoted by $\bar{\mathcal{S}}^{k+1, \mathbf{p}, \mathbf{w}} \to \bar{\mathcal{R}}^{k+1, \mathbf{p}, \mathbf{w}}$,
whose low dimensional boundary strata are products of moduli spaces of popsicles of other flavors.
For the purpose of counting popsicles to define the operations \eqref{structure maps for n}, we shall only restrict ourselves to injective popsicles, which are the ones with trivial $Aut(\mathbf{p})$,
or equivalently $|\mathbf{p}^{-1}(i)| \le 1$ for every $i = 1, \ldots, k$.
In addition, we making the following constraints on the weights
\begin{equation}\label{weights -1 and 0}
w_{j} \in \{-1, 0\}, \forall j = 0, 1, \ldots, k.
\end{equation}
Introduce a sign notation for every $j=0,\ldots,k$,
\begin{equation}\label{sign notation}
s_{j} = 
\begin{cases}
+, & \text{ if } w_{j} = 0, j>0 \text{ or } w_{j} = -1, j = 0\\
-, & \text{ if } w_{j} = -1, j>0 \text{ or } w_{j} = 0, j = 0.
\end{cases}
\end{equation}

Equip $S$ with strip-like ends $\{\epsilon_j\}_{j\in\{0, \dots, k\}}$ determined by the weights $\mathbf{w}$ in the following way. 
Let $Z^{\pm} = \R^{\pm} \times [0, 1]$,
 with standard coordinates $(s, t)$.
For $j=0,\ldots, k$, the $j$-th strip-like end near $z_{j}$ is a holomorphic embedding
\begin{equation}\label{strip end j}
\e_{j}: Z^{s_{j}} \to S.
\end{equation}
These should satisfy $\epsilon_{j}^{-1}(\partial S) = \{t = 0, 1\}$
and $\lim\limits_{s \to \pm \infty} u \circ \e_{j}= z_j$ for all $j = 0, 1, \ldots, k$, and have disjoint images.
A puncture $z_{j}$ is called a negative (resp. positive) puncture if the strip-like end has domain $Z^{-}$ (resp. $Z^{+}$).

Fix a $(k+1)$-tuple of Lagrangians $L_{0}, \ldots, L_{k} \in \mathbf{L}$, 
and assign $L_{j}$ to the $j$-th boundary component of $S$ between the punctures $z_{j}$ and $z_{j+1}$.
Denote by $\mathbf{x} = (x_{0}, \ldots, x_{k})$ a collection of time-one chords for $s_{j}H$,
where $x_{0} \in \chi(L_{0}, L_{d}; s_{0}H)$ and $x_{j} \in \chi(L_{j-1}, L_{j}; s_{j}H)$ for $j=1,\ldots,k$.

Because of the changes we have made to the choices to the strip-like ends for a popsicle compared to those for an ordinary boundary-pointed disk, 
we also need to change the conditions for a Floer datum.
In fact, the only condition we need to change is \eqref{stokes}, 
which now takes the form
\begin{equation}
\sum_{j=1}^{k} (-1)^{w_{j}} \nu_{j} \le (-1)^{w_{0}} \nu_{0}.
\end{equation}

Suppose we have made universal choices of Floer data for weighted popsicles.
Write down the inhomogeneous Cauchy-Riemann equation in the same form as \eqref{CR equation},
but keep in mind that Floer data and asymptotic conditions should be changed accordingly. 
The new equation is
\begin{equation}\label{CR equation for rw}
\begin{cases}
&(du - X_{H_{S}} \otimes \beta_{S})^{0, 1} = 0, \\
&u(z) \in (\psi^{\rho_{S}(z)})^{*} L_{j} \text{ for } z \in \p_{j} S, j = 0, \ldots, k, \\
&\lim\limits_{s \to s_{j} \infty} u \circ \epsilon_{j}(s, \cdot) = (\psi^{\nu_{j}})^{*} x_{j}^{s_{j}}, j = 1, \ldots, k, \\
&\lim\limits_{s \to -s_{0} \infty} u \circ \epsilon_{0}(s, \cdot) = (\psi^{\nu_{0}})^{*} x_{0}^{s_{0}}.
\end{cases}
\end{equation}
Let 
\begin{equation}\label{moduli space of popsicle maps}
\mathcal{R}^{k+1, \mathbf{p}, \mathbf{w}}(\mathbf{x})
\end{equation}
 the moduli space of such popsicle maps satisfying \eqref{CR equation} with the prescribed boundary conditions and asymptotic chords $x_{0}, \ldots, x_{k}$.
The moduli space has virtual dimension
\begin{equation}
\dim \mathcal{R}^{k+1, \mathbf{p}, \mathbf{w}}(\mathbf{x}) = k - 2 + |F| + n(-w_{0}+\sum_{j=1}^{k}w_{j}) + (-1)^{w_{0}} \deg(x_{0}) - \sum_{j=1}^{k} (-1)^{w_{j}} \deg(x_{j}) 
\end{equation}
Counting rigid elements in moduli spaces of dimension zero yields multilinear maps of degree $2-k+|F|$:
\begin{equation}\label{components of structure maps for rw}
\mu^{k, \mathbf{p}, \mathbf{w}}:  CW^{*}(L_{k-1}, L_{k}; s_{k}H) \otimes \cdots \otimes CW^{*}(L_{0}, L_{1}; s_{1}H) \to CW^{*}(L_{0}, L_{k}; s_{0}H) 
\end{equation}
which on basis elements take the form
\begin{equation}\label{muk on basis}
\begin{split}
&\mu^{k, \mathbf{p}, \mathbf{w}}  ([x_{k}^{s_{k}}] \otimes \cdots \otimes [x_{1}^{s_{1}}])  = \prod_{\substack{x_{0}\\ k - 2 + |F| + n(-w_{0}+\sum_{j=1}^{k}w_{j}) + (-1)^{w_{0}} \deg(x_{0}) - \sum_{j=1}^{k} (-1)^{w_{j}} \deg(x_{j})=0}} \\
& \sum_{u \in \mathcal{R}^{k+1, \mathbf{p}, \mathbf{w}}(\mathbf{x})} 
(-1)^{*_{k, \mathbf{p}, \mathbf{w}} + \Diamond_{k, \mathbf{p}, \mathbf{w}}} \mathcal{R}^{k+1, \mathbf{p}, \mathbf{w}}_{u} ([x_{k}^{s_{k}}] \otimes \cdots \otimes [x_{1}^{s_{1}}]),
\end{split}
\end{equation}
Here the signs $*_{k, \mathbf{p}, \mathbf{w}}$ are
\begin{equation}\label{sign formula 1 for popsicles}
*_{k, \mathbf{p}, \mathbf{w}} = \sum_{j=1}^{k}( j + w_{1} + \cdots + w_{j-1}) \deg(x_{j}^{s_{j}}) + \sum_{j=1}^{k} (k-j) w_{j},
\end{equation}
\begin{equation}\label{sign formula 2 for popsicles}
\Diamond_{k, \mathbf{p}, \mathbf{w}} = \sum_{j=1}^{k} | \mathbf{p}^{-1}(\{ j+1, \ldots, k\})| (w_{j} + |\mathbf{p}^{-1}(j)|).
\end{equation}
Whenever there are some $j$'s such that $w_{j} = -1$ so the wrapped Floer complex $CW^{*}(L_{j-1}, L_{j}; -H)$ are direct products \eqref{Floer chains}, 
which can be interpreted as the completion of the direct sum with respect to a filtration by putting upper bounds on the action,
the operations $\mu^{k, \mathbf{p}, \mathbf{w}}$ \eqref{muk on basis} can be extended to a well-defined map on the completion by an action filtration argument, 
because $\mu^{k, \mathbf{p}, \mathbf{w}}$ increases the action.
Then we combine all these maps \eqref{components of structure maps for rw} to define 
\begin{equation}\label{structure maps for rw}
\mu^{k}_{\rw}: RC^{*}(L_{k-1}, L_{k}) \otimes \cdots \otimes RC^{*}(L_{0}, L_{1}) \to RC^{*}(L_{0}, L_{k}),
\end{equation}
which on a given component takes the form \eqref{components of structure maps for rw}.
This has degree $2-k$ for the reason of the degree shift in \eqref{Rabinowitz complex}.

\begin{defn}
Fix an at most countable collection $\bar{\mathbb{L}}$ of exact cylindrical Lagrangians in the Liouville manifold $\bar{X}$ which contains all isotopy classes.
The Rabinowitz wrapped Fukaya category $\rw(\mathbb{L})$ is defined to be the $\ainf$-category with objects being exact Lagrangian submanifolds in $\mathbb{L}$,
morphism space being the Rabinowitz Floer complex
\[
\hom_{\rw}(L_{0}, L_{1}) = RC^{*}(L_{0}, L_{1})
\]
and $\ainf$-structure maps given by $\mu^{k}_{\rw}$ \eqref{structure maps for rw}.
When no confusion may occur, we refer to this as the Rabinowitz wrapped Fukaya category of $\bar{X}$, and denote it by $\rw(\bar{X})$.
\end{defn}

\begin{rem}\label{rem on chords}
When defining the complex $RC^{*}(L_{0}, L_{1})$ as in \eqref{Rabinowitz complex},
 we interpret the generators for the $CW^{*}(L_{0}, L_{1}; -H)[1]$ component as time-one $(-H)$-chords $x^{-}$ from $L_{0}$ to $L_{1}$.
 However, when defining the $\ainf$-structure maps $\mu^{k}_{\rw}$ using the various moduli spaces of popsicle maps, 
 we use \eqref{Floer chains} to identify them with time-one $H$-chords $x$ from $L_{1}$ to $L_{0}$.
 But the roles of inputs and outputs are interchanged, as \eqref{Floer chains} is a strict identification of the orientation line of $x$ with the negative orientation line of $x$.
 
 Later, when we discuss about popsicle maps, the actual asymptotic chords are all $H$-chords, 
 but keep in mind that an output $H$-chord at $j$-th puncture for $j \neq 0$ should be formally understood as an input $(-H)$-chord.
 But when we only talk about chain complexes, it is more convenient to keep the generators as $(-H)$-chords.
\end{rem}

\subsection{The negatively wrapped Fukaya category}\label{section: negatively wrapped Fukaya category}

Now let us turn our attention to a Liouville sector $X$, whose corresponding sutured Liouville manifold is $(\bar{X}, F_{0})$.
The goal of this subsection is to introduce a Floer-theoretic invariant for $X$ or equivalently $(\bar{X}, F_{0})$ that is analogous to the the Rabinowitz Fukaya category.
This construction does not seem to carry over to general stopped Liouville sectors, contrasting the partially wrapped Fukaya category (\cite{GPS2}).
The key observation leading to the well-definedness is the following:

\begin{lem}\label{image of continuation}
Suppose $H: \bar{X} \to \R$ is a positive quadratic Hamiltonian which is compatible with the stop $F$.
Then the image of the continuation map $c$ \eqref{continuation} is contained in the subspace $CW^{*}_{F}(L_{0}, F_{1}; H)$.
\end{lem}
\begin{proof}
Since the continuation map is defined as a copairing by counting certain disks with two outputs and no inputs,
we conclude by a standard homotopy argument that it factorizes through a map
\begin{equation}
c_{h}: CF^{*}(L_{0}, L_{1}; -h) \to CF^{*}(L_{0}, L_{1}; h)
\end{equation}
for a small positive Hamiltonian $h$ linear at infinity, in the sense that there is a commutative diagram:
\begin{equation}
\begin{tikzcd}
CW^{*}(L_{0}, L_{1}; -H) \arrow[r, "c"] \arrow[d] & CW^{*}(L_{0}, L_{1}; H)\\
CF^{*}(L_{0}, L_{1}; -h) \arrow[r, "c_{h}"] &CF^{*}(L_{0}, L_{1}; h) \arrow[u]
\end{tikzcd}
\end{equation}
where the vertical maps are continuation maps induced by an increasing homotopy from $h$ to $H$ and an increasing homotopy from $-H$ to $-h$.
For the homotopy from $-H$ to $-h$, we define the continuation map in a similar way as \eqref{continuation} from the copairing \eqref{copairing}.
We may choose $h$ carefully so that it is also compatible with the stop in the sense of Definition \ref{H compatible with stop}.

If the slope of the linear Hamiltonian $h$ is sufficiently small, say less than the length of the shortest Reeb chord from $\p_{\infty}L_{0}$ to $\p_{\infty}L_{1}$ on $\p_{\infty}X$,
the continuation images of any $h$-chord under the continuation map $CF^{*}(L_{0}, L_{1}; h) \to CW^{*}(L_{0}, L_{1}; H)$ do not correspond to any Reeb chord on $\p_{\infty} X$ from $\p_{\infty} L_{0}$ to $\p_{\infty} L_{1}$,
so must be some $H$-chords in the interior of the compact Liouville domain $\bar{X}_{0}$,
which never cross $F_{0} \times \R_{s>0}$.
It follows that the image of $CF^{*}(L_{0}, L_{1}; h) \to CW^{*}(L_{0}, L_{1}; H)$ is contained in the subspace $CW^{*}_{F}(L_{0}, L_{1}; H)$.
\end{proof}

The proof of Lemma \ref{image of continuation} in fact provides a more refined description of the image of the continuation map.
For this description, we introduce 
\begin{equation}\label{subcomplex of interior chords}
CW^{*}_{int}(L_{0}, L_{1}; H) \subset CW^{*}(L_{0}, L_{1}; H)
\end{equation}
the subspace generated by `interior chords', i.e. the ones contained in a compact set in the Liouville domain $\bar{X}_{0}$ which do not correspond to any Reeb chord from $\p_{\infty} L_{0}$ to $\p_{\infty} L_{1}$.
Without loss of generality, we may choose $H$ small enough in $C^{2}$-norm in the compact domain part, 
such that all these interior chords are trivial, i.e. intersection points, and have zero action.
Therefore, the subspace $CW^{*}_{int}(L_{0}, L_{1}; H)$ is a subcomplex.
By definition, this subcomplex is also contained in $CW^{*}_{F}(L_{0}, L_{1}; H)$, which is generated by all these interior chords,
 as well as those Hamiltonian chords in the cylindrical end corresponding to some Reeb chords but not crossing $F_{0} \times \R_{s>0}$.
Note that the wrapped Floer complex $CW^{*}(L_{0}, L_{1}; -H)$ also contains a subspace $CW^{*}_{int}(L_{0}, L_{1}; -H)$ generated by interior chords (which may adjusted to be intersection points). 
This is no longer a subcomplex, but is naturally identified with a quotient complex, 
together with a projection map $CW^{*}(L_{0}, L_{1}; -H) \to CW^{*}_{int}(L_{0}, L_{1}; -H)$ which is a chain map.
The more refined statement is

\begin{cor}\label{continuation factors through interior chords}
The continuation map $c: CW^{*}(L_{0}, L_{1}; -H) \to CW^{*}(L_{0}, L_{1}; H)$ \eqref{continuation} factors as 
\[
CW^{*}(L_{0}, L_{1}; -H) \to CW^{*}_{int}(L_{0}, L_{1}; -H) \to CW^{*}_{int}(L_{0}, L_{1}; H) \subset CW^{*}(L_{0}, L_{1}; H).
\]
In particular, the image of $c$ is contained in the subcomplex $CW^{*}_{int}(L_{0}, L_{1}; H)$ generated by interior chords.
\end{cor}

Lemma \ref{image of continuation} allows us to define the following cochain complex:
\begin{equation}\label{morphism in n}
\begin{split}
NC^{*}_{F}(L_{0}, L_{1}) &:= \cone(c: CW^{*}(L_{0}, L_{1}; -H) \to CW^{*}_{F}(L_{0}, L_{1}; H)) \\
&= CW^{*}(L_{0}, L_{1}; -H)[1] \oplus CW^{*}_{F}(L_{0}, L_{1}; H)),
\end{split}
\end{equation}
with the cone differential 
\begin{equation}
\mu^{1}_{NC} = 
\begin{pmatrix}
d_{-} & c[1] \\
0 & \mu^{1}_{\w(X)}
\end{pmatrix}
\end{equation}

Now we want to define an $\ainf$-structure on the complexes \eqref{morphism in n}.
Since the continuation map \eqref{continuation} has image in $CW^{*}_{F}(L_{0}, L_{1}; H)$, there is an obvious inclusion
\begin{equation}
NC^{*}_{F}(L_{0}, L_{1}) \subset RC^{*}(L_{0}, L_{1}),
\end{equation}
which is a chain map.
In order for the maps $\mu^{k}_{\rw}$ to descend to the subcomplexes $NC^{*}_{F}(L_{0}, L_{1})$, we will need the following identity about intersection points of chords with $F_{0} \times \R_{s>0}$ as well as those of inhomogeneous pseudoholomorphic curves with $F$.

\begin{lem}\label{lem: intersection number identity}
Consider any moduli space of popsicle maps $\mathcal{R}^{k+1, \mathbf{p}, \mathbf{w}}(\mathbf{x})$, 
which is among those defining the $\ainf$-structure maps $\mu^{k}_{\rw}$ \eqref{structure maps for rw}.
Then the intersection numbers of various $H$-chords satisfy the following identity
\begin{equation}\label{intersection number identity}
(-1)^{w_{0}} x_{0} \cdot (F_{0} \times \R_{s>0}) + u \cdot F = \sum_{j=1}^{k} (-1)^{w_{j}} x_{j} \cdot (F_{0} \times \R_{s>0}).
\end{equation}
where $u$ is any representative of an element in $\mathcal{R}^{k+1, \mathbf{p}, \mathbf{w}}(\mathbf{x})$.
\end{lem}
\begin{proof}
This is just a variant of Sylvan's argument \cite{sylvan1} to show that the $\ainf$-maps $\mu^{k}_{\w(\bar{X})}$ \eqref{mu for fully wrapped} do not increase the intersection numbers of chords with $F_{0} \times \R_{s>0}$ so that $\mu^{k}_{\w(X)}$ \eqref{mu for partially wrapped} are well defined.

For convenience, we include the proof, which is quite straightforward.
The difference in the intersection numbers of all the outputs and those of all the inputs is precisely the winding number of $\p u$ around $F$, 
which by topology is equal to $u \cdot F$.
\end{proof}

\begin{prop}\label{prop: empty moduli space for positive intersection}
Suppose for $j = 1, \ldots, k$, $x_{j}$ is either a time-one $H$-chord from $L_{j}$ to $L_{j-1}$, in which case $w_{j} = -1$ and $x_{j}$ is an output,
or a time-one $H$-chord from $L_{j-1}$ to $L_{j}$ which does not intersect $F_{0} \times \R_{s>0}$, in which case $w_{j} = 0$ and $x_{j}$ is an input.
If $x_{0}$ is a time-one $H$-chord from $L_{0}$ to $L_{k}$ which intersects $F_{0} \times \R_{s>0}$, then the moduli space $\mathcal{R}^{k+1, \mathbf{p}, \mathbf{w}}(\mathbf{x})$ is empty.
\end{prop}
\begin{proof}
Here the $x_{j}$'s on the right-hand side of \eqref{intersection number identity} are all understood to be $H$-chords (see Remark \ref{rem on chords}), so they all have non-negative intersection numbers with $F_{0} \times \R_{s>0}$.
Given our assumption, the right-hand side of \eqref{intersection number identity} is always non-positive. 
On the left-hand side, the term $u \cdot F$ is always non-negative, as it is the intersection number of an inhomogeneous pseudoholomorphic curve with an almost complex submanifold $F$, where the Hamiltonian is further chosen to be compatible with $F$.
In particular, if $x_{0}$ is a time-one $H$-chord from $L_{0}$ to $L_{k}$, treated as an output, then \eqref{intersection number identity} implies that $x_{0} \cdot (F_{0} \times \R_{s>0}) \le 0$. 
Thus the moduli space is empty for positive intersection number.
\end{proof}

Proposition \ref{prop: empty moduli space for positive intersection}, or rather its proof, has the following immediate implication:

\begin{cor}
For any $k \ge 1$, the maps \eqref{structure maps for rw} induces maps
\begin{equation}\label{structure maps for n}
\mu^{k}_{\n}: NC^{*}_{F}(L_{k-1}, L_{k}) \otimes \cdots \otimes NC^{*}_{F}(L_{0}, L_{1}) \to NC^{*}_{F}(L_{0}, L_{k})
\end{equation}
which satisfy the $\ainf$-relations.
\end{cor}

\begin{defn}
Fix an at most countable collection $\mathbf{L}$ of exact cylindrical Lagrangians in $\bar{X}$ disjoint from $F_{0}$ at infinity.
The negatively wrapped Fukaya category $\n(\mathbf{L})$ is defined to be the $\ainf$-category with objects being Lagrangians in $\mathbf{L}$,
morphism spaces given by the chain complexes
\[
NC^{*}_{F}(L_{0}, L_{1}),
\]
and $\ainf$-structure maps given by $\mu^{k}_{\n}$ \eqref{structure maps for n}.

When no confusion may occur, we shall also denote this category by $\n(\bar{X}, F_{0})$ for the sutured Liouville manifold $(\bar{X}, F_{0})$,
or $\n(X)$ for the corresponding Liouville sector $X$.
\end{defn}

\subsection{Deformation invariance}

The partially wrapped Fukaya category $\w(X) = \w(\bar{X}, F_{0})$ is invariant under deformation of Liouville sectors by \cite{GPS1}.
We shall see that the category $\n(X) = \n(\bar{X}, F_{0})$ also shares a similar property.

\begin{prop}
A deformation of Liouville sectors $\{X_{t}\}_{t \in [0, 1]}$ induces a quasi-equivalence $\n(X_{0}) \cong \n(X_{1})$.
\end{prop}
\begin{proof}
The proof basically follows that of \cite{GPS1}.
The technical issue here is that in our setup, $\n(X)$ is defined using Floer data on $\bar{X}$, so it is not automatic that an inclusion of sectors $X \xhookrightarrow{} X'$ always induces a functor $\n(X) \to \n(X')$.
However, when there is a deformation of Liouville sectors $\{X_{t}\}_{t \in [0, 1]}$, 
their convex completions are deformation equivalent, 
and the whole deformation can be broken up into small pieces such that every small piece is a zigzag of trivial inclusions.
Without loss of generality, we may assume that the deformation is a trivial inclusion, of the form $X_{-a} \xhookrightarrow{} X$ for $X_{-a} = e^{-aX_{I}}(X)$ where $X_{I}$ is the Hamiltonian vector field for the defining function $I$ of $X$.
Looking at the corresponding sutured Liouville manifolds,
such an inclusion can be described as an inclusion 
\[
\bar{X} \setminus (F_{0} \times \R_{|t| \le \e'} \times \R_{s \ge 0})^{\circ} \xhookrightarrow{} \bar{X} \setminus (F_{0} \times \R_{|t| \le \e} \times \R_{s \ge 0})^{\circ}
\]
by deforming the neighborhood $F_{0} \times \R_{|t| \le \e}$ of $F_{0} \subset \p \bar{X}_{0}$ to a bigger neighborhood for some $\e' > \e$, 
as long as there still exists a function $I_{F_{0} \times \R_{|t| \le \e'}}: F_{0} \times \R_{|t| \le \e'} \to \R$ whose derivative along the Reeb vector field is positive, and the contact vector field for this function is outward pointing along $\p (F_{0} \times \R_{|t| \le \e'})$.
Under this description, there is an obvious functor $\w(X_{-a}) \to \w(X)$ which send every Lagrangian $L$ to itself, 
and is the identity on morphism spaces, since there is no newly created Reeb chord linking $F_{0}$ when shrinking $F_{0} \times \R_{|t| \le \e'}$ to $F_{0} \times \R_{|t| \le \e}$.
This implies that the functor induces a functor $\n(X_{-a}) \to \n(X)$,
which by the above analysis is indeed an equivalence.
\end{proof}

\subsection{Other variants}\label{sec: variants}

The definition of $CW^{*}_{F}(L_{0}, L_{1}; H)$ as in \eqref{partially wrapped CW} indeed comes from an increasing exhausting filtration on $CW^{*}(L_{0}, L_{1}; H)$, 
defined by putting an upper bound on the signed number of intersection points $x \cdot (F_{0} \times \R_{s>0})$ of an $H$-chord $x$ with $F_{0} \times \R_{s>0}$.
Unlike $CW^{*}_{F}(L_{0}, L_{1}; H)$, the other subspaces with more intersection points are not closed under the $\ainf$-operations $\mu^{k}_{\w(\bar{X})}$.
However, passing to the complex $NC_{F}^{*}(L_{0}, L_{1})$, we shall see that there are other finite-dimensional subspaces that are closed under the $\ainf$-operations $\mu^{k}_{\n}$.

First, we note that for each $m \in \Z_{\le -1}$, the following direct product
\begin{equation}\label{negative less than m piece}
\prod_{\substack{y \in \chi(L_{0}, L_{1}; -H)\\y \cdot (F_{0} \times \R_{s>0}) \le m}} |o_{y}|_{\K}
\end{equation}
can be identified with a subspace of $CW^{*}(L_{0}, L_{1}; -H)$. 
One way of seeing this is to observe that the linear dual of the quotient of $CW^{n-*}(L_{1}, L_{0}; H)$ by the subspace generated by chords $x$ such that $x \cdot (F_{0} \times \R_{s>0}) \le -m-1$ is precisely \eqref{negative less than m piece} under the duality isomorphism \eqref{PD},
which is therefore identified with a subspace of the linear dual of $CW^{n-*}(L_{1}, L_{0}; H)$.

The filtration naturally extends to the Rabinowitz complex $RC^{*}(L_{0}, L_{1})$ as follows.
With respect to the direct sum \eqref{Rabinowitz complex}, for each $m \in \Z$ define the $m$-th filtered subspace to be 
\begin{equation}\label{filtered subspace}
C_{m} = C_{m}(L_{0}, L_{1}):= \prod_{\substack{y \in \chi(L_{0}, L_{1}; -H)\\y \cdot (F_{0} \times \R_{s>0}) \le m}} |o_{y}|_{\K}[1] \oplus \bigoplus_{\substack{x \in \chi(L_{0}, L_{1}; H)\\x \cdot (F_{0} \times \R_{s>0}) \le m}} |o_{x}|_{\K}.
\end{equation}
Note that, for $m=0$, $C_{0} = NC^{*}_{F}(L_{0}, L_{1})$, which is closed under the differential $\mu^{1}_{\n}$ and moreover under the $\ainf$-operations $\mu^{k}_{\n}$.
For arbitrary $m$, we have:

\begin{lem}
For any $m \in \Z$, $C_{m}$ is closed under the differential $\mu^{1}_{\rw}$.
\end{lem}
\begin{proof}
The Floer differential on each piece does not increase the number of intersection points of a chord with $F_{0} \times \R_{s>0}$.
By Lemma \ref{image of continuation}, the continuation map \eqref{continuation} has image in $CW^{*}_{F}(L_{0}, L_{1}; H)$,
which are generated by those chords $x \in \chi(L_{0}, L_{1}; H)$ with $x \cdot (F_{0} \times \R_{s>0}) = 0$.
For $m < 0$, the summand $\bigoplus_{\substack{x \in \chi(L_{0}, L_{1}; H)\\x \cdot (F_{0} \times \R_{s>0}) \le m}} |o_{x}|_{\K}$ in \eqref{filtered subspace} is zero, as the set of such chords is empty by positivity of intersection.
\end{proof}

Consider the sub-quotient $C_{m'}/C_{m}$ for $m'>m, m, m' \in \Z$, which as $\K$-modules is naturally isomorphic to
\begin{equation}
C_{m'}/C_{m} \cong \bigoplus_{\substack{y \in \chi(L_{0}, L_{1}; -H)\\m+1 \le y \cdot (F_{0} \times \R_{s>0}) \le m'}} |o_{y}|_{\K}[1] \oplus \bigoplus_{\substack{x \in \chi(L_{0}, L_{1}; H)\\m+1 \le x \cdot (F_{0} \times \R_{s>0}) \le m'}} |o_{x}|_{\K}.
\end{equation}
For $m'=0$, the space $C_{m'}$ is closed under $\ainf$-operations. 
Now we shall show that the quotients $C_{0}/C_{m}$ inherit an $\ainf$-structure from $\mu^{k}_{\n}$. 
Denote by 
\begin{equation}\label{morphism in nm}
NC^{*}_{>m}(L_{0}, L_{1}): = C_{0}/C_{m} = \bigoplus_{\substack{y \in \chi(L_{0}, L_{1}; -H)\\m+1 \le y \cdot (F_{0} \times \R_{s>0}) \le 0}} |o_{y}|_{\K}[1] \oplus \bigoplus_{\substack{x \in \chi(L_{0}, L_{1}; H)\\x \cdot (F_{0} \times \R_{s>0}) = 0}} |o_{x}|_{\K}
\end{equation}
for each $m \in \Z_{\le -1}$. 

\begin{prop}\label{prop: ainf structure on the quotient}
For each $m \in \Z_{\le -1}$, there exist maps
\begin{equation}\label{structure maps on nm}
\mu^{k}_{>m}: NC^{*}_{>m}(L_{k-1}, L_{k}) \otimes \cdots \otimes NC^{*}_{>m}(L_{0}, L_{1}) \to NC^{*}_{>m}(L_{0}, L_{k})
\end{equation}
satisfying the $\ainf$-relations.
\end{prop}

To prove this, we shall see that for $m \le -1$, $C_{m}$ is indeed an $\ainf$-ideal of $C_{0}$. 
This can be proved using the following fact about the relevant moduli spaces:

\begin{lem}
Let $m \in \Z_{\le -1}$.
For $j = 1, \ldots, k$, $x_{j}$ be either a time-one $H$-chord from $L_{j-1}$ to $L_{j}$ such that $x_{j} \cdot (F_{0} \times \R_{s>0}) = 0$, for which $w_{j} = 0$ so that $x_{j}$ is treated as an input for a popsicle map $u \in \mathcal{R}^{k+1, \mathbf{p}, \mathbf{w}}(\mathbf{x})$, 
or a time-one $H$-chord from $L_{j}$ to $L_{j-1}$, for which $w_{j}=-1$ so that $x_{j}$ is treated as an output for $u$.
Suppose there is some $1 \le i \le k$ such that $w_{i} = -1$ and $x_{i}$ is a time-one $H$-chord from $L_{i}$ to $L_{i-1}$ with $x_{i} \cdot (F_{0} \times \R_{s>0}) \ge -m$, treated as an output.
Then the moduli space $\mathcal{R}^{k+1, \mathbf{p}, \mathbf{w}}(\mathbf{x})$ is empty unless $x_{0}$ also satisfies $x_{0} \cdot (F_{0} \times \R_{s>0}) \ge -m$.
\end{lem}
\begin{proof}
Again, when treating various chords as asymptotics of pseudoholomorphic maps, we should follow the principle declared in Remark \ref{rem on chords}.
By \eqref{intersection number identity}, the number of intersection points of input chords with $F_{0} \times \R_{s>0}$ must be greater than or equal to that of output chords with $F_{0} \times \R_{s>0}$.
Since $x_{i}$ is an output with $x_{i} \cdot (F_{0} \times \R_{s>0}) \ge -m$, 
and the other $x_{j}$ with $1 \le j \le k, j \neq i$ must be either an input with $x_{j} \cdot (F_{0} \times \R_{s>0}) = 0$, 
or an output with $x_{j} \cdot (F_{0} \times \R_{s>0}) \ge 0$,
it follows that $x_{0}$ must be an input, in which case $w_{0} = -1$, and $x_{0} \cdot (F_{0} \times \R_{s>0}) \ge -m$.
\end{proof}

\begin{cor}\label{ainf ideal}
For any $m \in \Z_{\le -1}$, the $\ainf$-operations $\mu^{k}_{\n}$ satisfy 
\begin{equation}
\mu^{k}_{\n}(C_{0}(L_{k-1}, L_{k}) \otimes \cdots \otimes C_{m}(L_{i-1}, L_{i}) \otimes \cdots \otimes C_{0}(L_{0}, L_{1})) \subset C_{m}(L_{0}, L_{k}).
\end{equation}
\end{cor}

\begin{proof}[Proof of Proposition \ref{prop: ainf structure on the quotient}]
We define $\mu^{k}_{>m}$ \eqref{structure maps on nm} as follows. 
For $[c_{j}] \in NC^{*}_{>m}(L_{j-1}, L_{j}) = C_{0}(L_{j-1}, L_{j})/C_{m}(L_{j-1}, L_{j})$, 
choose a representative $c_{j} \in C_{0}(L_{j-1}, L_{j}) = NC^{*}_{F}(L_{j-1}, L_{j})$, and define
\begin{equation}
\mu^{k}_{>m}([c_{k}], \ldots, [c_{1}]) = [\mu^{k}_{\n}(c_{k}, \ldots, c_{1})] \in C_{0}(L_{0}, L_{k})/C_{m}(L_{0}, L_{k}).
\end{equation}
If $\tilde{c}_{j}$ is another representative of $[c_{j}]$, so that $c_{j} - \tilde{c}_{j} \in C_{m}(L_{j-1}, L_{j})$.
By Corollary \ref{ainf ideal}, the difference 
\[
\mu^{k}(c_{k}, \ldots, c_{1}) - \mu^{k}(\tilde{c}_{k}, \ldots, \tilde{c}_{1}) 
\]
belongs to $C_{m}(L_{0}, L_{k})$, so $\mu^{k}_{>m}([c_{k}], \ldots, [c_{1}])$ is well-defined.
The $\ainf$-relations for $\mu^{k}_{>m}$ follow from the $\ainf$-relations for $\mu^{k}_{\n}$.
This completes the proof of Proposition \ref{prop: ainf structure on the quotient}.
\end{proof}

\begin{defn}
For each $m \in \Z_{\le -1}$ we define an $\ainf$-category 
\begin{equation}
\n_{>m}(\mathbf{L})
\end{equation}
 whose objects are Lagrangian submanifolds in $\mathbf{L}$,
morphism spaces are $NC^{*}_{>m}(L_{0}, L_{1})$,
and $\ainf$-structure maps are given by $\mu^{k}_{>m}$ \eqref{structure maps on nm}.

When no confusion may occur, 
we shall denote this category by$\n_{>m}(\bar{X}, F_{0})$ for the sutured Liouville manifold $(\bar{X}, F_{0})$,
or $\n_{>m}(X)$ for the corresponding Liouville sector $X$.
\end{defn}

\subsection{A diagram of categories}\label{sec: diagram of Fukaya categories}

Each of the $\ainf$-categories $\n_{>m}(\bar{X}, F_{0}), m \in \Z_{\le -1}$ constructed in the previous subsection has finite dimensional morphism spaces of the form $C_{0}/C_{m}$.
For different values of $m$, these quotient spaces are related in by the obvious quotient maps
\begin{equation}\label{map on quotients}
C_{0}/C_{m} \to C_{0}/C_{m'},
\end{equation}
which are induced by the inclusions $C_{m} \subset C_{m'}$ for all $m \le m'$.

\begin{lem}
The maps \eqref{map on quotients} can be extended to an $\ainf$-functor
\begin{equation}
\pi_{m, m'}: \n_{>m}(\bar{X}, F_{0}) \to \n_{>m'}(\bar{X}, F_{0})
\end{equation}
with vanishing higher order terms, where $\pi_{m, m} = \id$.
\end{lem}
\begin{proof}
Let $[c_{j}^{(m)}] \in NC^{*}_{>m}(L_{j-1}, L_{j}), j = 1, \ldots, k$. Choose a representative $c_{j} \in C_{0}(L_{j-1}, L_{j}) = NC^{*}_{F}(L_{j-1}, L_{j})$. 
Then the image of $[c_{j}^{(m)}]$ under $\pi_{m, m'}$ is the equivalence class of $c_{j}$ in $C_{0}/C_{m'}$, 
denoted by $[c_{j}^{(m')}] \in NC^{*}_{>m}(L_{j-1}, L_{j})$.
We need to show that the equivalence class of the product
$\mu^{k}_{\n}(c_{k}, \ldots, c_{1})$ in $NC^{*}_{>m'}(L_{0}, L_{k})$ is independent of the choices of representatives.
Suppose $\tilde{c}_{j}$ is another representative of $[c_{j}^{(m)}]$, so that $c_{j} - \tilde{c}_{j} \in C_{m}(L_{j-1}, L_{j})$.
By Corollary \ref{ainf ideal}, the difference 
\[
\mu^{k}(c_{k}, \ldots, c_{1}) - \mu^{k}(\tilde{c}_{k}, \ldots, \tilde{c}_{1}) 
\]
belongs to $C_{m}(L_{0}, L_{k})$.
Since $C_{m} \subset C_{m'}$ for $m < m'$, this difference also belongs to $C_{m'}(L_{j-1}, L_{j})$.
It follows that 
\[
[\mu^{k}(c_{k}, \ldots, c_{1})^{(m')}] = [\mu^{k}(\tilde{c}_{k}, \ldots, \tilde{c}_{1})^{(m')}] \in NC^{*}_{>m'}(L_{0}, L_{k}).
\]
\end{proof}

The following proposition follows immediately from the construction and the lemma above.

\begin{prop}
The functors $\pi_{m, m'}$ form a strictly commuting diagram of $\ainf$-categories
\begin{equation}
\{\n_{>m}(\bar{X}, F_{0})\}_{m \in \Z_{\le -1}}.
\end{equation}
\end{prop} \qed

\subsection{Relations between $\w, \n$ and $\n_{>m}$}\label{sec: various functors}

By construction, the category $\n(\bar{X}, F_{0})$ comes with a natural $\ainf$-functor 
\begin{equation}\label{negative wrapping functor}
j: \w(\bar{X}, F_{0}) \to \n(\bar{X}, F_{0}).
\end{equation}
This is defined as follows.
On objects, $j$ is the identity.
Define the multilinear maps $j^{k}$ to be such that $j^{1}$ is the inclusion $CW^{*}_{F}(L_{0}, L_{1}; H) \subset NC^{*}_{F}(L_{0}, L_{1}; H)$ with respect to the decomposition \eqref{morphism in n},
and $j^{k} = 0$ for all $k \ge 2$.

\begin{lem}
The multilinear maps constituting $j$ \eqref{negative wrapping functor} defined as above form an $\ainf$-functor.
\end{lem}
\begin{proof}
Suppose all the inputs are from $CW^{*}_{F}(L_{j-1}, L_{j}; H)$, and are represented by time-one $H$-chords from $L_{j-1}$ to $L_{j}$.
Then any popsicle map in $\mathcal{R}^{k+1, \mathbf{p}, \mathbf{w}}(\mathbf{x})$ with these inputs must have a trivial flavor $\mathbf{p} = \varnothing$ because of \eqref{weight sum condition} and \eqref{sprinkles bounded by weight}.
That is, such a popsicle map is an ordinary $\ainf$-disk, and has output in $CW^{*}(L_{0}, L_{k}; H)$ with zero intersection with $F_{0} \times \R_{s>0}$, i.e. in $CW^{*}_{F}(L_{0}, L_{k}; H)$.
\end{proof}

For every $m \in \Z_{\le -1}$, there are $\ainf$-functors
\begin{equation}
\pi_{m}: \n(\bar{X}, F_{0}) \to \n_{>m}(\bar{X}, F_{0}),
\end{equation}
which is identity on objects,
whose linear terms are induced by the projections $C_{0} \to C_{0}/C_{m}$,
and all the higher order terms are zero.
The fact that this is an $\ainf$-functor follows immediately from the construction of $\n_{>m}(\bar{X}, F_{0})$.
Also, define
\begin{equation}
j_{m}: \w(\bar{X}, F_{0}) \to \n_{>m}(\bar{X}, F_{0})
\end{equation}
by the composition $j_{m} = \pi_{m} \circ j$.

Compatibility between various functors above can be achieved by consistent choices of Floer data involved in the definitions of these categories, the existence of which follows from a standard induction argument.

\begin{prop}\label{compatible diagram of functors}
There exist generic choices of compatible families of Floer data,
such that the functors $\pi_{m, m'}$ form a diagram of $\ainf$-categories
\begin{equation}\label{diagram of s categories}
\{\n_{>m}(\bar{X}, F_{0})\}_{m \in \Z_{\le -1}},
\end{equation}
and $\pi_{m}, j_{m}$ form strictly compatible diagrams of functors
\begin{equation}\label{diagram of functors pi}
\{\pi_{m}: \n(\bar{X}, F_{0}) \to \n_{>m}(\bar{X}, F_{0})\}_{m \in \Z_{\le -1}}
\end{equation}
and
\begin{equation}\label{diagram of functors j}
\{j_{m}: \w(\bar{X}, F_{0}) \to \n_{>m}(\bar{X}, F_{0})\}_{m \in \Z_{\le -1}}.
\end{equation}
\end{prop}
\begin{proof}[Sketch of proof]
By definition, the morphism spaces in $\n_{>m}(\bar{X}, F_{0})$ are quotients $C_{0}/C_{m}$ of morphism spaces in $\n(\bar{X}, F_{0})$,
and the $\ainf$-structure maps are exactly induced by those of $\n(\bar{X}, F_{0})$. 
Thus, universal and conformally consistent choices of Floer data defining the $\mu^{k}_{\n(\bar{X}, F_{0})}$ automatically determine Floer data defining $\mu^{k}$ on $\n_{>m}(\bar{X}, F_{0})$.
It follows that the functors $\pi_{m, m'}$ are strictly commuting,
and that $\pi_{m}$ are strictly compatible with them.
\end{proof}

The natural partial order $\le$ on $\Z_{\le -1}$ has a maximal element $-1$,
so the homotopy colimit is not interesting.
However, there is not a minimal element, so we will be interested in the homotopy limit of \eqref{diagram of s categories} 
\begin{equation}
\holim_{m \in \Z_{\le -1}} \n_{>m}(\bar{X}, F_{0}).
\end{equation}
A concrete construction of the homotopy limit is given in Appendix \ref{app: homotopy limits}.

This diagram is an inverse system indexed by $\Z_{\le -1}$.
Take a pair of Lagrangians $L_{0}, L_{1}$ as objects in $\n_{>m}(\bar{X}, F_{0})$ for any $m$, 
and consider the morphism space 
\[
\n_{>m}(L_{0}, L_{1}) := \hom_{\n_{>m}(\bar{X}, F_{0})}(L_{0}, L_{1}) = C_{0}(L_{0}, L_{1})/C_{m} (L_{0}, L_{1}).
\]
These form an inverse system of cochain complexes.
By our construction, the following observation about this inverse system of cochain complexes is immediate:

\begin{lem}\label{nm satisfies ml}
For every pair $(L_{0}, L_{1})$, inverse system of cochain complexes $\{\n_{>m}(L_{0}, L_{1})\}$ satisfies the degree-wise {\it Mittag-Leffler condition} in the sense of Definition \ref{def: Mittag-Leffler}.
\end{lem}
\begin{proof}
Every functor in the diagram is induced by a projection onto a quotient space, $C_{0}/C_{m} \to C_{0}/C_{m'}$, which is clearly surjective.
\end{proof}

Since $\{\n_{>m}(\bar{X}, F_{0})\}_{m \in \Z_{\le -1}}$ is an inverse system in which all the $\n_{>m}(\bar{X}, F_{0})$ have the same collection of objects and the functors $\pi_{m', m}$ act as identity on objects,
it follows that the constructions and results in Appendix \S\ref{app: map from inverse limit} apply. 
In particular, for each Lagrangian $L$, there is a corresponding constant section $s^{L}$ associated to $L$, as an object of the homotopy limit $\holim_{m \in \Z_{\le -1}} \n_{>m}(\bar{X}, F_{0})$.
By Lemmata \ref{nm satisfies ml} and \ref{lem: ml implies lim = holim}, there is a quasi-isomorphism from the ordinary inverse limit of cochain complexes to the morphism complex in the homotopy limit category:
\begin{equation}\label{iso of limit of chain complexes}
 \varprojlim_{m \in Z_{\le -1}} \n_{>m}(L_{0}, L_{1}) \stackrel{\sim}\to \holim_{m \in \Z_{\le -1}} \n_{>m}(\bar{X}, F_{0})(s^{L_{0}}, s^{L_{1}}) 
\end{equation}
for every pair of Lagrangians $(L_{0}, L_{1})$.
The latter morphism space agrees with the homotopy limit of the cochain complexes $\{\n_{>m}(L_{0}, L_{1})\}_{m \in \Z_{\le -1}}$,
\begin{equation}\label{hom in holim}
\holim_{m \in \Z_{\le -1}} \n_{>m}(\bar{X}, F_{0})(s^{L_{0}}, s^{L_{1}}) = \holim_{m \in \Z_{\le -1}} \{\n_{>m}(L_{0}, L_{1})\}_{m \in \Z_{\le -1}}.
\end{equation}
The existence of such a map \eqref{iso of limit of chain complexes} is given in \eqref{chain map from lim to holim} in Appendix \ref{app: map from inverse limit}. 

The diagrams of functors \eqref{diagram of functors pi} and \eqref{diagram of functors j} define natural functors
\begin{equation}\label{homotopy limit of pi}
\Pi: \n(\bar{X}, F_{0}) \to \holim_{m \in \Z_{\le -1}} \n_{>m}(\bar{X}, F_{0})
\end{equation}
and 
\begin{equation}\label{homotopy limit of j}
J: \w(\bar{X}, F_{0}) \to \holim_{m \in \Z_{\le -1}} \n_{>m}(\bar{X}, F_{0}),
\end{equation}
such that $J = \Pi \circ j$.
We shall show that $\Pi$ \eqref{homotopy limit of pi} is fully faithful in favorable circumstances.

\begin{prop}\label{n as homotopy limit of nm}
Suppose $\p_{\infty} \bar{X}$ has bounded Reeb dynamics relative to $F$ in the sense of Definition \ref{def: bounded Reeb dynamics}.
Then the natural functor $\Pi: \n(\bar{X}, F_{0}) \to \holim_{m \in \Z_{\le -1}} \n_{>m}(\bar{X}, F_{0})$ \eqref{homotopy limit of pi} is fully faithful.
\end{prop}

Given the quasi-isomorphism \eqref{iso of limit of chain complexes} and the identification \eqref{hom in holim}, it will be sufficient to show that the inverse system formed by the cohomology groups of the chain complexes also satisfy the degree-wise Mittag-Leffler condition.
This is the place where we need the assumption on the Reeb dynamics on $\p_{\infty} \bar{X}$, in the sense of Definition \ref{def: bounded Reeb dynamics}.
Recall that the morphism space in $\n_{>m}(\bar{X}, F_{0})$ is a quotient complex \eqref{morphism in nm}.

\begin{lem}\label{ml on cohomology}
Suppose $\p_{\infty} \bar{X}$ has bounded Reeb dynamics relative to $F$.
Then for each pair of Lagrangians $L_{0}, L_{1}$, the inverse system of  the cohomology groups
\[
\{H^{*}(\n_{>m}(L_{0}, L_{1}))\}_{m \in \Z_{\le -1}}
\]
satisfies the degree-wise Mittag-Leffler condition. 
\end{lem}
\begin{proof}
The key point is that for each $m$, the quotient complex $C_{0}/C_{m}$ is finite-dimensional by the assumption that $\p_{\infty} \bar{X}$ has bounded Reeb dynamics relative to $F$.
We would like to show that, for every $m \le -1$, there exists $m' \le m$ such that for all $m'' \le m'$, we have
\begin{equation}\label{mittag-leffler for nm}
\im(H^{*}(C_{0}/C_{m''}) \to H^{*}(C_{0}/C_{m})) = \im(H^{*}(C_{0}/C_{m'} \to H^{*}(C_{0}/C_{m})).
\end{equation}

Let $[c]_{m} \in C_{0}/C_{m}$ be any element, which we may assume to be a cocycle with respect to the induced quotient differential on $C_{0}/C_{m}$, hence representing a cohomology class.
Pick a representative $c \in C_{0}$ of $[c]_{m}$, i.e. any element $c \in C_{0}$ whose projection to $C_{0}/C_{m}$ is $[c]_{m}$.
If $c$ can be chosen such that $\mu^{1}_{\n}(c) = 0$, then the equivalence class $[c]_{m'}$ of $c$ in $C_{0}/C_{m'}$ for all $m' \le m$ is also a cocycle with respect to the induced quotient differential on $C_{0}/C_{m'}$.
By definition, the cohomology class of $[c]_{m'}$ is mapped to the cohomology class of $[c]_{m}$.
This implies that $H^{*}(C_{0}/C_{m'}) \to H^{*}(C_{0}/C_{m})$ is surjective.
If $\mu^{1}_{\n}(c) \neq 0$ for all representatives $c$, then there exists $m' < m$ such that the equivalence class $[c]_{m'}$ of $c$ in $C_{0}/C_{m'}$ is not closed under the induced quotient differential on $C_{0}/C_{m'}$.
Then, for any $m'' < m'$, the equivalence class $[c]_{m''}$ is not closed under the differential on $C_{0}/C_{m''}$, which implies that both maps on cohomology 
\[
H^{*}(C_{0}/C_{m'}) \to H^{*}(C_{0}/C_{m})
\]
and
\[
H^{*}(C_{0}/C_{m''}) \to H^{*}(C_{0}/C_{m})
\]
have zero image. 
Apply the same argument to all the vectors in a basis for the finite-dimensional $\K$-vector space $C_{0}/C_{m}$.
Since there are finitely many vectors, we may find the smallest $m'$ in the above argument, such that \eqref{mittag-leffler for nm} holds for all $m'' < m$.
\end{proof}

The degree-wise Mittag-Leffler condition for the inverse system formed by cohomology groups implies that $\varprojlim^{1}$ vanishes for this inverse system, and therefore:

\begin{cor}
The natural map 
\begin{equation}\label{h of limit to limit of h}
H^{*}(\varprojlim_{m \in Z_{\le -1}} \hom_{\n_{>m}}(L_{0}, L_{1})) \to \varprojlim_{m \in Z_{\le -1}} H^{*}(\hom_{\n_{>m}}(L_{0}, L_{1}))
\end{equation}
is an isomorphism.
\end{cor}

We have gathered all the ingredients and are ready to prove Proposition \ref{n as homotopy limit of nm}.

\begin{proof}[Proof of Proposition \ref{n as homotopy limit of nm}]
It suffices to show that the linear map $\Pi^{1}$ for every pair $(L_{0}, L_{1})$ is a quasi-isomorphism.
We compose the map $H^{*}(\Pi^{1})$ with the cohomology map of the quasi-isomorphism \eqref{iso of limit of chain complexes} and further with \eqref{h of limit to limit of h} to get a map
\[
H^{*}(\n(L_{0}, L_{1})) \to H^{*}(\varprojlim_{m \in Z_{\le -1}} \n_{>m}(L_{0}, L_{1})) \to \varprojlim_{m \in Z_{\le -1}} H^{*}(
\n_{>m}(L_{0}, L_{1})).
\]
By the universal property, this is the inverse limit of the diagram of maps
\[
H^{*}(\pi_{m}): H^{*}(\n(L_{0}, L_{1})) \to H^{*}(\n_{>m}(L_{0}, L_{1})),
\]
i.e. the unique lift of these maps to the inverse limit $\varprojlim_{m \in Z_{\le -1}} H^{*}(\n_{>m}(L_{0}, L_{1}))$.
It is straightforward to check that this lift is an isomorphism, using the descending filtration 
\[
\cdots \subset C_{m} \subset \cdots \subset C_{0} = \n(L_{0}, L_{1}).
\]
\end{proof}

In the homotopy limit $\holim_{m \in \Z_{\le -1}} \n_{>m}(\bar{X}, F_{0})$, 
there is the subcategory $\mathcal{S}(\{\n_{>m}(\bar{X}, F_{0})\}_{m \in \Z_{\le -1}})$ of constant sections defined in Definition \ref{def: subcat of constant sections}.
Finally we make the following observation about the functor $\Pi$:

\begin{lem}\label{lem: image of pi}
The functor $\Pi: \n(\bar{X}, F_{0}) \to \holim_{m \in \Z_{\le -1}} \n_{>m}(\bar{X}, F_{0})$ \eqref{homotopy limit of pi} has image landing in the subcategory $\mathcal{S}(\{\n_{>m}(\bar{X}, F_{0})\}_{m \in \Z_{\le -1}})$.
\end{lem}
\begin{proof}
This is a statement on the level of objects, which immediately follows from the definitions of the diagram $\{\n_{>m}(\bar{X}, F_{0})\}_{m \in \Z_{\le -1}}$ and the functor $\Pi$.
\end{proof}

\subsection{Pushforward functors}\label{section: pushforward}

An inclusion of Liouville sectors $\iota: X \xhookrightarrow{} X'$ is a proper exact embedding which is a diffeomorphism onto the image,
such that either $X = X'$ or $\iota(\p X) \cap X' = \varnothing$.
An inclusion of stopped Liouville sectors $\iota: (X, \fk) \xhookrightarrow{} (X', \fk')$ is an inclusion of Liouville sectors $\iota: X \xhookrightarrow{} X'$ such that $\fk' \cap (\p_{\infty} X)^{\circ} = \fk$.
The partially wrapped Fukaya category is covariantly functorial with respect to inclusions of stopped Liouville sectors,
in particular for a single inclusion $\iota: X \xhookrightarrow{} X'$, there is a natural pushforward functor
\begin{equation}\label{pushforward on w}
\iota_{*}: \w(X, \fk) \to \w(X', \fk').
\end{equation}
The functor \eqref{pushforward on w} is fully faithful if the inclusion is {\it forward stopped} (Corollary 8.17 of \cite{GPS2}).

The negatively wrapped Fukaya category $\n(X) = \n(\bar{X}, F_{0})$ has similar functoriality under inclusions, 
which is nonetheless more restricted due to the nature of our definition relying on the presence of a ribbon $F$.
We consider the situation where an inclusion of Liouville sectors is given by an inclusion of stops $\fk' \subset \fk$,
realized by an inclusion of the Liouville sub-domain $F'_{0} \subset F_{0}$,
while the convex completions are the same, $\bar{X} = \bar{X}'$.
Note that here we allow slightly general notions of Liouville domains and sub-domains: $F_{0}$ may be disconnected and have several connected components, 
and $F'_{0}$ is the union of several sub-domains in the connected components, some of which are allowed to be empty.

\begin{prop}\label{prop: pushforward on n}
In the case of a stop removal inclusion described as above, there is a pushforward functor
\begin{equation}\label{pushforward on n}
\iota_{\sharp}: \n(\bar{X}, F_{0}) \to \n(\bar{X}, F'_{0}),
\end{equation}
which is identity on objects.
\end{prop}
\begin{proof}
The key point is that any quadratic Hamiltonian compatible with $F_{0}$ is automatically compatible with $F'_{0}$.
So we can define maps $\iota_{\sharp}^{k}$ by letting the linear term $\iota_{\sharp}^{1}$ be the obvious inclusion of subcomplexes,
and the higher order terms be zero.
It is straightforward to check that the maps defined this way satisfy the $\ainf$-functor equations.
\end{proof}


\subsection{Fully-faithfulness}\label{section: nonessential Lagrangians}

We conclude \S\ref{section: Fukaya categories} by studying the property of the functor $j$ on a certain class of Lagrangians.
These Lagrangians serve as candidates of zero objects in the partially wrapped Fukaya category.

\begin{defn}\label{defn: nonessential Lagrangian}
We say an (exact cylindrical) Lagrangian $L$ in $\bar{X}$ is {\it nonessential},
 if there is a Liouville hypersurface (with boundary) $S_{0} \subset \p_{\infty} \bar{X}$ as well as a sectorial inclusion $S \times \C_{\re \ge 0} \xhookrightarrow{} \bar{X}$ such that $L$ is can be isotoped into the interior of the image of the inclusion.
\end{defn}

In other words, a nonessential Lagrangian is one that can be isotoped into the complement of the sector $X'$ that is obtained by completing the Liouville domain $\bar{X}_{0}$ (whose completion is $\bar{X}$) away from $S_{0} \subset \p \bar{X}_{0}$.
The reason for the term `nonessential' is the following:

\begin{lem}[Corollary 9.2 of \cite{GPS2}]\label{nonessential Lagrangian zero in fully wrapped}
Let $L$ be a nonessential Lagrangian of the Liouville manifold $\bar{X}$.
Then $L$ is a zero object in $\w(\bar{X})$.
Furthermore, if there is such a Liouville hypersurface $S_{0}$ such that $F_{0} \cap \p_{\infty}(S \times \C_{\re \ge 0})^{\circ} = \varnothing$, 
then $L$ is a zero object in $\w(\bar{X}, F_{0})$.
\end{lem}

Nonessential Lagrangians are not necessarily zero objects in $\w(X) = \w(\bar{X}, F_{0})$.
Of course the second half of Lemma \ref{nonessential Lagrangian zero in fully wrapped} gives a criterion for a nonessential Lagrangian $L$ for being a zero object in $\w(\bar{X}, F_{0})$,
but that criterion relies on the existence of some $S_{0}$,
which is not a priori guaranteed for a given Lagrangian $L$.
The general expectation is that any Lagrangian which does not intersect the {\it relative core} $\ck_{X, \fk} = \ck_{X} \cup (\fk \times \R)$ is a zero object; 
in the real analytic case where the relative core is subanalytic singular isotropic, this generality follows from \cite{GPS3}.
However, for our purpose of studying the category $\n(\bar{X}, F_{0})$ we just need the following result:

\begin{prop}\label{faithful on nonessential Lagrangians}
The functor $j: \w(\bar{X}, F_{0}) \to \n(\bar{X}, F_{0})$ is fully faithful on the subcategory whose objects are nonessential Lagrangians.
\end{prop}

The first half of Lemma \ref{nonessential Lagrangian zero in fully wrapped} immediately implies that Proposition \ref{faithful on nonessential Lagrangians} holds automatically for the functor $\w(\bar{X}) \to \rw(\bar{X})$, i.e. in the case where $F_{0} = \varnothing$.
In fact, we shall see that the general case follows by the same reasoning.

\begin{proof}[Proof of Proposition \ref{faithful on nonessential Lagrangians}]
It suffices to show that the induced map on cohomology is an isomorphism for each pair $(K, L)$ of nonessential Lagrangians.
Let $L$ be any nonessential Lagrangian.
By Lemma \ref{nonessential Lagrangian zero in fully wrapped},
 $L$ is a zero object in the fully wrapped Fukaya category $\w(\bar{X})$, 
and is also a zero object in the Rabinowitz Fukaya category $\rw(\bar{X})$.
By definition of the Rabinowitz complex \eqref{Rabinowitz complex},
it also follows that $H^{*}(CW^{*}(L, L; -H)) = 0$.
Therefore, the inclusion 
\[
CW^{*}_{F}(L, L; H) \subset NC^{*}_{F}(L, L) = \cone(c: CW^{*}(L, L; -H) \to CW^{*}_{F}(L, L; H))
\]
is quasi-isomorphism. 
The same argument can be applied to any pair $(K, L)$ of nonessential Lagrangians.
\end{proof}

Proposition \ref{faithful on nonessential Lagrangians} together with Lemma \ref{n as homotopy limit of nm} yields the following consequence.

\begin{cor}\label{faithful on nonessential Lagrangians for limit}
The functor $J: \w(\bar{X}, F_{0}) \to \holim_{m \in \Z_{\le -1}^{op}} \n_{>m}(\bar{X}, F_{0})$ \eqref{homotopy limit of j} is fully faithful on nonessential Lagrangians.
\end{cor}

The proofs of Proposition \ref{faithful on nonessential Lagrangians} and its prerequisite Lemma \ref{nonessential Lagrangian zero in fully wrapped} indeed imply a slightly stronger statement:

\begin{prop}\label{quasi-iso on nonessential Lagrangians}
Let $L$ be a nonessential Lagrangian, and $K$ any exact cylindrical Lagrangian, both disjoint from $F_{0}$ at infinity.
Then the linear term of $j: \w(\bar{X}, F_{0}) \to \n(\bar{X}, F_{0})$ is a quasi-isomorphism for the pair $(L, K)$ or $(K, L)$:
\[
j^{1}: \hom_{\w(\bar{X}, F_{0})}(L, K) \stackrel{\sim}\to \hom_{\n(\bar{X}, F_{0})}(j(L), j(K)),
\]
\[
j^{1}: \hom_{\w(\bar{X}, F_{0})}(K, L) \stackrel{\sim}\to \hom_{\n(\bar{X}, F_{0})}(j(K), j(L)),
\]
\end{prop}
\begin{proof}
The vanishing of $H^{*}(CW^{*}(L, L; -H))$ implies the vanishing of $H^{*}(CW^{*}(L, K; -H))$ or $H^{*}(CW^{*}(K, L; -H))$ for any $K$.
\end{proof}

\section{The categorical equivalence}\label{section: equivalence}

\subsection{Overview of the construction}\label{section: functor overview}

Consider the cup functor
\[
i_{F}: \w(F) \to \w(\bar{X}, F_{0})
\]
defined in \eqref{cup functor}, which is the composition of the K\"{u}nneth stabilization functor $\w(F) \to \w(F \times T^{*}[0, 1])$ with the pushforward functor $\w(F \times T^{*}[0, 1]) \to \w(\bar{X}, F_{0})$ \eqref{pushforward on w}.
Let $\tilde{\dd}(F)$ be the essential image of $i_{F}$.
Then we can define the categorical completion of $\w(\bar{X}, F_{0})$ along the subcategory $\tilde{\dd}(F)$
\begin{equation}\label{formal completion along image of cup functor}
\wf:= \widehat{\w(\bar{X}, F_{0})}_{\tilde{\dd}(F)}
\end{equation}
as in Definition \ref{defn: formal completion}.

If $\fk = \ck_{F}$ is a mostly Legendrian stop, there is alternative replacement of $\tilde{\dd}(F)$ along which we can perform the categorical formal completion.
Let $\dd(\fk)$ the full subcategory of $\w(\bar{X}, F_{0})$ with objects being small Lagrangian disks linking $\fk^{crit}$.
We can also define 
\begin{equation}\label{formal completion along linking disks}
\wcm := \widehat{\w(\bar{X}, \fk)}_{\dd(\fk)}.
\end{equation}
If $F$ is Weinstein, the cocores of $F$ generate $\w(F)$, and these cocores are sent to these Lagrangian linking disks under the cup functor $i_{F}$,
the subcategories $\dd(\fk)$ and $\tilde{\dd}(F)$ are quasi-equivalent after passing to twisted complexes.
By Lemma \ref{invariance of formal completion}, in this case $\wcm$ \eqref{formal completion along linking disks} is quasi-equivalent to $\widehat{\w}_{F}$ \eqref{formal completion along image of cup functor}.

The main goal of this section is to construct a canonical $\ainf$-functor
\begin{equation}\label{functor to completion}
\Psi: \n(\bar{X}, F_{0}) \to \widehat{\w}_{F},
\end{equation}
and prove the following:

\begin{thm}\label{thm:equivalence}
Let $(\bar{X}, F)$ be a Liouville pair such that $\bar{X}$ and $F$ are both non-degenerate Liouville manifolds.
Suppose $F$ satisfies stop removal for $(\bar{X}, F_{0})$, i.e. the functor \eqref{stop removal functor} is a quasi-equivalence, 
and that $\fk = \ck_{F}$ is a full stop.
Then the functor $\Psi: \n(\bar{X}, F_{0}) \to \widehat{\w}_{F}$ \eqref{functor to completion} is a quasi-equivalence.
\end{thm}

There are some special and interesting cases in which the hypotheses of Theorem \ref{thm:equivalence} can be easily checked.
For example, if $F$ is Weinstein, $\w(F)$ has a collection of strong generators by cocore disks,
and has the core $\fk = \ck_{F}$ being a mostly Legendrian stop, 
so by Theorem \ref{thm: stop removal} $F$ satisfies stop removal.
In addition, if $\bar{X}$ is the total space of a symplectic Landau-Ginzburg model with compact critical locus with a regular fiber-at-infinity $F$ being Weinstien,
such that $\bar{X}$ itself is a non-degenerate Liouville manifold, 
then $\w(\bar{X}, F_{0})$ is proper (Example 1.19 of \cite{GPS2}), 
so that all the assumptions of Theorem \ref{thm:equivalence} are satisfied.

We first give a construction of the functor for an arbitrary Liouville pair $(\bar{X}, F)$ independently of the hypotheses.
The components of the functor \eqref{functor to completion} are maps of the following form
\begin{equation}\label{components of psi}
\Psi^{k}: \hom_{\n(\bar{X}, F_{0})}(L_{k-1}, L_{k}) \otimes \cdots \otimes \hom_{\n(\bar{X}, F_{0})}(L_{0}, L_{1}) \to \hom_{\wf}(L_{0}, L_{k}).
\end{equation}
We find the following description of the morphism spaces in $\wf$ helpful in attempting to construct the functor $\Psi$:

\begin{lem}\label{lem: wf as prop d}
The category $\wf$ is quasi-equivalent to $ess-im(\w(\bar{X}, F_{0}) \to \r{Prop}\tilde{\dd}(F))$.
\end{lem}
\begin{proof}
Recall from \eqref{morphism in completion} that in the category $\wf$, the morphism is
\[
\hom_{\wf}(L_{0}, L_{k}) = \hom_{\rmod Z_{\tilde{\dd}(F)}^{op}}(P_{\tilde{\dd}(F)}(-, L_{0}), P_{\tilde{\dd}(F)}(-, L_{k})).
\]
Moreover, Proposition \ref{prop: morphism in t = in zt} in fact implies that
the subcategory of $\rmod Z_{\tilde{\dd}(F)}^{op}$ whose objects are modules of the form $P_{\tilde{\dd}(F)}(-, L)$ 
is quasi-equivalent to the subcategory of $\rmod \tilde{\dd}(F)$ whose objects are the pullback right Yoneda modules $i^{*}Y^{r}_{L}$ for $L \in \ob \w(X) =  \ob \w(\bar{X}, F_{0})$,
which by definition is $ess-im(\w(\bar{X}, F_{0}) \to \r{Prop}\tilde{\dd}(F))$.
\end{proof}

Thus a prototype of the map \eqref{components of psi} can be constructed as follows.
We first define maps
\begin{equation}\label{maps to cc of d}
\Theta^{k}: \hom_{\n(\bar{X}, F_{0})}(L_{k-1}, L_{k}) \otimes \cdots \otimes \hom_{\n(\bar{X}, F_{0})}(L_{0}, L_{1}) \to \r{CC}^{*}(\tilde{\dd}(F)^{op}, \hom_{\K}(i^{*}Y^{r}_{L_{0}}, i^{*}Y^{r}_{L_{k}})),
\end{equation}
and compose them with the quasi-isomorphism
\begin{equation}\label{pullback of cc of d}
\r{CC}^{*}(\tilde{\dd}(F)^{op}, \hom_{\K}(i^{*}Y^{r}_{L_{0}}, i^{*}Y^{r}_{L_{k}})) \stackrel{\sim}\to \r{CC}^{*}(Z_{\tilde{\dd}(F)}, \hom_{\K}(P_{\tilde{\dd}(F)}(-, L_{0}), P_{\tilde{\dd}(F)}(-, L_{k})))
\end{equation}
as in \eqref{pullback of cc} to obtain maps \eqref{components of psi}.

\subsection{Moduli spaces of domains and maps}\label{section: moduli spaces of functor disks}

The maps \eqref{maps to cc of d} are constructed via the count of certain popsicles,
which are pseudoholomorphic maps from boundary-punctured disks (equipped with additional data) to the target stopped Liouville manifold subject to certain restricted types of Lagrangian boundary conditions.

We first describe the domains of these pseudoholomorphic maps.
Let 
\begin{equation}\label{domain of functor disk}
S = D^{2} \setminus \{z_{out} , z'_{l}, \ldots, z'_{1}, z_{in}, z_{1}, \ldots, z_{k} \}
\end{equation}
be a Riemann surface isomorphic to a disk with $k+l+2$ boundary points removed.
The punctures are in the counterclockwise order along $\p D^{2}$ as indicated in \eqref{domain of functor disk},
where we shall call $z_{out}$ the $0$-th puncture,
$z'_{j}$ the $(l-j+1)$-th puncture for $j=1, \ldots, l$, 
$z_{in}$ the $(l+1)$-th puncture, 
and $z_{i}$ the $(l+1+i)$-th puncture for $i = 1, \ldots, k$.

Let $C_{i}, C_{in}$ and $C'_{j}$ be the unique hyperbolic geodesic connecting $z_{i}, z_{in}$ and respectively $z'_{j}$ to $z_{out}$.

\begin{defn}
We say a popsicle structure $\sigma$ on $S$ of flavor $\bar{\mathbf{p}}$ is $\Psi$-{\it adapted}, 
if the sprinkles can only lie on the geodesics $C_{i}, i = 1, \ldots, k$.
\end{defn}

This flavor $\bar{\mathbf{p}}$ can be regarded as the extension of a flavor $\mathbf{p}: F \to \{1, \ldots, k\}$ by 
\begin{equation}\label{extended popsicle flavor}
\bar{\mathbf{p}}(f) = \mathbf{p}(f) + l + 1.
\end{equation}
Again, we shall only consider injective popsicles, meaning that $Aut(\bar{\mathbf{p}})$ is trivial.
In addition, we want to equip the punctures with non-positive integer weights
 $w_{out}$ for $z_{out}$, $w_{i}$ for $z_{i}$, $w_{in}$ for $z_{in}$ and $w'_{j}$ for $z'_{j}$, 
which should satisfy
\begin{equation}
w_{out} = w'_{j} = w_{in} = 0, \text{ and } w_{i} \in \{-1, 0\}, i = 1, \ldots, k.
\end{equation}
Given these weight conditions, the standard conditions \eqref{weight sum condition} and \eqref{sprinkles bounded by weight} on weighted popsicles become simplified to the following
\begin{equation}
|\bar{\mathbf{p}}^{-1}(i)| = -w_{i} \in \{0, 1\}, i = 1, \ldots, k.
\end{equation}
The strip-like ends for $S$ near the punctures are chosen compatibly with the weights according to the rules \eqref{strip end j}. 
In this case, the punctures $z'_{j}, j=1,\ldots,k$ and $z_{in}$ are positive punctures, equipped with positive strip-like ends
\[
\e_{in}/\e'_{j}: Z^{+} \to S,
\]
the puncture $z_{out}$ is a negative puncture, equipped with a negative strip-like end
\[
\e_{out}: Z^{-} \to S.
\]
For the other punctures $z_{i}$, the strip-like ends are
\[
\e_{i}: Z^{s_{i}} \to S,
\]
where the sign notation $s_{i}$ is defined in \eqref{sign notation}.

A description of degenerations of $\Psi$-adapted popsicle structures is given below, 
which is essentially part of Lemma 3.1 of \cite{GGV}, originally due to \cite{seidel6}.

\begin{lem}\label{boundary degeneration}
Consider a weighted $\Psi$-adapted popsicle structure $\sigma$ of flavor $\bar{\mathbf{p}}$ on $S$ as above. 
Then for any broken popsicle $((S_{1}, \sigma_{1}),(S_{2}, \sigma_{2}))$ of flavor $(\mathbf{p}_{1}, \mathbf{p}_{2})$ in a codimension one stratum of $\p \bar{\mathcal{R}}^{k+l+2, \bar{\mathbf{p}}}$, exactly one of the following holds:
\begin{enumerate}[label=(\roman*)]

\item both $\sigma_{1}$ and $\sigma_{2}$ are $\Psi$-adapted;

\item $\sigma_{2}$ is $\Psi$-adapted, and $\sigma_{1}$ is a trivial popsicle structure, i.e. one carrying no sprinkles;

\item $\sigma_{1}$ is $\Psi$-adapted, and $\sigma_{2}$ is a trivial popsicle structure;

\item neither $\sigma_{1}$ nor $\sigma_{2}$ is $\Psi$-adapted. In this case, there are two possible values for $i$ such that $\mathbf{p}_{2}^{-1}(i) $, while $\mathbf{p}_{1}^{-1}(k+l+2-i)$ consists of exactly one element.

\item $|\mathbf{p_{1}}^{-1}(i)| \ge 2$

\end{enumerate}
\end{lem}
\begin{proof}
A $\Psi$-adapted popsicle structure is a special case considered in Lemma 3.1 of \cite{GGV}, 
so a detailed track of the sprinkles goes through.
\end{proof}

Now let us consider maps from such popsicles to our target stopped Liouville manifold $(\bar{X}, F_{0})$. 
These are maps $u: S \to \bar{X}$ which are smooth in the interior of $S$, 
satisfying the equation \eqref{CR equation for rw} but with some modification on Lagrangian boundary conditions.
To write down the new equation precisely, we shall introduce some more notations.
We denote by $\p'_{l} S$ the boundary component of $\p S$ between $z_{out}$ and $z'_{l}$, 
$\p'_{j} S$ the boundary component of $\p S$ between $z'_{j}$ and $z'_{j-1}$ for $j = 2, \ldots, l$,
$\p'_{0} S$ the boundary component of $\p S$ between $z'_{l}$ and $z_{in}$, 
$\p_{0} S$ the boundary component of $\p S$ between $z_{in}$ and $z_{1}$,
$\p_{i} S$ the boundary component of $\p S$ between $z_{i}$ and $z_{i+1}$ for $i = 1, \ldots, k-1$,
and $\p_{k} S$ the boundary component of $\p S$ between $z_{k}$ and $z_{out}$.

Choose a Lagrangian label
\[
 L'_{l}, \ldots, L'_{0}, L_{0}, \ldots, L_{k},
\]
where $L_{i} \in \ob \w(\bar{X}, F_{0})$, and
\begin{equation}\label{testing Lagrangians in image of cup functor}
L'_{j} \in \ob \tilde{\dd}(F), j = 0, \ldots, l.
\end{equation}
Assigning $L_{i}$ to $\p_{i} S$ and $L'_{j}$ to $\p'_{j} S$,
we write down the corresponding inhomogeneous Cauchy-Riemann equation in the following form:

\begin{equation}\label{CR equation for functor disk}
\begin{cases}
&(du - X_{H_{S}} \otimes \beta_{S})^{0, 1} = 0, \\
&u(z) \in (\psi^{\rho_{S}(z)})^{*} L_{i}, \text{ for } z \in \p_{i} S, i = 0, \ldots, k, \\
&u(z) \in (\psi^{\rho_{S}(z)})^{*} L'_{j}, \text{ for } z \in \p'_{j} S, j = 0, \ldots, l, \\
&\lim\limits_{s \to s_{i} \infty} u \circ \epsilon_{i}(s, \cdot) = (\psi^{\nu_{i}})^{*} x_{i},i = 1, \ldots, k, \\
&\lim\limits_{s \to +\infty} u \circ \e_{in}(s, \cdot) = (\psi^{\nu_{in}})^{*} y_{in}, \\
&\lim\limits_{s \to +\infty} u \circ \e'_{j}(s, \cdot) = (\psi^{\nu'_{j}})^{*} x'_{j}, j = 1, \ldots, l, \\
&\lim\limits_{s \to - \infty} u \circ \epsilon_{out}(s, \cdot) = (\psi^{\nu_{out}})^{*} y_{out}.
\end{cases}
\end{equation}
In addition, we impose the following constraints on the intersection numbers of the asymptotic Hamiltonian chords with $F_{0} \times \R_{s>0}$:
\begin{align}
\label{intersection xj}  x_{i} \cdot (F_{0} \times \R_{s>0}) & 
\begin{cases}
= 0, & w_{i} = 0, \\
 > 0, & w_{i} = -1,
 \end{cases} i = 1, \ldots, k, \\ 
\label{intersection in} y_{in} \cdot (F_{0} \times \R_{s>0}) &= 0, \\ 
\label{intersection xi} x'_{j} \cdot (F_{0} \times \R_{s>0}) & = 0, j = 1, \ldots, l,\\ 
\label{intersection out} y_{out} \cdot (F_{0} \times \R_{s>0}) &= 0. 
\end{align}

Write
\begin{align}
\mathbf{x} & = (x_{1}, \ldots, x_{k}),\\
\mathbf{x}' &= (x'_{l}, \ldots, x'_{1})
\end{align}
for the collections of the Hamiltonian chords.
Denote by 
\begin{equation}\label{moduli space of functor disks}
\mathcal{R}^{k, l, \bar{\mathbf{p}}, \bar{\mathbf{w}}}_{2}(y_{out}, \mathbf{x}', y_{in}, \mathbf{x})
\end{equation}
the moduli space of $(S, \sigma, u)$ satisfying all the conditions \eqref{CR equation for functor disk} and \eqref{intersection xj}, \eqref{intersection in}, \eqref{intersection xi}, \eqref{intersection out}.
By definition, this is just a copy of the popsicle moduli spaces $\mathcal{R}^{k+l+2, \bar{\mathbf{p}}, \bar{\mathbf{w}}}(\bar{\mathbf{x}})$ as in \eqref{moduli space of popsicle maps},
\[
\mathcal{R}^{k, l, \bar{\mathbf{p}}, \bar{\mathbf{w}}}_{2}(y_{out}, \mathbf{x}', y_{in}, \mathbf{x}) = \mathcal{R}^{k+l+2, \bar{\mathbf{p}}, \bar{\mathbf{w}}}(\bar{\mathbf{x}}),
\]
where the weights $\bar{\mathbf{w}}$ are defined by
\begin{equation}\label{extended weights}
\bar{\mathbf{w}} = (\underbrace{0, \ldots, 0}_{l + 1}, w_{1}, \ldots, w_{k}),
\end{equation}
and the collection of chords $\bar{\mathbf{x}}$ is 
\[
\bar{\mathbf{x}} = (y_{out}, \mathbf{x}', y_{in}, \mathbf{x}).
\]
But we choose the new notation \eqref{moduli space of functor disks} to address the following important conditions:
\begin{enumerate}[label=(\roman*)]

\item  the Lagrangian boundary conditions are chosen in a special way as described by \eqref{testing Lagrangians in image of cup functor} and the second and third lines in \eqref{CR equation for functor disk};

\item the popsicle flavor $\bar{\mathbf{p}}$ satisfies \eqref{extended popsicle flavor} and the weights take the form \eqref{extended weights};

\item the asymptotic Hamiltonian chords must satisfy the constraints on their intersection numbers with $F_{0} \times \R_{s>0}$ given by \eqref{intersection xj}, \eqref{intersection in}, \eqref{intersection xi}, \eqref{intersection out}.

\end{enumerate}

\subsection{Constructing the functor $\Psi$}\label{section: construction of psi}

Now we use the moduli spaces introduced in \S\ref{section: moduli spaces of functor disks} to construct the maps \eqref{maps to cc of d}.
Take a moduli space $\mathcal{R}^{k, l, \bar{\mathbf{p}}, \bar{\mathbf{w}}}_{2}(y_{out}, \mathbf{x}', y_{in}, \mathbf{x})$ as in \eqref{moduli space of functor disks}, which is zero-dimensional when the degree of the chords satisfy
\begin{equation}\label{degree condition for zero dim moduli space}
\deg(y_{out}) = \deg(y_{in}) + \sum_{i=1}^{k} (-1)^{w_{i}} \deg(x_{i}) + \sum_{j=1}^{l} \deg(x'_{j}) + 1 - k - l -(n+1) \sum_{i=1}^{k} (w_{1} + \cdots + w_{k}).
\end{equation}
Again, recall from Remark \ref{rem on chords} that some of the chords $x_{i}$ with $w_{i}=-1$ are defined as time-one $H$-chords,
and their corresponding $(-H)$-chords $x_{i}^{-}$ have degree $\deg(x_{i}^{-}) = n - \deg(x_{i})$.
Thus, for any $i = 1, \ldots, k$ with $w_{i} = -1$, the two terms $(-1)^{w_{i}} \deg(x_{i}) = - \deg(x_{i})$ and $-(n+1)w_{i} = n+1$ add up to $n+1-\deg(x_{i}) = \deg(x_{i}^{-}) + 1$,
which is the degree of the generator $[x_{i}^{-}]$ for the shifted complex $CW^{*}(L_{i-1}, L_{i}; -H)[1]$ in the mapping cone $NC^{*}_{F}(L_{i-1}, L_{i})$ as in \eqref{morphism in n}.

Each rigid element $u \in \mathcal{R}^{k, l, \bar{\mathbf{p}}, \bar{\mathbf{w}}}_{2}(y_{out}, \mathbf{x}', y_{in}, \mathbf{x})$ induces an isomorphism of orientation lines
\begin{equation}
\mathcal{R}^{k, l, \bar{\mathbf{p}}, \bar{\mathbf{w}}}_{2, u}: o_{\mathbf{x}}^{\mathbf{s}}   \otimes o_{y_{in}} \otimes o_{\mathbf{x}'} \stackrel{\cong}\to o_{y_{out}},
\end{equation}
where 
\begin{equation}
o_{\mathbf{x}'} = o_{x'_{1}} \otimes \cdots \otimes o_{x'_{l}},
\end{equation}
and
\begin{equation}
o_{\mathbf{x}}^{\mathbf{s}} = o_{x_{k}}^{s_{k}} \otimes \cdots \otimes o_{x_{1}}^{s_{1}}.
\end{equation}
For convenience, introduce the following notation
\begin{equation}
CW^{*}_{F, s_{i}}(L_{i-1}, L_{i}; s_{i}H) =
\begin{cases}
CW^{*}_{F}(L_{i-1}, L_{i}; H), & \text{ if } s_{i} = +, \text{ i.e. } w_{i} = 0, \\
CW^{*}(L_{i-1}, L_{i}; -H), & \text{ if } s_{i} = -, \text{ i.e. } w_{i} = -1.
\end{cases}
\end{equation}
We define a map
\begin{equation}\label{theta kl}
\begin{split}
\Theta^{k, l, \bar{\mathbf{p}}, \bar{\mathbf{w}}}: &  CW^{*}_{F, s_{k}}(L_{k-1}, L_{k}; s_{k}H) \otimes \cdots \otimes CW^{*}_{F, s_{k}}(L_{0}, L_{1}; s_{1}H)  \otimes CW^{*}_{F}(L'_{0}, L_{0}; H) \\
& \otimes CW^{*}_{F}(L'_{l}, L'_{l-1}; H) \otimes \cdots \otimes CW^{*}_{F}(L'_{1}, L'_{0}; H) \to CW^{*}_{F}(L'_{l}, L_{k}; H)
\end{split}
\end{equation}
by defining it on basis elements as
\begin{equation}\label{theta on basis}
\begin{split}
& \Theta^{k, l, \bar{\mathbf{p}}, \bar{\mathbf{w}}}([x_{k}^{s_{k}}] \otimes \cdots \otimes [x_{1}^{s_{1}}] \otimes [y_{in}] \otimes [x'_{1}] \otimes \cdots \otimes [x'_{l}])= \\
& \prod_{\substack{y_{out} \\ \eqref{degree condition for zero dim moduli space}}} 
\sum_{u \in \mathcal{R}^{k, l, \bar{\mathbf{p}}, \bar{\mathbf{w}}}_{2}(y_{out}, \mathbf{x}', y_{in}, \mathbf{x})} 
(-1)^{*_{k+l+2, \bar{\mathbf{p}}, \bar{\mathbf{w}}} + \Diamond_{k+l+2, \bar{\mathbf{p}}, \bar{\mathbf{w}}}}
 \mathcal{R}^{k, l, \bar{\mathbf{p}}, \bar{\mathbf{w}}}_{2, u}([x_{k}^{s_{k}}] \otimes \cdots \otimes [x_{1}^{s_{1}}] \otimes [y_{in}] \otimes [x'_{1}] \otimes \cdots \otimes [x'_{l}])
\end{split}
\end{equation}
where the signs $*_{k+l+2, \bar{\mathbf{p}}, \bar{\mathbf{w}}}$ and $\Diamond_{k+l+2, \bar{\mathbf{p}}, \bar{\mathbf{w}}}$ follow the general formulae \eqref{sign formula 1 for popsicles} and \eqref{sign formula 2 for popsicles}.

\begin{prop}
The operation \eqref{theta on basis} extends to a well-defined map \eqref{theta kl}.
\end{prop}
\begin{proof}
We need to show that the map is well-defined on the infinite direct product \eqref{Floer chains}.
This will immediately follow from the next lemma, which asserts that $\Psi^{k, l}$ in fact vanishes except on finitely many basis elements.
\end{proof}

\begin{lem}
The moduli space $\mathcal{R}^{k, l, \bar{\mathbf{p}}, \bar{\mathbf{w}}}_{2}(y_{out}, \mathbf{x}', y_{in}, \mathbf{x})$ is empty unless every chord $x_{i}, i = 1, \ldots, k$ has zero intersection number with $F_{0} \times \R_{s>0}$.
\end{lem}
\begin{proof}
This follows from the constraints \eqref{intersection xj}, \eqref{intersection in}, \eqref{intersection xi}, \eqref{intersection out},
as well as \eqref{intersection number identity} in Lemma \ref{lem: intersection number identity}.
\end{proof}

Then we combine all the maps \eqref{theta kl} to get a map of the following form
\begin{equation}
\begin{split}
\Theta^{k, l}: & NC^{*}_{F}(L_{k-1}, L_{k}) \otimes \cdots \otimes NC^{*}_{F}(L_{0}, L_{1}) \otimes CW^{*}_{F}(L'_{0}, L_{0}; H) \\
& \otimes  CW^{*}_{F}(L'_{l}, L'_{l-1}; H) \otimes \cdots \otimes CW^{*}_{F}(L'_{1}, L'_{0}; H) \to CW^{*}_{F}(L'_{l}, L_{k}; H).
\end{split}
\end{equation}
Equivalently, we may write it as a map 
\begin{equation}
\begin{split}
\Theta^{k, l}: & NC^{*}_{F}(L_{k-1}, L_{k}) \otimes \cdots \otimes NC^{*}_{F}(L_{0}, L_{1}) \\
& \to \hom_{\K}(CW^{*}_{F}(L'_{l}, L'_{l-1}; H) \otimes \cdots \otimes CW^{*}_{F}(L'_{1}, L'_{0}; H), \\
&\hom_{\K}(CW^{*}_{F}(L'_{0}, L_{0}; H), CW^{*}_{F}(L'_{l}, L_{k}; H))).
\end{split}
\end{equation}
Now we take the direct product over all possible $l$-tuples of testing objects $L'_{0}, \ldots, L'_{l}$ and over $l$ to obtain the desired map $\Theta^{k}$ \eqref{maps to cc of d}.
Next, we compose the map $\Theta^{k}$ \eqref{maps to cc of d} with the quasi-isomorphism \eqref{pullback of cc of d} to get the map $\Psi^{k}$ \eqref{components of psi}.
It is by now a standard gluing-degeneration argument to check:

\begin{prop}
The maps $\Theta^{k}$ \eqref{components of psi} satisfy the $\ainf$-functor equations, so do $\Psi^{k}$.
\end{prop}
\begin{proof}
Since the $\ainf$-structure maps as well as the operations $\Theta^{k, l}$ are defined by counting rigid elements in appropriate popsicle moduli spaces, 
we can prove the $\ainf$-functor equations by studying boundary strata of the compactifications $\bar{\mathcal{R}}^{k, l, \bar{\mathbf{p}}, \bar{\mathbf{w}}}_{2}(y_{out}, \mathbf{x}', y_{in}, \mathbf{x})$ of one-dimensional moduli spaces.
This type of argument is quite standard and follows a pattern similar to the proofs of Lemma 5.3 of \cite{GGV}, 
so we will only outline below the correspondence between the algebraic terms and the strata of the moduli spaces.
More precisely, the first type of boundary strata (i) described in Lemma \ref{boundary degeneration} contribute exactly to terms of the forms $\Theta^{k_{2}} \circ \Theta^{k_{1}}$.
This composition is understood as the composition of Hochschild cochains with values in $\hom_{\K}$ between pullback Yoneda modules $\r{CC}^{*}(\tilde{\dd}(F)^{op}, \hom_{\K}(i^{*}(-), i^{*}(-))$, 
which by definition is by evaluating the Hochschild cochains on testing objects,
 taking composition in the space of linear homomorphisms $\hom_{\K}$ in all possible ways and summing up.
 In other words, this is the composition in the module category $\rmod \tilde{\dd}(F)$, which is a dg category (which also explains why there are no higher $\ainf$-compositions acting on outputs of $\Theta^{k}$'s).
The second and the third types (ii) and (iii) contribute to terms of the form $\mu^{1} \circ \Theta^{k-1}(\ldots)$ and $\Theta^{k_{1}}(\ldots, \mu^{k_{2}}(\ldots), \ldots)$ respectively.
Contributions from strata of the fourth type (iv) algebraically cancel,
as those two different kinds subject to the two possible values for the index $i$ are counted in the opposite signs due to the sign conventions \eqref{sign formula 1 for popsicles} and \eqref{sign formula 2 for popsicles}.
The last type has no contribution because only injective popsicles are counted.
Thus these maps form a contravariant $\ainf$-functor $\Theta$ from $\n(\bar{X}, F_{0})$ to $\rmod \tilde{\dd}(F)$.

By Proposition \ref{prop: morphism in t = in zt}, the map \eqref{pullback of cc} is part of a (partially defined) functor with vanishing higher order terms;
in particular, it respects the differentials and products on $\rmod \tilde{\dd}(F)$ and $\rmod Z_{\tilde{\dd}(F)}^{op}$.
It follows that $\Psi^{k}$ also satisfy the $\ainf$-functor equations and thus form an $\ainf$-functor $\n(\bar{X}, F_{0}) \to \rmod Z_{\tilde{\dd}(F)}^{op}$.
\end{proof}

Finally, by construction the image of $\Psi$ in $\rmod Z_{\tilde{\dd}(F)}^{op}$ lands in the essential image of $F_{\tilde{\dd}(F)}: \w(\bar{X}, F_{0}) \to \wf$ as in \eqref{module-valued functor},
it follows that $\Psi$ defines a functor \eqref{functor to completion}.

\subsection{An algebraic short exact sequence}\label{section: algebraic ses}

We will first prove Theorem \ref{thm:equivalence} under the additional assumption that $\bar{X}$ is a non-degenerate Liouville manifold in \S\ref{section: algebraic ses}, \S\ref{section: geometric ses} and \S\ref{section: cd}.
This argument relies on an existing result by \cite{GGV} interpreting the Rabinowitz Fukaya category $\rw(\bar{X})$ as the categorical formal punctured neighborhood of infinity of the fully wrapped Fukaya category $\w(\bar{X})$.
We will also provide an alternative proof in \S\ref{section: n as limit} which is not only independent of that result, 
but also makes our conclusion stronger as non-degeneracy condition is not required therein.

\begin{thm}[Theorem 1.1 of \cite{GGV}]\label{rw=winf}
For any Liouville manifold $\bar{X}$, there is a canonical $\ainf$-functor
\begin{equation}\label{equivalence of rw}
\Phi: \rw(\bar{X}) \to \widehat{\w(\bar{X})}_{\infty},
\end{equation}
which is a quasi-equivalence whenever $\bar{X}$ is non-degenerate.
\end{thm}

An equivalent statement of the stop removal formula, Theorem \ref{thm: stop removal}, is that there is a short exact sequence of $\ainf$-categories
\begin{equation}
\mathcal{D}(\fk) \xhookrightarrow{} \w(\bar{X}, \fk) \to \w(\bar{X}),
\end{equation}
whenever the core $\fk$ of $F$ is mostly Legendrian.
For a general hypersurface $F$, 
we need the existence of such a short exact sequence 
\begin{equation}\label{stop removal ses}
\tilde{\dd}(F) \xhookrightarrow{} \w(\bar{X}, F_{0}) \to \w(\bar{X})
\end{equation}
as a working assumption - this is exactly the stop removal condition.

Now suppose that the stop is also a full stop.
Non-degeneracy of $\bar{X}$ and $F$, together with stop removal \eqref{stop removal ses} in the hypotheses of Theorem \ref{thm:equivalence} implies that $\w(X) = \w(\bar{X}, F_{0})$ is (homologically) smooth.
By Corollary \ref{unique fit into ses}, this gives rise to a short exact sequence
\begin{equation}\label{algebraic ses}
\tilde{\dd}(F) \xhookrightarrow{} \wf \to \widehat{\w(\bar{X})}_{\infty}.
\end{equation}
The first functor $\tilde{\dd}(F) \xhookrightarrow{} \wf$ supplied by \eqref{k functor}, 
which is the restriction of the natural functor $\w(X) \to \wf$ to the subcategory $\tilde{\dd}(F)$, 
is fully faithful by Lemma \ref{property of k functor}.
The second functor $\wf \to \widehat{\w(\bar{X})}_{\infty}$ is a projection on objects,
 sending each object $L \in \ob \wf = \ob \w(\bar{X}, F_{0})$ to its image (i.e. isomorphism class) in $\widehat{\w(\bar{X})}_{\infty}$, 
 and acts on morphisms following the procedure described in Theorem \ref{quotient of formal completion} (or rather its proof) in \S\ref{section: relation to formal punctured neighborhood}.

In the case of a Lefschetz fibration $\pi: \bar{X} \to \C$ with Weinstein fibers $F$,
the Fukaya-Seidel category $\mathcal{FS}(\bar{X}, \pi) = \w(\bar{X}, F_{0})$ is generated by the Lefschetz thimbles by \cite{GPS2},
which implies that the wrapped Fukaya category $\w(\bar{X})$ of the total space is also generated by the Lefschetz thimbles,
and therefore that $\bar{X}$ is non-degenerate. 
Non-degeneracy of the fiber $F$ follows from the Weinstein condition: $\w(F)$ is generated by cocore disks.

\subsection{A geometric short exact sequence}\label{section: geometric ses}

The short exact sequence \eqref{algebraic ses} in the previous subsection has the following geometric counterpart:
\begin{equation}\label{geometric ses}
\tilde{\dd}(F) \xhookrightarrow{} \n(\bar{X}, F_{0}) \to \rw(\bar{X})
\end{equation}
The first functor $\tilde{\dd}(F) \to \n(\bar{X}, F_{0})$ is the composition of the inclusion of the subcategory $\tilde{\dd}(F) \xhookrightarrow{} \w(\bar{X}, F_{0})$ and the functor $j: \w(\bar{X}, F_{0}) \to \n(\bar{X}, F_{0})$ as in \eqref{negative wrapping functor}.
The second functor 
\begin{equation}\label{acceleration at infinity}
\iota_{\sharp}: \n(\bar{X}, F_{0}) \to \rw(\bar{X}) = \n(\bar{X}, \varnothing)
\end{equation}
is the pushforward functor \eqref{pushforward on n} induced by the inclusion of stopped Liouville manifolds $\iota: (\bar{X}, F_{0}) \xhookrightarrow{} (\bar{X}, \varnothing)$.

\begin{prop}
The sequence $\tilde{\dd}(F) \xhookrightarrow{} \n(\bar{X}, F_{0}) \to \rw(\bar{X})$ \eqref{geometric ses} is exact.
\end{prop}
\begin{proof}
First, we want to show that $\tilde{\dd}(F) \to \n(\bar{X}, F_{0})$ is fully faithful.
Since $\tilde{\dd}(F) \xhookrightarrow{} \w(\bar{X}, F_{0})$ is a full subcategory, 
it suffices to show that the functor $j: \w(\bar{X}, F_{0}) \to \n(\bar{X}, F_{0})$ is fully faithful on the subcategory $\tilde{\dd}(F)$.
This is based on the following observation:
any Lagrangian $L \in \ob \tilde{\dd}(F)$ is by definition in the image of the cup functor $i_{F}: \w(F) \to \w(\bar{X}, F_{0})$ \eqref{cup functor},
which is contained in arbitrary small neighborhood of the stop at infinity,
and therefore a nonessential Lagrangian in the sense of Definition \ref{defn: nonessential Lagrangian}.
So by Proposition \ref{faithful on nonessential Lagrangians},
the functor $j: \w(\bar{X}, F_{0}) \to \n(\bar{X}, F_{0})$ is fully faithful on objects in $\tilde{\dd}(F)$.

Second, it is easy to see that any object $L \in \ob \tilde{\dd}(F)$ is mapped to zero in $\rw(\bar{X})$.
Since $L$ is a nonessential Lagrangian,
by Lemma \ref{nonessential Lagrangian zero in fully wrapped}, it is a zero object in the fully wrapped Fukaya category $\w(\bar{X})$,
and therefore also a zero object in $\rw(\bar{X})$.

It remains to prove that if $L \in \ob \n(\bar{X}, F_{0})$ is mapped to a zero object in $\rw(\bar{X})$, 
then it is isomorphic to an object in $\r{Tw}\tilde{\dd}(F)$.
While we expect there should be a proof in the same spirit as that of Theorem \ref{thm: stop removal} as in \cite{GPS2},
that may not be directly applicable here due to the difference in the analytic setup of Fukaya categories.
On the other hand, we do not have the assumption that objects of $\tilde{\dd}(F)$ are Lagrangian linking disks,
so that we lack a simple description of `wrapping through the stop'.
For that reason, we shall give an alternatively proof, 
and the result is obtained only up to twisted complexes in $\tilde{\dd}(F)$ (and that is exactly the meaning of exactness of the sequence of categories).

Since $L$ is mapped to a zero object in $\rw(\bar{X})$, 
it defines a proper module over $\w(\bar{X})$.
If $\w(\bar{X}) = 0$ then the short exact sequence \eqref{stop removal ses} immediately implies that $L$ is an object of $\tilde{\dd}(F)$.
Now we suppose $\w(\bar{X}) \neq 0$.
Since a connected non-compact Lagrangian cannot define a proper module over $\w(\bar{X})$ unless it is a zero object (\cite{GGV}, \cite{Gnonproper}),
it follows that $L$ is isomorphic to the disjoint union of connected closed exact Lagrangians and connected non-compact nonessential Lagrangians (the latter of which are zero objects in $\w(\bar{X})$).
The above isomorphism holds in $\w(\bar{X})$, so by stop removal short exact sequence \eqref{stop removal ses}, passing to an isomorphism in $\w(X)$ one must further add a twisted complex in $\r{Tw}\tilde{\dd}(F)$.
Now observe that for any closed exact Lagrangian $K$ we have $CW^{*}(K, K; \pm H) = CW^{*}_{int}(K, K; \pm H)$, 
where the complexes of interior chords are defined in \eqref{subcomplex of interior chords}.
Furthermore, the continuation map $c: CW^{*}(K, K; -H) \to CW^{*}(K, K; H)$ \eqref{continuation} is a chain homotopy equivalence (corresponding to the Poincar\'{e} duality between chains and cochains for the closed manifold $K$).
It follows that $K$ is a zero object in $\n(\bar{X}, F_{0})$ as well.
And for any nonessential Lagrangian $L'$ which is not in $\r{Tw}\tilde{\dd}(F)$, it is zero in $\w(\bar{X}, F_{0})$ because of stop removal short exact sequence \eqref{stop removal ses}.
Combining these, we see that $L$ must be isomorphic to an object in $\r{Tw}\tilde{\dd}(F)$.

\end{proof}

\subsection{A commutative diagram}\label{section: cd}

Combining the two exact sequences \eqref{algebraic ses} and \eqref{geometric ses}, 
we obtain a diagram, for which we would like to show:

\begin{prop}\label{commutative}
Under the assumptions of Theorem \ref{thm:equivalence}, the following diagram
\begin{equation}\label{comm cd of ses}
\begin{tikzcd}
\tilde{\dd}(F) \arrow[r] \arrow[d, "\id"] & \n(\bar{X}, F_{0}) \arrow{r}{\iota_{\sharp}} \arrow[d, "\Psi"] & \rw(\bar{X}) \arrow[d, "\Phi"] \\
\tilde{\dd}(F) \arrow[r] & \wf \arrow{r}{\hat{\pi}} & \widehat{\w(\bar{X})}_{\infty}.
\end{tikzcd}
\end{equation}
commutes up to homotopy.
\end{prop}

The main technical difficulty to prove Proposition \ref{commutative} is to show that the right square commutes up to homotopy,
i.e. that the two compositions $\Phi \circ \iota_{\sharp}$ and $\hat{\pi} \circ \Psi$ are homotopic $\ainf$-functors.
On the level of objects, the functor $\Phi$ sends every $L \in \ob \rw(\bar{X})$ it its equivalence class in 
\begin{equation}\label{winf as calk}
\widehat{\w(\bar{X})}_{\infty} = ess-im(\w(\bar{X}) \to \fun(\w(\bar{X})^{op}, \calk_{\K})),
\end{equation}
which under the identification \eqref{winf as prop d} is the module $i^{*}Y^{r}_{L}$. 
It follows that $\Phi \circ \iota_{\sharp}(L) = \hat{\pi} \circ \Psi(L)$.

First we would like to conceptually understand why such a diagram should commute up to homotopy.
We begin by recalling that by by \eqref{binf as prop t} in Lemma \ref{lem: binf as prop t}, the formal punctured neighborhood of infinity can be alternatively identified with
\begin{equation}\label{winf as prop d}
\widehat{\w(\bar{X})}_{\infty} \cong ess-im(\bar{r}_{\tilde{\dd}(F)}: \w(\bar{X}) \to \r{Prop} \tilde{\dd}(F) / \perf \tilde{\dd}(F))
\end{equation}
in the presence of the Liouville pair $(\bar{X}, F)$.
Such an identification is more convenient for the purpose of studying the diagram \eqref{comm cd of ses},
compared to the general case \eqref{winf as calk} used in \cite{GGV}.
Under this identification, the morphisms in $\widehat{\w(\bar{X})}_{\infty}$ can be explicitly computed in the quotient category of $\r{Prop} \tilde{\dd}(F)$.
Note that in $\r{Prop} \tilde{\dd}(F)$, the morphism spaces between pullbacks of Yoneda modules of Lagrangians (which are objects of $\w(\bar{X})$) are given by Hochschild cochain complexes:
\begin{equation}\label{hom in prop}
\hom_{\r{Prop} \tilde{\dd}(F)}(i^{*}Y^{r}_{K}, i^{*}Y^{r}_{L}) = \hom_{\rmod \tilde{\dd}(F)}(i^{*}Y^{r}_{K}, i^{*}Y^{r}_{L}) = \r{CC}^{*}(\tilde{\dd}(F)^{op}, \hom_{\K}(i^{*}Y^{r}_{K}, i^{*}Y^{r}_{L})).
\end{equation}
These are precisely the domains of the quasi-isomorphism \eqref{pullback of cc of d},
and also the target of the maps $\Theta^{k}$ \eqref{maps to cc of d} via which the maps $\Psi^{k}$ \eqref{components of psi} constituting the functor $\Psi$ are defined.
On the other hand, by Lemma \ref{lem: wf as prop d}, there is a quasi-equivalence
\[
\wf \cong ess-im(\w(\bar{X}, F_{0}) \to \r{Prop} \tilde{\dd}(F)).
\]
From this point of view, the functor $\hat{\pi}$ is simply induced by the projection to the quotient category $\r{Prop} \tilde{\dd}(F) \to \r{Prop} \tilde{\dd}(F) / \perf \tilde{\dd}(F)$ having only a first-order term on each morphism space,
and the composition $\hat{\pi} \circ \Psi$ is a Yoneda-type functor since by the construction in \S\ref{section: moduli spaces of functor disks}, \ref{section: construction of psi}, $\Psi$ is.
On the other hand, the functor $\Phi$ constructed in \cite{GGV} is also essentially a Yoneda-type functor defined by counting disks of similar kinds, and $\iota_{\sharp}$ is just the acceleration functor which by construction also has only nonzero first-order terms.
Thus it is a straightforward observation that these two composite functors should agree up to homotopy.

The only matter that remains to be resolved is that the functor $\Phi$ constructed in \cite{GGV} goes to \eqref{winf as calk} while the functor $\Psi$ has target \eqref{winf as prop d}, 
although both of these are quasi-equivalent models for the same category $\widehat{\w({\bar{X})}}_{\infty}$.
Recall that the identification of these two models of $\widehat{\w({\bar{X})}}_{\infty}$ is supplied by Lemma \ref{lem: binf as prop t}.

\begin{lem}\label{linear identification winf=propd}
The identification \eqref{winf as prop d} is induced by an $\ainf$-functor with only nonzero first-order terms.
\end{lem}
\begin{proof}
Note that in the proof of Lemma \ref{lem: binf as prop t}, the identification is obtained by tracing a commutative diagram of short exact sequences of categories obtained by applying $\bar{\calk}$ and $\fun(-, \calk_{\K})$ to the short exact sequence $\tilde{\dd}(F) \to \w(\bar{X}, F_{0}) \to \w(\bar{X})$.
Since in our situation the functor $\tilde{\dd}(F) \to \w(\bar{X}, F_{0})$ is simply the inclusion of a full subcategory, 
whose first-order terms on morphism spaces are identity maps and higher-order terms all zero,
it follows that the identification \eqref{winf as prop d} is also given by a functor with the same property.
\end{proof}

We have gathered all the necessary ingredients for proving Proposition \ref{commutative}.

\begin{proof}[Proof of Proposition \ref{commutative}]
The left square is easily seen to be commutative up to homotopy, 
since the restriction of $\Psi$ \eqref{functor to completion} to $\tilde{\dd}(F)$ is homotopic to the identity by Proposition \ref{prop: morphism in t = in zt} and Lemma \ref{property of k functor}.

Now let us look at the right square. 
Recall that the acceleration functor $\iota_{\sharp}$ and the projection to quotient functor $\hat{\pi}$ have only nonzero first-order terms and vanishing higher-order terms.
On each morphism space, $\iota_{\sharp}$ is an inclusion map of a subcomplex by Proposition \ref{prop: pushforward on n} (and also by its proof).
This implies that in the composition $\Psi \circ \iota_{\sharp}$, 
the only nontrivial terms are $\Psi^{k}(\iota_{\sharp}^{1}(\cdot), \ldots, \iota_{\sharp}^{1}(\cdot))$.
On the chain level, $\iota_{\sharp}^{1}$ sends each generator in $\hom_{\n(\bar{X}, F_{0})} = NC^{*}_{F}$ to the corresponding generator in $\hom_{\rw(\bar{X})} = RC^{*}$ that is represented by the same chain.

To see what $\hat{\pi}$ does on the chain level, note that its first-order term $\hat{\pi}^{1}$ has source  $\hom_{\r{Prop} \tilde{\dd}(F)}(Y^{r}_{L_{0}}, Y^{r}_{L_{k}})$.
To have a simple description of the target morphism space in the quotient category $\r{Prop} \tilde{\dd}(F)/\perf \tilde{\dd}(F)$, 
recall that $\tilde{\dd}(F)$, being a full subcategory of the proper category $\w(\bar{X}, F_{0})$, is proper.
In particular, all Yoneda modules of $\tilde{\dd}(F)$ are proper modules.
Since the Yoneda image of $\tilde{\dd}(F)$ split generates $\perf \tilde{\dd}(F)$ in the category $\rmod \tilde{\dd}(F)$, in particular in proper modules $\r{Prop} \tilde{\dd}(F)$, 
the natural map 
\[
\r{Prop} \tilde{\dd}(F)/ y(\tilde{\dd}(F)) \to \r{Prop} \tilde{\dd}(F)/\perf \tilde{\dd}(F)
\]
 is a quasi-isomorphism. 
So we may identify the target of $\hat{\pi}^{1}$ with
\begin{equation}\label{hom in propd/perfd}
\begin{split}
& \hom_{\r{Prop} \tilde{\dd}(F)/ y(\tilde{\dd}(F))}(Y^{r}_{L_{0}}, Y^{r}_{L_{k}}) \\
=  \bigoplus_{\substack{p \ge 0\\ L'_{1}, \ldots, L'_{p} \in \tilde{\dd}(F)}} & \hom_{\r{Prop} \tilde{\dd}(F)}(Y^{r}_{L'_{p}}, Y^{r}_{L_{k}}) \otimes \hom_{\r{Prop} \tilde{\dd}(F)}(Y^{r}_{L'_{p-1}}, Y^{r}_{L'_{p}}) \otimes \cdots \\
& \otimes \hom_{\r{Prop} \tilde{\dd}(F)}(Y^{r}_{L'_{1}}, Y^{r}_{L'_{2}}) \otimes \hom_{\r{Prop} \tilde{\dd}(F)}(Y^{r}_{L_{0}}, Y^{r}_{L'_{1}}),
\end{split}
\end{equation}
so that $\hat{\pi}^{1}$ is an inclusion map that sends $\hom_{\r{Prop} \tilde{\dd}(F)}(Y^{r}_{L_{0}}, Y^{r}_{L_{k}})$ to the $p=0$ piece of \eqref{hom in propd/perfd},
where it is the identity map.
Lemma \ref{linear identification winf=propd} then implies that it suffices to compare the maps constituting $\Psi$ and $\Phi$ if we use the identification \eqref{winf as prop d} to identify the two models \eqref{winf as calk}, \eqref{winf as prop d} of $\widehat{\w({\bar{X})}}_{\infty}$.

Recall from \eqref{maps to cc of d}, \eqref{pullback of cc of d} that the maps $\Psi^{k}$ are defined via the maps $\Theta^{k}$, 
which are constructed by counting rigid disks in the moduli spaces introduced in \S\ref{section: moduli spaces of functor disks}.
On the level of moduli spaces, the effect of applying the map $\iota_{\sharp}^{1}$ to the inputs $x_{i}$ is that these inputs are naturally considered as inputs from the morphism spaces in the category $\rw(\bar{X})$,
so that under the model \eqref{winf as prop d} for $\widehat{\w({\bar{X})}}_{\infty}$,
these disks are the same as those for defining the maps $\Phi^{k}$.
It follows that the right square strictly commutes on the chain level.
\end{proof}

Now we are ready to finish the proof of Theorem \ref{thm:equivalence}.

\begin{proof}[Proof of Theorem \ref{thm:equivalence}]
The conclusion immediately follows from the two short exact sequences \eqref{algebraic ses} and \eqref{geometric ses}, Theorem \ref{rw=winf} and the commutative diagram \eqref{comm cd of ses} in Proposition \ref{commutative}.
\end{proof}

\subsection{Computing the functor via limits}\label{section: n as limit}

This subsection is devoted to proving Theorem \ref{thm:main2}.
Importantly, it drops the non-degeneracy assumption compared to Theorem \ref{thm:equivalence},
which suggests that this type of quasi-equivalence between the category $\n(\bar{X}, F_{0})$ and the categorical formal completion has its intrinsic meaning,
 and is in particular independent of generation of wrapped Fukaya categories.
The approach will be very different from the one throughout \S\ref{section: functor overview}-\S\ref{section: cd}.
The assumptions under which we are going to work are:
\begin{enumerate}[label=(\roman*)]

\item $F$ satisfies stop removal for $(\bar{X}, F)$,

\item $\p_{\infty} \bar{X}$ has bounded Reeb dynamics relative to $F$, and

\item $\w(\bar{X}, F_{0})$ is smooth.

\end{enumerate}

Recall from \S\ref{sec: various functors} that there are natural $\ainf$-functors $j_{m}: \w(\bar{X}, F_{0}) \to \n_{>m}(\bar{X}, F_{0})$, $\pi_{m}: \n(\bar{X}, F_{0}) \to \n_{>m}(\bar{X}, F_{0})$  for all $m \le -1$,
 and $\pi_{m', m}: \n_{>m'}(\bar{X}, F_{0}) \to \n_{>m}(\bar{X}, F_{0})$
 for all $m'<m$. 
By Lemma \ref{compatible diagram of functors}, we have a diagram of $\ainf$-categories
\[
\{\n_{>m}(\bar{X}, F_{0})\}_{m \in \Z_{\le -1}}
\]
as well as a strictly compatible diagram of functors 
\[
\{j_{m}: \w(\bar{X}, F_{0}) \to \n_{>m}(\bar{X}, F_{0})\}_{m \in \Z_{\le -1}}.
\]
Our first task is to show that these functors satisfy the conditions of Lemma \ref{formal completion as limits}.

\begin{prop}\label{prop: map from colimit of Floer complexes}
Suppose $\p_{\infty} \bar{X}$ has bounded Reeb dynamics relative to $F$.
Also suppose $\w(\bar{X}, F_{0})$ is smooth.
The for any nonessential Lagrangian $K$ and any arbitrary exact cylindrical Lagrangian $L$,
both of which are disjoint from $F_{0}$ at infinity, 
the following natural map
\begin{equation}\label{map from colimit of Floer complexes}
\hocolim_{m \in \Z_{\le -1}} \hom_{\rmod \w(\bar{X}, F_{0})}(j_{m}^{*} Y^{r}_{j_{m}(L)}, Y^{r}_{K}) \to \hom_{\rmod \w(\bar{X}, F_{0})}(Y^{r}_{L}, Y^{r}_{K}).
\end{equation}
is a quasi-isomorphism of chain complexes.
\end{prop}
\begin{proof}
The formal left adjoint $j_{m}^{*}: \rmod \n_{>m}(\bar{X}, F_{0}) \to \rmod \w(\bar{X}, F_{0})$ of the functor $j_{m}: \w(\bar{X}, F_{0}) \to \n_{>m}(\bar{X}, F_{0})$ sends proper modules to proper modules.
Since $\p_{\infty} \bar{X}$ has bounded Reeb dynamics relative to $F$, the category $\w(\bar{X}, F_{0})$ is proper. 
By our assumption, $\w(\bar{X}, F_{0})$ is also smooth.
It follows that $\r{Prop} \w(\bar{X}, F_{0}) = \perf \w(\bar{X}, F_{0})$, and that the formal adjoint has image in the category of perfect complexes:
\begin{equation}
j_{m}^{*, \r{Prop}}: \r{Prop} \n_{>m}(\bar{X}, F_{0}) \to \r{Prop} \w(\bar{X}, F_{0}) = \perf \w(\bar{X}, F_{0}).
\end{equation}
Since a Yoneda module $Y^{r}_{j_{m}(L)}$ for $L \in \ob \w(\bar{X}, F_{0})$ is a perfect module over $\n_{>m}(\bar{X}, F_{0})$, 
and $\n_{>m}(\bar{X}, F_{0})$ is proper on the chain level,
it follows that $Y^{r}_{j_{m}(L)}$ is a proper module over $\n_{>m}(\bar{X}, F_{0})$,
 so that $j_{m}^{*, \r{Prop}} Y^{r}_{j_{m}(L)}$ is a well-defined perfect module over $\w(\bar{X}, F_{0})$ whose Yoneda module is $j_{m}^{*} Y^{r}_{j_{m}(L)}$.
By Yoneda lemma extended to perfect complexes 
\[
\perf \w(\bar{X}, F_{0}) \to \rmod \perf \w(\bar{X}, F_{0}) = \rmod \w(\bar{X}, F_{0}),
\]
we obtain a quasi-isomorphism of chain complexes
\begin{equation}\label{yoneda with jm}
\hom_{\perf \w(\bar{X}, F_{0})}(j_{m}^{*, \r{Prop}} Y^{r}_{j_{m}(L)}, K) \stackrel{\sim}\to  \hom_{\rmod \w(\bar{X}, F_{0})}(j_{m}^{*} Y^{r}_{j_{m}(L)}, Y^{r}_{K}).
\end{equation}
Note that the Yoneda natural transformation (with respect to each $\pi_{m', m}$) is contravariant, so taking the homotopy limit of the left-hand side yields the homotopy colimit of the right-hand side:
\begin{equation}\label{lim to colim}
 \holim_{m \in \Z_{\le -1}^{op}} \hom_{\perf \w(\bar{X}, F_{0})}(j_{m}^{*, \r{Prop}} Y^{r}_{j_{m}(L)}, K) \stackrel{\sim}\to \hocolim_{m \in \Z_{\le -1}} \hom_{\rmod \w(\bar{X}, F_{0})}(j_{m}^{*} Y^{r}_{j_{m}(L)}, Y^{r}_{K}).
\end{equation}
We claim that the former homotopy limit is quasi-isomorphic to $\hom_{\w(X)}(L, K) = Y^{r}_{K}(L)$.

To prove the claim, observe that by adjunction we have
\begin{equation}\label{adjunction}
\begin{split}
& \hom_{\n_{>m}(X)}(j_{m}(L), j_{m}(K)) \stackrel{\sim}\to \hom_{\rmod \n_{>m}(X)}(Y^{r}_{j_{m}(L)}, Y^{r}_{j_{m}(K)})\\
 \cong & \hom_{\rmod \w(X)}(j_{m}^{*} Y^{r}_{j_{m}(L)}, Y^{r}_{K}) \cong \hom_{\perf \w(X)}(j_{m}^{*, \r{Prop}} Y^{r}_{j_{m}(L)}, K) 
\end{split} 
\end{equation}
By Proposition \ref{quasi-iso on nonessential Lagrangians} and Lemma \ref{n as homotopy limit of nm}, since $K$ is a nonessential Lagrangian, 
we have the following quasi-isomorphism given by the functor $J: \w(\bar{X}, F_{0}) \to \holim_{m \in \Z_{\le -1}^{op}} \n_{>m}(\bar{X}, F_{0})$:
\[
\hom_{\w(\bar{X}, F_{0})}(L, K) \stackrel{\sim}\to  \holim_{m \in \Z_{\le -1}^{op}} \hom_{\n_{>m}(\bar{X}, F_{0})}(j_{m}(L), j_{m}(K))
\]
Composing this with the homotopy limit of the maps \eqref{adjunction}, then \eqref{lim to colim} and \eqref{map from colimit of Floer complexes}, we get the Yoneda map $\hom_{\w(\bar{X}, F_{0})}(L, K) \to \hom_{\rmod \w(\bar{X}, F_{0})}(Y^{r}_{L}, Y^{r}_{K})$,
so \eqref{map from colimit of Floer complexes} is a quasi-isomorphism.
\end{proof}

In particular, since any object $K$ of $\w(\bar{X}, F_{0})$ that is in the image of the cup functor $i_{F}: \w(F) \to \w(\bar{X}, F_{0})$ \eqref{cup functor} is a nonessential Lagrangian,
Proposition \ref{prop: map from colimit of Floer complexes} holds true in this case.
Now we wish to use Lemma \ref{formal completion as limits} to prove that the homotopy limit $\holim_{m \in \Z_{\le -1}^{op}} \n_{>m}(\bar{X}, F_{0})$ is quasi-equivalent to the categorical formal completion of $\w(\bar{X}, F_{0})$ along $\tilde{\dd}(F)$.
The last step is to verify the other assumption of Lemma \ref{formal completion as limits}, 
which is the statement that all the pullback Yoneda modules $j_{m}^{*}Y^{r}_{L}$ actually come from $\tilde{\dd}(F)$.

\begin{thm}\label{representable of nm module by d}
Suppose $F$ satisfies stop removal for $X$, that $\p_{\infty} \bar{X}$ has bounded Reeb dynamics relative to $F$,
and that $\w(\bar{X}, F_{0})$ is smooth.
Then for any $L \in \ob \w(\bar{X}, F_{0})$, the pullback Yoneda module $j_{m}^{*}Y^{r}_{j_{m}(L)} \in \ob \rmod \w(\bar{X}, F_{0})$ is representable by an object in $\perf \tilde{\dd}(F)$.
That is, the perfect complex $j_{m}^{*, \r{Prop}}Y^{r}_{j_{m}(L)} \in \perf \w(\bar{X}, F_{0})$ is quasi-isomorphic to some object in $\perf \tilde{\dd}(F)$.
\end{thm}
\begin{proof}
For any Lagrangian $L$ of $X$, let $L \rightsquigarrow L^{w}$ be a positive isotopy in $\bar{X}$, 
such that $L^{w}$ is still disjoint from the stop at infinity. 
By definition, $L$ and $L^{w}$ are isomorphic in the fully wrapped Fukaya category $\w(\bar{X})$.
By general position, we may assume that $L^{w}$ is still in $X$.
When the entire isotopy is within $X$, $L^{w}$ is isomorphic to $L$ in $\w(X) = \w(\bar{X}, F_{0})$. 
Otherwise, since $F$ satisfies stop removal for $X$, the cone (taken in $\r{Tw}\w(\bar{X}, F_{0})$) of the continuation morphism $L^{w} \to L$ is in $\r{Tw} \tilde{\dd}(F)$; write $C = \cone(L^{w} \to L)$ for this cone.
For a testing object $K \in \ob \w(\bar{X}, F_{0})$, the module $j_{m}^{*}Y^{r}_{j_{m}(L)}$ takes values 
\begin{equation}
j_{m}^{*}Y^{r}_{j_{m}(L)}(K) = Y^{r}_{j_{m}(L)}(j_{m}(K)) = \hom_{\n_{>m}}(j_{m}(K), j_{m}(L)),
\end{equation}
where $j_{m}(K) = K$ geometrically but we want to remember that we take the module structure to be over $\n_{>m}(\bar{X}, F_{0})$.
Map the exact triangle in $\r{Tw} \w(\bar{X}, F_{0})$
\[
L^{w} \to L \to C \stackrel{[1]}\to
\] 
by $j_{m}$ to an exact triangle in $\r{Tw} \n_{>m}(\bar{X}, F_{0})$, 
and apply the resulting triangle to the testing object $j_{m}(K)$ to obtain a long exact sequence
\begin{equation}\label{testing les}
\begin{split}
\cdots \to &H^{*}(\hom_{\n_{>m}}(j_{m}(K), j_{m}(L^{w}))) \to H^{*}(\hom_{\n_{>m}}(j_{m}(K), j_{m}(L)))\\
 \to &H^{*}(\hom_{\r{Tw}\n_{>m}}(j_{m}(K), j_{m}(C))) \to H^{*+1}(\hom_{\n_{>m}}(j_{m}(K), j_{m}(L^{w})))
 \end{split}
\end{equation}
Consider $H^{*}(\hom_{\n_{>m}}(j_{m}(K), j_{m}(L^{w})))$. 
Applying the simultaneous negative isotopies to both objects, we have an isomorphism
\[
H^{*}(\hom_{\n_{>m}}(j_{m}(K), j_{m}(L^{w}))) \cong H^{*}(\hom_{\n_{>m}}(j_{m}(K)^{-w}, j_{m}(L))).
\]
The key observation is that, in the category $\n_{>m}(\bar{X}, F_{0})$, for sufficiently negative isotopy $j_{m}(K)^{-w}$ of $j_{m}(K)$,
the induced map
\[
H^{*}(\hom_{\n_{>m}}(j_{m}(K)^{-w}, j_{m}(L))) \to H^{*}(\hom_{\n_{>m}}(j_{m}(K), j_{m}(L)))
\]
given by multiplication with the continuation element in $HF^{*}(j_{m}(K), j_{m}(K)^{-w})$ is zero.
This can be checked using the definition of the morphism space in $\n_{>m}(X)$ given as \eqref{morphism in nm}.
It follows that the long exact sequence \eqref{testing les} breaks up to give an isomorphism
\[
H^{*}(\hom_{\n_{>m}}(j_{m}(K), j_{m}(L))) \cong H^{*}(\hom_{\r{Tw}\n_{>m}}(j_{m}(K), j_{m}(C))).
\]
Thus, the perfect complex $j_{m}^{*, \r{Prop}}Y^{r}_{j_{m}(L)}$ is quasi-isomorphic to some object of the form $j_{m}^{*, \r{Prop}} Y^{r}_{j_{m}(C)}$ for $C \in \ob \r{Tw} \tilde{\dd}(F)$.
\end{proof}

Proposition \ref{representable of nm module by d} together with Lemma \ref{formal completion as limits} (which is applicable here because all the categories $\n_{>m}(\bar{X}, F_{0})$  have the same collection of objects and the functors between them act as identity on objects) implies the following result.

\begin{cor}\label{cor: holim is formal completion}
Suppose $F$ satisfies stop removal and that $\p_{\infty} \bar{X}$ has bounded Reeb dynamics relative to $F$.
Also suppose $\w(\bar{X}, F_{0})$ is smooth.
Then there is a natural quasi-equivalence of $\ainf$-categories
\begin{equation}\label{quasi-equivalence to wf from limit}
\hat{\pi}^{*}: \mathcal{S}(\{\n_{>m}(\bar{X}, F_{0})\}_{m \in \Z_{\le -1}}) \stackrel{\sim}\to \wf,
\end{equation}
where $\mathcal{S}(\{\n_{>m}(\bar{X}, F_{0})\}_{m \in \Z_{\le -1}})$ is the subcategory of $\holim_{m \in \Z_{\le -1}^{op}} \n_{>m}(\bar{X}, F_{0})$ defined in Definition \ref{def: subcat of constant sections}.
\end{cor}
\begin{proof}
While the existence follows from the general theory, i.e. Lemma \ref{formal completion as limits}, we can give an explicit description of the functor $\hat{\pi}^{*}$ as follows.
On the level of objects, $\hat{\pi}^{*}$ sends a constant section $s=s^{K}$ associated to some object $K \in \ob \n_{>m}(\bar{X}, F_{0})$ to the object $K \in \ob \wf$ (recall $\wf$ has the same collection of objects as $\w(\bar{X}, F_{0})$, and therefore also as $\n(\bar{X}, F_{0})$ and $\n_{>m}(\bar{X}, F_{0})$).
Let $s, s'$ be two constant sections associated to Lagrangians $K, L$. 
On the level of morphisms, $\hat{\pi}^{*}$ sends a natural transformation $T$ from $s$ to $s'$ to a pre-$Z_{\tilde{\dd}(F)}^{op}$-module morphism from $P_{\tilde{\dd}(F)}(-, K)$ to $P_{\tilde{\dd}(F)}(-, L)$ induced by $\ainf$-multiplication with $T(m)$ for every $m$.
This is well-defined because 
\begin{enumerate}[label=(\roman*)]

\item for every $m$, $T(m) \in \hom_{\n_{>m}}(s(m), s'(m)) = \hom_{\n_{>m}}(K, L)$ is a representative Floer cochain in the quotient complex $C_{0}/C_{m}$;

\item and by Proposition \ref{prop: morphism in t = in zt} the right $Z_{\tilde{\dd}(F)}^{op}$-module $P_{\tilde{\dd}(F)}(-, K)$ has the same underlying cochain complex as the pullback Yoneda module $i^{*}Y^{r}_{K}$ as a module over $\tilde{\dd}(F)$.

\end{enumerate}
Thus we can define a pre-module morphism from $P_{\tilde{\dd}(F)}(-, K)$ to $P_{\tilde{\dd}(F)}(-, L)$ by $\ainf$-multiplication with a cochain representative of $T(m)$,
which clearly induces a pre-module morphism on the Yoneda modules.
A priori, this depends on choices of representatives and also on $m$.
However, this induces a well-defined map from the homotopy limit by our construction of the diagram $\{\n_{>m}(\bar{X}, F_{0})\}_{m \in \Z_{\le -1}}$, 
particularly because the functors $\pi_{m, m'}$ have vanishing higher order terms
(see \S\ref{sec: variants} for the notations of these Floer complexes and Appendix \ref{app: map from inverse limit} for the definitions of sections and natural transformations appearing in the homotopy limit).

It is a consequence of Yoneda lemma that the functor $\hat{\pi}^{*}$ defined above is fully faithful.
\end{proof}

We are ready to prove Theorem \ref{thm:main2}.

\begin{proof}[Proof of Theorem \ref{thm:main2}]
By Lemma \ref{lem: image of pi}, the fully-faithful functor 
\[
\Pi: \n(\bar{X}, F_{0}) \stackrel{\sim}\to \holim_{m \in \Z_{\le -1}^{op}} \n_{>m}(\bar{X}, F_{0})
\]
supplied by Proposition \ref{n as homotopy limit of nm} has image in the subcategory $\mathcal{S}(\{\n_{>m}(\bar{X}, F_{0})\}_{m \in \Z_{\le -1}})$ of the homotopy limit $\holim_{m \in \Z_{\le -1}^{op}} \n_{>m}(\bar{X}, F_{0})$.
Thus we may compose it with the quasi-equivalence $\hat{\pi}^{*}$ \eqref{quasi-equivalence to wf from limit} in Corollary \ref{cor: holim is formal completion} to get a quasi-equivalence
\begin{equation}\label{alternative equivalence}
\tilde{\Psi}:= \hat{\pi}^{*} \circ \Pi: \n(\bar{X}, F_{0}) \stackrel{\sim}\to \wf.
\end{equation}
This concludes the proof.
\end{proof}

Given the naturality of its construction, this quasi-equivalence is expected to be homotopic to the functor $\Psi$ \eqref{functor to completion},
which would thus improve the statement of Theorem \ref{thm:main2} to a conclusion similar to that of Theorem \ref{thm:equivalence},
and allow us to conclude that under the hypotheses of Theorem \ref{thm:main2},
the Yoneda-type functor $\Psi$ is still a quasi-equivalence.
Proving this, however, requires further analysis of the quasi-equivalence $\hat{\pi}^{*}$, 
which would be purely algebraic and of independent interest.

\begin{conj}
The quasi-equivalence $\tilde{\Psi}$ \eqref{alternative equivalence} is homotopic to the functor $\Psi$ \eqref{functor to completion}.
\end{conj}

%

\appendix 

\section{Homotopy limits}\label{app: homotopy limits}

The notion of homotopy limit is dual to homotopy colimit, where an explicit construction for the homotopy colimit of a diagram of $\ainf$-categories is given in \cite{GPS2} using the Grothendieck construction of the diagram, adapted from \cite{Lurie}.
Although we work with $\ainf$-categories whose properties are often referred to as `up to homotopy', 
in this construction we assume that diagrams are indexed by posets $I$, 
and require that all commuting diagrams be strictly commuting.
For the homotopy limit (the lax limit), we can give an explicit model, via the category of sections of the Grothendieck construction.
For simplicity, all $\ainf$-categories in this section are assumed to be strictly unital.
The reader should be warned that this Appendix consist of facts well-known to experts in category theory;
but we decide to spell them out here for the purpose of fixing conventions in this paper that exclusively fit into the framework of $\ainf$-categories.

\subsection{Grothendieck construction}\label{section: Grothendieck construction}

Let $I$ be a poset, so that for every $i, j \in I$ with $i \le j$, there is a unique morphism $i \to j$, with the morphism $i \to i$ being the identity morphism.
By a diagram $\{\cc_{i}\}_{i \in I}$ of $\ainf$-categories indexed by a poset $I$, we mean a functor $I \to \ainf-\r{Cat}$ which is strictly commuting, i.e. with identity natural transformation for compositions,
and satisfies $\cc_{i} \to \cc_{i}$ is the identity functor.
Concretely, this consists the following data:
\begin{itemize}

\item for each $i \in I$, an $\ainf$-category $\cc_{i}$,

\item for each morphism $i \to j$ in $I$ (i.e. $i \le j$ with a unique morphism), an $\ainf$-functor $f_{ij}: \cc_{i} \to \cc_{j}$,

\end{itemize}
such that for every chain of morphisms $i \to j \to k$ in $I$, we have $f_{jk} \circ f_{ij} = f_{ik}$.

First let us recall the Grothendieck construction $\r{Gr}(\{\cc_{i}\}_{i \in I})$ of a diagram $\{\cc_{i}\}_{i \in I}$, following \cite{Lurie}, \cite{GPS2}.
Given a $(\cc, \dd)$-bimodule $\B$, the {\it semi-orthogonal gluing} $\langle \dd, \cc \rangle_{\B}$ of $\dd$ and $\cc$ along $\B$ is the $\ainf$-category whose objects are $\ob \cc \sqcup \ob \dd$,
and morphisms are
\begin{equation}
\langle \dd, \cc \rangle_{\B}(X, Y) = 
\begin{cases}
\dd(X, Y), & X, Y \in \ob \dd, \\
\cc(X, Y), & X, Y \in \ob \cc, \\
\B(Y, X), & X \in \ob \dd, Y \in \ob \cc,\\
0, & X \in \ob \cc, Y \in \ob \dd.
\end{cases}
\end{equation}
The Grothendieck construction of a finite directed system $\cc_{0} \to \cdots \to \cc_{k}$ is defined inductively by $\r{Gr}(\cc) = \cc$, and
\begin{equation}
\r{Gr}(\cc_{0} \stackrel{f_{01}}\to \cdots \stackrel{f_{k-1,k}}\to \cc_{k}) = \r{Gr}(\cc_{0} \stackrel{f_{01}}\to \cdots \stackrel{f_{k-2,k-1}}\to \langle \cc_{k-1}, \cc_{k} \rangle_{ (\id_{\cc_{k}}, f_{k-1,k})^{*} (\cc_{k})_{\D} }).
\end{equation}

For any diagram $\{\cc_{i}\}_{i \in I}$ indexed by a poset $I$, 
define the Grothendieck construction $\r{Gr}(\{\cc_{i}\}_{i \in I})$ to be the $\ainf$-category with objects $\sqcup_{i \in I} \ob \cc_{i}$, with zero morphisms from objects in $\cc_{i}$ to those in $\cc_{j}$ unless there is a morphism $i \to j$ in $I$ (i.e. $i \le j$),
such that for every chain of composable morphisms $i_{0} \to \cdots \to i_{p}$ in $I$, 
the full subcategory with objects $\sqcup_{j=0}^{p} \ob \cc_{i_{j}}$ is given by $\r{Gr}(\cc_{i_{0}} \to \cdots \to \cc_{i_{p}})$.

Define the class $A_{I}$ of adjacent morphisms in $H^{0}(\r{Gr}(\{\cc_{i}\}_{i \in I}))$ to be those morphisms $X_{i} \to f_{ij}X_{i}$ which correspond to the unit in $H^{0}\cc_{j}(f_{ij}X_{i}, f_{ij}X_{i}) = H^{0}(\r{Gr}(\{\cc_{i}\}_{i \in I})(X_{i}, f_{ij}X_{i})$ for $i \le j$.

\subsection{The category of sections}

Next, we define a new $\ainf$-category $\Gamma\r{Gr}(\{\cc_{i}\}_{i \in I})$.
The objects of $\Gamma\r{Gr}(\{\cc_{i}\}_{i \in I})$ are called {\it sections} of the Grothendieck construction $\r{Gr}(\{\cc_{i}\}_{i \in I})$,
which are functors $I \to \r{Gr}(\{\cc_{i}\}_{i \in I})$ (not $\ainf$-functors)
for which the composition $I \to \r{Gr}(\{\cc_{i}\}_{i \in I}) \to I$ is the identity. 
Concretely, an object $s$ of $\Gamma\r{Gr}(\{\cc_{i}\}_{i \in I})$ consists of the following data:
\begin{itemize}

\item for every object $i \in I$, an object $s(i) \in \cc_{i}$,

\item for each morphism $i \to j$ in $I$ (unique for every $i \le j$), a morphism $s_{ij} \in \cc_{j}(f_{ij}(s(i)), s(j))$,

\end{itemize}
which are required to satisfy the following compatibility conditions
\begin{enumerate}[label=(\roman*)]

\item $s_{ii} = 1_{s(i)} \in \cc_{i}(s(i), s(i))$ (this makes sense as we assume $f_{ii} = \id_{\cc_{i}}$ for all the diagrams we consider here),

\item for morphism $i \to j \to k$ (which exist when $i \le j \le k$), the associated morphisms $s_{ij}, s_{jk}, s_{ik}$ should satisfy
\begin{equation}\label{section composition}
\mu^{2}_{\cc_{k}}(s_{jk}, f_{jk}(s_{ij})) = s_{ik}.
\end{equation}

\end{enumerate}
Notice that the second condition is quite restrictive, but it is reasonable in our setting as we will only consider strictly commuting diagrams.
A morphism $T \in \Gamma\r{Gr}(\{\cc_{i}\}_{i \in I})(s, s')$ is a natural transformation from $s$ to $s'$ lying above the identity of $I$, which concretely consists of a morphism $T(i) \in \cc_{i}^{g}(s(i), s'(i))$ for every $i \in I$,
satisfying compatibility conditions for natural transformations between functors $s, s': I \to \r{Gr}(\{\cc_{i}\}_{i \in I})$, i.e.
\begin{equation}\label{nattrans compatible}
\mu^{2}_{\cc_{j}}(s'_{ij}, f_{ij}(T(i))) = \mu^{2}_{\cc_{j}}(T(j), s_{ij}),
\end{equation}
for every $i \to j$.
The $\ainf$-structure on $\Gamma\r{Gr}(\{\cc_{i}\}_{i \in I})$ is induced by $\ainf$-structures on $\cc_{i}$ for all $i$.
That is, for $T^{(j)} \in \Gamma\r{Gr}(\{\cc_{i}\}_{i \in I})(s^{(j-1)}, s^{(j)}), j = 1, \ldots, k$, we let
\begin{equation}\label{muk of nattrans}
T: = \mu^{k}_{\Gamma\r{Gr}(\{\cc_{i}\}_{i \in I})}(T^{(k)}, \ldots, T^{(1)})
\end{equation}
be the natural transformation $T$ such that 
\begin{equation}\label{muk of nattrans component}
T(i) = \mu^{k}_{\cc_{i}}(T^{(k)}(i), \ldots, T^{(1)}(i)).
\end{equation}
This is well-defined because of \eqref{section composition} and \eqref{nattrans compatible}.
The $\ainf$-equations of $\Gamma\r{Gr}(\{\cc_{i}\}_{i \in I})$ immediately follow from $\ainf$-equations of $\cc_{i}$ and the strictly commuting condition of the diagram $\{\cc_{i}\}_{i \in I}$.

\subsection{The homotopy limit}

To define the homotopy limit, we first need to define the class of adjacent morphisms $B_{I} \subset H^{0}(\Gamma\r{Gr}(\{\cc_{i}\}_{i \in I}))$.
These are induced by adjacent morphisms in $A_{I} \subset H^{0}(\r{Gr}(\{\cc_{i}\}_{i \in I}))$. 
That is to say, morphisms in $B_{I}$ are defined to be those whose components are in $A_{I}$. 

\begin{defn}
We define the homotopy limit 
\begin{equation}
\holim_{I} \{\cc_{i}\}_{i \in I} = \Gamma\r{Gr}(\{\cc_{i}\}_{i \in I})[B_{I}^{-1}]
\end{equation}
 to be the localization of $\Gamma\r{Gr}(\{\cc_{i}\}_{i \in I})$ at the collection $B_{I}$ of adjacent morphisms in $H^{0}(\Gamma\r{Gr}(\{\cc_{i}\}_{i \in I}))$.
 \end{defn}
 
 \begin{rem}
 At the beginning of \S\ref{section: Grothendieck construction}, we assume the diagram is indexed by a poset $I$,
 in which there is a unique morphism $i \to j$ iff $i \le j$.
 All the above constructions also make sense for the opposite poset $I^{op}$, where there is a unique morphism $i \to j$ iff $j \le i$.
 \end{rem}
 
 By construction, the homotopy limit comes with natural projections 
 \begin{equation}
 \pi_{i}: \holim_{i \in I}\cc_{i} \to \cc_{i}
 \end{equation}
  for every $i$.
 This is induced from a functor $\Gamma\r{Gr}(\{\cc_{i}\}_{i \in I}) \to \cc_{i}$,
 which on objects sends a section $s$ to its evaluation $s(i) \in \ob \cc_{i}$,
 and on morphisms sends a natural transformation $T$ to the morphism $T(i)$.

If $I$ has a maximal element $\bar{i}$, then the homotopy colimit $\hocolim_{i \in I} \cc_{i}$ is quasi-equivalent to $\cc_{\bar{i}}$.
As for the homotopy limit, we have
 
 \begin{lem}
 If $I$ has a minimal element $\underline{i}$, then the projection functor $\holim_{i \in I} \cc_{i} \to \cc_{\underline{i}}$ is a quasi-equivalence.
 \end{lem}
 \begin{proof}
It is not hard to construct an inverse to the projection, which is induced by a functor 
 \[
\cc_{\underline{i}} \to \Gamma\r{Gr}(\{\cc_{i}\}_{i \in I}) 
 \]
 sending every $X_{\underline{i}} \in \ob \cc_{\underline{i}}$ to the section $s$ defined by
 \begin{equation}
 s(i) = f_{\underline{i}, i} (X_{\underline{i}})
 \end{equation}
 and for $i \to j$, the morphism $s_{ij}$
is uniquely inductively defined by the insertion of $\underline{i} \to i \to j$ and using \eqref{section composition},
with $s_{\underline{i}, i} = 1_{f_{\underline{i}, i}(X_{\underline{i}})} \in \hom_{\cc_{i}}(f_{\underline{i}, i}(X_{\underline{i}}), f_{\underline{i}, i}(X_{\underline{i}}))$.
By definition, such an $s$ is a well-defined functor $I \to \r{Gr}(\{\cc_{i}\}_{i \in I})$ whose composition with the natural functor $\r{Gr}(\{\cc_{i}\}_{i \in I}) \to I$ is the identity. 

The maps on morphism spaces
\begin{equation}
\cc_{\underline{i}}(X_{\underline{i}, k-1}, X_{\underline{i}, k}) \otimes \cdots \otimes \cc_{\underline{i}}(X_{\underline{i}, 0}, X_{\underline{i}, 1}) \to \hom_{\Gamma\r{Gr}(\{\cc_{i}\}_{i \in I})}(s_{0}, s_{k})
\end{equation}
are induced by $\ainf$-multiplications with the elements from $\cc_{\underline{i}}(X_{\underline{i}, k-1}, X_{\underline{i}, k}) \otimes \cdots \otimes \cc_{\underline{i}}(X_{\underline{i}, 0}, X_{\underline{i}, 1})$ under the images of $f_{\underline{i}, i_{k}}$.
In particular, the first order map induces an isomorphism on cohomology after passing to the quotient by adjacent morphisms.
And it is a direct computation to check the induced map on cohomology is inverse to the projection.
 \end{proof}
 
 Let $\{\cc_{i}\}_{i \in I}, \{\dd_{i}\}_{i \in I}$ be two diagrams indexed by the same poset $I$.
 Suppose we have a diagram of strictly compatible $\ainf$-functors $\{F_{i}: \cc_{i} \to \dd_{i}\}_{i \in I}$.
 Then, the diagram of functors induces a functor on the homotopy limit
 \begin{equation}
 F = \holim_{i \in I} F_{i}: \holim_{i \in I} \cc_{i} \to \holim_{i \in I} \dd_{i},
 \end{equation}
 which is induced by the functor on the section categories $\Gamma\r{Gr}(\{\cc_{i}\}_{i \in I}) \to \Gamma\r{Gr}(\{\dd_{i}\}_{i \in I})$ sending each $s$ to the section $F \circ s$ defined on the level of objects by $F \circ s(i) = F_{i}(s(i))$.
In particular, the usual {\it universal property} is a special consequence of the above construction:

\begin{lem}\label{lem: induced functors on limit}
If $\{F_{i}: \cc \to \dd_{i}\}_{i \in I}$ is a diagram of strictly compatible $\ainf$-functors, 
then there is an induced functor
\begin{equation}
F = \holim_{i \in I} F_{i}: \cc \to \holim_{i \in I} \dd_{i}.
\end{equation}
Moreover, if any other functor $G: \cc \to \holim_{i \in I} \dd_{i}$ satisfies $\pi_{i} \circ G = F_{i}$, then $G$ is homotopic to $F$.
\end{lem}

However, unlike the homotopy colimit, the homotopy limit is not functorial with respect to a diagram of functors covering an arbitrary functor between posets.
Nonetheless, we have the following special case where the diagram of functors is exactly induced by pullback of a functor of diagrams:

\begin{lem}
Let $\{\dd_{j}\}_{j \in J}$ be a diagram of $\ainf$-categories, and $\alpha: I \to J$ a functor of posets.
Let $\{\cc_{i}\}_{i \in I}$ be the pullback diagram defined by $\cc_{i} = \dd_{\alpha_{i}}$. 
Then there is an induced functor
\begin{equation}
\alpha^{*}: \holim_{j \in J} \dd_{j} \to \holim_{i \in I} \cc_{i}.
\end{equation}
\end{lem}
\begin{proof}
On the section categories $\Gamma\r{Gr}(\cdot)$, clearly there is a functor by pullback by $\alpha$.
It is straightforward to check that the adjacent morphisms in $H^{0}\Gamma\r{Gr}(\{\cc_{i}\}_{i \in I})$ are exactly the pullbacks of adjacent morphisms in $H^{0}\Gamma\r{Gr}(\{\dd_{j}\}_{j \in J})$ under this pullback.
\end{proof}

\subsection{The natural transformation from the limit}\label{app: map from inverse limit}

Since the diagram $\{\cc_{i}\}_{i \in I}$ is strictly commuting and all the $\ainf$-categories $\cc_{i}$ are small, we can also consider the ordinary limit
\begin{equation}
\lim_{i \in I} \cc_{i}
\end{equation}
which is an $\ainf$-category with set of objects being the following subset of $\prod_{I} \ob \cc_{i}$
\begin{equation}
\ob \lim_{i \in I} \cc_{i} = \{(X_{i})_{i \in I}: X_{i} \in \ob \cc_{i}, f_{ij}(X_{i}) = X_{j} \text{ if there is a morphism } i \to j \}.
\end{equation}
The morphism space $\lim_{i \in I} \cc_{i}((X_{i})_{i \in I}, (Y_{i})_{i \in I})$ is the following subspace of $\prod_{I} \cc_{i}(X_{i}, Y_{i})$
\begin{equation}
\lim_{i \in I} \cc_{i}((X_{i})_{i \in I}, (Y_{i})_{i \in I}) = \{ (c_{i})_{i \in I}: c_{i} \in \cc_{i}(X_{i}, Y_{i}), f_{ij}(c_{i}) = c_{j} \}.
\end{equation}
The $\ainf$-structure maps are defined component-wise, namely
\begin{equation}
\mu^{k}_{\lim_{i \in I} \cc_{i}}((c^{(k)}_{i})_{i \in I}, \ldots, (c^{(1)}_{i})_{i \in I}) = (\mu^{k}_{\cc_{i}}(c^{(k)}_{i}, \ldots, c^{(1)}_{i})_{i \in I}.
\end{equation}
All of these are well-defined exactly because the diagram is strictly commuting.

\begin{prop}\label{prop: lim to holim}
There exists a natural transformation from $\lim$ to $\holim$.
That is, for every such diagram there is a functor 
\begin{equation}
H: \lim_{i \in I} \cc_{i} \to \holim_{i \in I} \cc_{i},
\end{equation}
which is natural with respect to functors of diagrams.
\end{prop}
\begin{proof}
We will give the construction for a fixed diagram $\{\cc_{i}\}_{i \in I}$, and naturality will easily follow from the construction.
On the level of objects, the functor $H$ sends $(X_{i})_{i \in I} \in \ob \lim_{i \in I} \cc_{i}$ to a section $s$ such that $s(i) = X_{i}$ on objects,
and for any morphism $i \to j$, $s_{ij} \in \cc_{j}(f_{ij}(X_{i}), X_{j}) = \cc_{j}(X_{j}, X_{j})$ is given by the unit of $X_{j}$ in $\cc_{j}$.
On the level of morphisms, $H$ sends $(c_{i})_{i \in I} \in \lim_{i \in I} \cc_{i}((X_{i})_{i \in I}, (Y_{i})_{i \in I})$ to the natural transformation $T$ defined by $T(i) = c_{i}$.
With this definition, it is a straightforward computation to check the compatibility conditions \eqref{section composition} and \eqref{nattrans compatible}.
Define higher order terms of $H$ to be zero, $H^{k} = 0$ for $k \ge 2$.

It remains to check that $H$ defined above is an $\ainf$-functor.
Since all the higher order terms are zero, the only $\ainf$-equation to verify is
\begin{equation}
\mu^{k}_{\holim_{i \in I} \cc_{i}}(H((c^{(k)}_{i})_{i \in I}), \ldots, H((c^{(1)}_{i})_{i \in I})) = H(\mu^{k}_{\lim_{i \in I} \cc_{i}}((c^{(k)}_{i})_{i \in I}, \ldots, (c^{(1)}_{i})_{i \in I})).
\end{equation}
Let $T^{(j)} = H((c^{(j)}_{i})_{i \in I}), j = 1, \ldots, k$, which is the natural transformation such that $T^{(j)}(i) = c^{(j)}_{i}$.
By the definition of the $\ainf$-structure of $\holim_{i \in I} \cc_{i}$, $\mu^{k}_{\holim_{i \in I} \cc_{i}}(T^{(k)}, \ldots, T^{(1)})$ is induced by $\ainf$-structures of $\cc_{i}$, by \eqref{muk of nattrans} and \eqref{muk of nattrans component}.
Thus we have
\begin{equation}
\mu^{k}_{\holim_{i \in I} \cc_{i}}(T^{(k)}, \ldots, T^{(1)})(i) = \mu^{k}_{\cc_{i}}(c^{(k)}_{i}, \ldots, c^{(1)}_{i})
\end{equation}
for every $i$, 
verifying the $\ainf$-functor equation.
\end{proof}

A special case of interest to us is where all the $\cc_{i}$'s have the same collection of objects such that $f_{i, i+1}$ acts as identity on objects.

\begin{defn}\label{def: constant section}
In the above situation, every object $X \in \ob \cc_{i}$ determines a unique object $(X_{i})_{i \in I} \in \ob \lim_{i \in I} \cc_{i}$ where $X_{i} = X$ for all $i$, which we call the {\it constant sequence} associated to $X$.

Under the functor $H: \lim_{i \in I} \cc_{i} \to \holim_{i \in I} \cc_{i}$, the corresponding section in $\holim_{i \in I} \cc_{i}$, denoted by $s^{X}$, is called the {\it constant section} associated to $X$.
\end{defn}

Even more restrictively, suppose $\{\cc_{i}\}_{i \in I}$ is an inverse system (where $I = \Z_{\le -1}$).
The usual limit, a.k.a. the inverse limit in this case, is the $\ainf$-category whose objects are constant sequences.
And the morphism space is by definition the inverse limit of cochain complexes
\begin{equation}
\lim_{i \in I} \cc_{i}((X)_{i \in I}, (Y)_{i \in I}) = \varprojlim_{i} \cc_{i}(X, Y).
\end{equation}
Thus we get for each pair of objects $X, Y \in \ob \cc_{i}$, a cochain map
\begin{equation}\label{chain map from lim to holim}
H: \varprojlim_{i} \cc_{i}(X, Y) \to \holim_{i \in I} \cc_{i}(s^{X}, s^{Y}) = \holim_{i \in I} \{\cc_{i}(X, Y)\}_{i \in I},
\end{equation}
where the latter cochain complex is precisely the homotopy limit of the inverse system of cochain complexes $\{\cc_{i}(X, Y)\}_{i \in I}$.
A natural question is whether this map is a quasi-isomorphism.
The answer depends on a condition on the inverse system, called the degree-wise {\it Mittag-Leffler} condition.

\begin{defn}\label{def: Mittag-Leffler}
An inverse system of cochain complexes $\{C_{i}\}_{i \in I}$ with maps $f_{ij}: C_{i} \to C_{j}, i \le j$ is said to satisfy the degree-wise Mittag-Leffler condition, 
if for every $i$ there exists $i' \le i$ such that for all $i'' \le i'$, $f_{i'', i}(C_{i''}) = f_{i', i}(C_{i'})$ in each degree.
\end{defn}

The following fact in homological algebra provides an answer to the question concerning the limits.
A standard reference is \cite{weibel}.

\begin{lem}\label{lem: ml implies lim = holim}
Suppose an inverse system of cochain complexes $\{C_{i}\}_{i \in I}$ satisfies the degree-wise Mittag-Leffler condition.
Then the natural map $H: \varprojlim_{i} C_{i} \to \holim_{i} C_{i}$ is a quasi-isomorphism.
\end{lem} \qed

\bibliography{completion}
 
\end{document}